\documentclass[reqno,11pt]{amsart}

\usepackage{amsmath}
\usepackage{amsthm}
\usepackage{amssymb}
\usepackage{amscd}
\usepackage{xypic}
\usepackage{verbatim}
\usepackage{mathabx}
\usepackage{graphicx}
\usepackage{palatino}
\usepackage{bbm}
\usepackage{extarrows}

\usepackage{xr}
\usepackage[plainpages=false,colorlinks,hyperindex,pdfpagemode=None,bookmarksopen,linkcolor=red,citecolor=blue,urlcolor=blue]{hyperref}
\usepackage{pdflscape}
\usepackage{stmaryrd}

\usepackage{multirow}

\usepackage{comment}

\swapnumbers

\theoremstyle{plain}
\newtheorem{theorem}[subsubsection]{Theorem}

\newtheorem*{theorem*}{Theorem}
\newtheorem*{mainconclusions*}{Main Conclusions}
\newtheorem{proposition}[subsubsection]{Proposition}
\newtheorem*{proposition*}{Proposition}

\newtheorem*{lemma*}{Lemma}

\newtheorem{corollary}[subsubsection]{Corollary}
\newtheorem*{corollary*}{Corollary}

\theoremstyle{definition}

\theoremstyle{remark}
\newtheorem{remark}[subsubsection]{Remark}
\newtheorem{remarks}[subsubsection]{Remarks}

\newtheorem{example}[subsubsection]{Example}

\DeclareFontFamily{OT1}{rsfs}{}
\DeclareFontShape{OT1}{rsfs}{n}{it}{<-> rsfs10}{}
\DeclareMathAlphabet{\mathscr}{OT1}{rsfs}{n}{it}

\newcommand{\from}{\leftarrow}

\newcommand{\sgn}{\mathrm{sgn}}

\newcommand{\Ad}{\mathrm{Ad}}
\newcommand{\Res}{\mathrm{Res}}

\newcommand{\Z}{\mathbb{Z}}
\newcommand{\C}{\mathfrak{C}}

\newcommand{\ad}{{\operatorname{ad}}}

\newcommand{\0}{{(0)}}

\newcommand{\CC}{\mathbb{C}}

\renewcommand{\P}{\mathbb{P}}
\newcommand{\RR}{\mathbb{R}}
\newcommand{\Rplus}{{\RR^\times_+}}

\newcommand{\Hom}{\operatorname{Hom}}
\newcommand{\End}{\operatorname{End}}

\newcommand{\Gm}{\mathbb{G}_m}
\newcommand{\Ga}{\mathbb{G}_a}

\newcommand{\GL}{\operatorname{GL}}
\newcommand{\Mat}{\operatorname{Mat}}
\newcommand{\Sym}{\operatorname{Sym}}

\newcommand{\PGL}{\operatorname{PGL}}

\newcommand{\SL}{\operatorname{SL}}

\newcommand{\SO}{{\operatorname{SO}}}

\newcommand{\tr}{\operatorname{tr}}

\newcommand{\spec}{\operatorname{spec}}
\newcommand{\Vol}{\operatorname{Vol}}
\newcommand{\diag}{{\operatorname{diag}}}
\newcommand{\adiag}{{\operatorname{adiag}}}

\newcommand{\Id}{\operatorname{Id}}

\newcommand{\Std}{{\operatorname{Std}}}
\newcommand{\st}{{\operatorname{st}}}

\newcommand{\val}{{\operatorname{val}}}

\newcommand{\temp}{{\operatorname{temp}}}

\newcommand{\RTF}{{\operatorname{RTF}}}

\newcommand{\AvgVol}{{\operatorname{AvgVol}}}
\newcommand{\PW}{{\operatorname{PW}}}

\newcommand{\Meas}{{\operatorname{Meas}}}

\newcommand{\cl}{{\operatorname{cl}}}

\newcommand{\LX}{{^LX}}

\newcommand{\LG}{{^LG}}

\newcommand{\Dfrac}[2]{ 
  \ooalign{ 
    $\genfrac{}{}{2.0pt}0{\phantom{#1}}{\phantom{#2}}$\cr 
    $\color{white}\genfrac{}{}{1.2pt}0{\normalcolor{#1}}{\normalcolor{#2}}$} 
}

\begin{document}

\numberwithin{equation}{section}
\setcounter{tocdepth}{2}
\title[Transfer operators and Hankel transforms, I]{Transfer operators and Hankel transforms between relative trace formulas, I: Character theory}
\author{Yiannis Sakellaridis}
\email{sakellar@rutgers.edu}

\address{Department of Mathematics and Computer Science, Rutgers University at Newark, 101 Warren Street, Smith Hall 216, Newark, NJ 07102, USA.}

\address{School of Mathematics, Institute For Advanced Study, 1 Einstein Drive, Princeton, NJ 08540, USA.}



\begin{abstract}
The Langlands functoriality conjecture, as reformulated in the ``beyond endoscopy'' program, predicts comparisons between the (stable) trace formulas of different groups $G_1, G_2$ for every morphism $\LG_1\to \LG_2$ between their $L$-groups. This conjecture can be seen as a special case of a more general conjecture, which replaces reductive groups by spherical varieties and the trace formula by its generalization, the relative trace formula. 

The goal of this article and its continuation \cite{SaTransfer2} is to demonstrate, by example, the existence of ``transfer operators'' betweeen relative trace formulas, that generalize the scalar transfer factors of endoscopy. These transfer operators have all properties that one could expect from a trace formula comparison: matching, fundamental lemma for the Hecke algebra, transfer of (relative) characters. Most importantly, and quite surprisingly, they appear to be of abelian nature (at least, in the low-rank examples considered in this paper), even though they encompass functoriality relations of non-abelian harmonic analysis. Thus, they are amenable to application of the Poisson summation formula in order to perform the global comparison. Moreover, we show that these abelian transforms have some structure --- which presently escapes our understanding in its entirety --- as deformations of well-understood operators when the spaces under consideration are replaced by their ``asymptotic cones''. 

In this first paper we study (relative) characters for the Kunzetsov formula and the stable trace formula for $\SL_2$ and their degenerations (as well as for the relative trace formula for torus periods in $\PGL_2$), and we show how they correspond to each other under explicit transfer operators. \end{abstract}

\maketitle

\tableofcontents

\section{Introduction}

\subsection{Functoriality}

The goal of this paper and its continuation,\footnote{Any references to sections or equations numbered 6 or above refer to \cite{SaTransfer2}.} \cite{SaTransfer2}, is to demonstrate a local theory ``beyond endoscopy'', which indicates that functorial transfer  between trace formulas is governed by ``transfer operators'' that have a certain structure: First, in the low-rank cases that we are examining, these operators can be explicitly described in terms of Fourier transforms and other relatively innocuous operators (such as multiplication by an automorphic character) which, \emph{in principle}, are amenable to application of the Poisson summation formula and, hence, to a global comparison. Secondly, these transfer operators are deformations of completely understood transfer operators that one obtains when letting the spaces degenerate to their \emph{asymptotic cones} (or \emph{boundary degenerations}). In the process, we develop the local theory behind ``non-standard'' comparisons of trace formulas that have appeared in the literature, namely in Rudnick's and Venkatesh's theses \cite{Rudnick, Venkatesh}, revealing a structure that is not evident in the analytic number theory approach. Our results should also be related to Herman's trace-formula-theoretic proof of the functional equation of the standard $L$-function of $\GL_2$ \cite{Herman} (which should correspond to the local Godement--Jacquet theory on the Kuznetsov formula, developed by Jacquet in \cite{Jacquet}), and Altug's theory of Poisson summation for the stable trace formula of $\SL_2$ \cite{Altug1,Altug2,Altug3}. Most of the results of this paper were announced, without proofs, in \cite{SaHanoi}.

``Functoriality'', here, is understood in the generalized sense, in the tradition of Jacquet, that includes \emph{spherical varieties} in addition to reductive groups (which are a special case). Thus, the trace formulas compared are \emph{relative trace formulas}, which include the trace formula in the group case. With the exception of Venkatesh's thesis, that we revisit, the functorial lifts that underlie our comparisons correspond to an \emph{isomorphism of $L$-groups}. To the reader eager to see non-trivial cases of functoriality between reductive groups, this might seem, and is, a disappointment. However, the ``relative'' generalization demonstrates that the problem of functoriality is highly non-trivial already for the identity maps of $L$-groups. What hopes do we have to tackle the general problem if we don't understand this basic case? Moreover, the examples collected here demonstrate the central role played by the \emph{Kuznetsov formula} in any successful ``beyond endoscopy'' comparison that I know. This fact is reinforced by the upcoming paper \cite{SaRankone}, which will compare any relative trace formula of rank one to the Kuznetsov formula. Although it is too early to pass verdict on this, it may just be that the idea of directly comparing stable Arthur--Selberg trace formulas is infeasible, and one has to move to the ``relative'' setting of the Kuznetsov formula in order to prove the functoriality conjecture; this idea appeared early on in the short history of ``beyond endoscopy'', in a letter of Sarnak to Langlands \cite{Sarnak}.

Our methods of proof are quite classical, and rely heavily on Rankin--Selberg theory. Not surprisingly, in retrospect, several different methods to study the same problem converge to the same, when viewed as a comparison between the appropriate trace formulas. The hope is that this new trace formula-theoretic approach will generalize to cases where Rankin--Selberg theory and other techniques do not. Again, in rank one this will be confirmed in the upcoming paper \cite{SaRankone}.

\subsection{Overview}

The relative trace formula is a distribution (with a geometric and a spectral expansion) on the adelic points of a stack. In this paper we work over a local field $F$, and only with relative trace formulas attached to quotients of the form $[X\times X/ G^\diag]$, where $X$ is a homogeneous spherical $G$-variety of low rank. We will allow the case $X= (N,\psi)\backslash G$, that is, the Whittaker case of the variety $N\backslash G$, equipped with a generic character of the maximal unipotent subgroup $N$, except that in that case we will take the character to be $\psi^{-1}$ on the second copy of $X$. The Arthur--Selberg trace formula corresponds to the case $X=H$, a reductive group, under the action of $G=H\times H$ by left and right multiplication; in that case, $[X\times X/G^\diag] = [\frac{H}{H}]$, by which we denote the quotient of $H$ by $H$-conjugacy. In the general case where $X=H\backslash G$, we have an isomorphism $[X\times X/G^\diag]=[H\backslash G/H]$, and we will use these two ways to represent this quotient interchangeably.

In the entire paper we work over a local field $F$. The variety $X$ is assumed to be quasi-affine, and let us denote by $\C_X$ the invariant-theoretic quotient $X\times X\sslash G = \spec F[X\times X]^G$ (from now on, $G$ will always be understood to be acting diagonally on such a quotient). The space $$ \mathcal S(X\times X/G)$$
of \emph{stable test measures} for the associated relative trace formula is the push-forward of the Schwartz space of measures on $X\times X$ to $\C_X$. (In the Whittaker case, we need to define a ``twisted push-forward''; see \S \ref{sstwistedpf}.)

There is a notion of \emph{stable relative character} $J_\Pi^X$ attached to an ($L$-packet $\Pi$ of) irreducible representation(s) of $G$, generalizing the stable character of a reductive group; it is a functional on $\mathcal S(X\times X/G)$. If there is a non-zero such relative character, we say that $\Pi$ is \emph{$X$-distinguished}. A morphism of $L$-groups
$$ \LX_1\to \LX_2$$
of two spherical varieties $X_1, X_2$ should, according to the relative (local) functoriality conjecture, induce a map
$$ \{X_1\mbox{-distinguished $L$-packets}\} \longrightarrow \{X_2\mbox{-distinguished $L$-packets}\},$$
at least for those $L$-packets that participate in the corresponding Plancherel formulas (the \emph{$X_i$-tempered} ones). A basic proposition of ``beyond endoscopy'' is that the resulting map of stable relative characters
$$ J_{\Pi_1}^{X_1} \mapsto J_{\Pi_2}^{X_2}$$
should be realized as the adjoint of a ``transfer operator'' between test measures:
$$ \mathcal T: \mathcal S(X_2\times X_2/G_2) \to \mathcal S(X_1\times X_1/G_1).$$
In the group case, this operator has been studied by Langlands \cite{Langlands-ST} and Johnstone \cite{Johnstone} when $X_1$ is a torus and $X_2$ is $\GL_n$. Ideally, this operator should be used in a global comparison of trace formulas, to establish the corresponding functorial transfer of \emph{automorphic} representations, but it is not clear at this moment how to do that in the context of the Arthur--Selberg trace formula.

The goal of this paper is to study such transfer operators $\mathcal T$, and to prove that they have some structure. In our comparisons, $X_2$ will always be the Whittaker model for a group $G^*$ (so that $\LX_2 = \LG^*$, and $\mathcal S(X_2\times X_2/G^*) = \mathcal S(N,\psi\backslash G^*/N,\psi)$ is the space of test measures for the Kuznetsov formula), and $X_1$ will be a different spherical variety $X$ for a group $G$. In all cases but one (the study of Venkatesh's thesis in Section \ref{sec:Venkatesh}), we will have $\LX = \LG^*$. The transfer operator that we construct originates from an \emph{enlargement} $\mathcal S^-_{L_X}(N,\psi\backslash G^*/N,\psi)$ of the space of test measures for the Kuznetsov formula. This is to be expected, because in a global comparison where the following diagram would commute:
\begin{equation}\label{commRTF} \xymatrix{ \mathcal S^-_{L_X}(N,\psi\backslash G^*/N,\psi) \ar[rr]^{\mathcal T}\ar[dr]_{\RTF} & & \ar[dl]^{\RTF}\mathcal S(X\times X/G)\\
& \CC &}\end{equation}
(where $\RTF$ denotes the relative trace formula functional), the spectral side of the relative trace formula for $X$ is, roughly (and conjecturally, cf.\ \cite[\S 17]{SV}), an integral over $X$-distinguished automorphic representations $\pi$ of certain quotients of $L$-values of the form $\frac{L_X(\sigma_\pi)}{L(\sigma_\pi,\Ad, 1)}$, where $\sigma_\pi$ is an automorphic representation of $G^*$ of which $\pi$ is a lift, and $L_X$ is a certain $L$-value of $\sigma_\pi$, while the corresponding term for the Kuznetsov formula with \emph{standard} test functions is just $\frac{1}{L(\sigma_\pi,\Ad, 1)}$ --- thus, the factor $L_X(\sigma_\pi)$ has to be inserted. The enlargement of the standard space of test measures should also be expected for representation-theoretic reasons (both global and local), because the spectrum of the space $X$ is typically larger, containing, for example, the trivial representation, which is not present on the Kuznetsov side with usual test functions, but is added in this larger space. (This feature already appeared in \cite{SaBE1, SaBE2}, where the first comparison of this type was performed, both locally and globally.)  This enlargement of the Schwartz space (see \S \ref{ssnonstandard}) will be done in a somewhat ad-hoc way; I do not yet know a good way\footnote{Although Rapha\"el Beuzart-Plessis has informed me that he does!} to characterize the enlarged space $\mathcal S^-_{L_X}(N,\psi\backslash G^*/N,\psi)$ for any $L$-value $L_X$, but at least, in the non-Archimedean case, it should contain the image of the ``generating Whittaker measure of the local $L$-value $L_X$'', that measure (or function) whose Poincar\'e series, at the unramified level, extracts the desired $L$-value from each automorphic representation (at least, formally). In approaches to ``beyond endoscopy'' expressed in classical language, this non-standard Poincar\'e series corresponds to ``a series of Kuznetsov formulas'' with varying test functions, closely modelling the Dirichlet series of the desired $L$-value $L_X$.

We may also let our varieties degenerate to their \emph{asymptotic cones} (or boundary degenerations). For the space of Whittaker measures $\mathcal S(N,\psi\backslash G^*)$ this means letting the character $\psi$ become trivial. For the spherical variety $X$, the asymptotic cone $X_\emptyset$ is obtained by considering a graded version of its coordinate ring; for example, when $X=\SL_2$, we have $X_\emptyset = $ the variety of $2\times 2$-matrices of rank one. We will explain that, in this case, there is a very natural transfer operator
$$ \xymatrix{ \mathcal S^\pm_{L_X}(N\backslash G^*/N) \ar[rr]^{\mathcal T_\emptyset} & & \mathcal S^\pm(X\times X/G)}$$
between suitable spaces of test measures that we denote by $\mathcal S^\pm$, given as a composition of \emph{multiplicative Fourier convolutions} $\mathscr F_{\check\lambda, s}$:
\begin{equation}\label{Fourierconv} \mathscr F_{\check\lambda, s}f(\xi):= \int_{\Gm} f( \check\lambda(x)^{-1} \xi) \psi(x) |x|^s d^\times x.\end{equation}
Here, $f$ is a measure on a ``universal Cartan'' $A_X = A_{G^*}$, and $\check\lambda$ is a cocharacter into that torus. Notice that we habitually denote the set of $F$-points of a variety $X$ simply by $X$, so an integral over $\Gm$ just means an integral over $F^\times$.

To summarize, we have a good understanding of ``degenerate'' transfer operators $\mathcal T_\emptyset$, and we would like to study the transfer operators $\mathcal T$ of the original problem. The findings of this paper can be summarized as follows:

\begin{mainconclusions*}
\begin{enumerate}
 \item In a range of examples, the transfer operators $\mathcal T$ can be calculated explicitly, and are linear isomorphisms
 $$\mathcal T:\xymatrix{ \mathcal S^-_{L_X}(N,\psi\backslash G^*/N,\psi) \ar[rr]^{\sim} & & \mathcal S(X\times X/G)}$$
 satisfying the fundamental lemma (i.e., sending the ``basic vector'' of one space to the ``basic vector'' of the other, and also the ``fundamental lemma for the Hecke algebra'').
 \item The transfer operators computed are indeed the ``correct'' ones for functoriality: the pullbacks of relative characters for $X\times X/G$ are relative characters for the Kuznetsov trace formula of $G^*$.
 \item They are \emph{deformations} of the degenerate transfer operators $\mathcal T_\emptyset$, in the sense that they are given by similar-looking formulas.
 \item They are \emph{amenable to a global Poisson summation formula}, in the sense that they are given by multiplicative Fourier convolutions and multiplication by factors that are trivial on global rational points, thus they seem to fit in a commutative diagram like \eqref{commRTF}. 
\end{enumerate}
\end{mainconclusions*}

I do not perform such a global comparison here, but I expect that the methods employed in \cite{SaBE2} would apply to all cases discussed here.

The above conclusions are not restricted to the transfer operators that should be responsible for the functoriality map corresponding to 
$$ \LX_2\to \LX_1,$$
but also to the so-called \emph{Hankel transforms} that are responsible for the functional equation of certain $L$-functions. These are addressed in the continuation to this paper, \cite{SaTransfer2}. While transfer operators have the property that they pull back relative characters to relative characters normalized in a distinguished way:
\begin{equation}
 \mathcal T^* J_{\Pi_1}^{X_1}  = J_{\Pi_2}^{X_2},
\end{equation}
the Hankel transforms are transforms between enlarged spaces of test measures of the \emph{same} relative trace formula (for us, always the Kuznetsov formula), but with the enlargement corresponding to dual $L$-functions:
 $$\mathcal H_r:\xymatrix{ \mathcal S^-_{L(r)}(N,\psi\backslash G/N,\psi) \ar[rr]^{\sim} & & \mathcal S^-_{L(r^\vee)}(N,\psi\backslash G/N,\psi)}$$
and having the property that they act on relative characters \emph{by the gamma factor of the local functional equation}:
\begin{equation}
 \mathcal H_r^* J_{\Pi}  = \gamma(r) J_{\Pi}.
\end{equation}
For example, when $G= \GL_n$ and $r=\Std$, the standard representation of the dual group, the space $\mathcal S^-_{L(r)}(N,\psi\backslash G/N,\psi)$ is just the image of the Schwartz space $\mathcal S(\Mat_n)$ of the variety of $n\times n$-matrices, and the Hankel transform $\mathcal H_{\Std}$ is the descent of Fourier transform, computed by Jacquet in \cite{Jacquet}. In this paper, we will compute the Hankel operator $\mathcal H_{\Sym^2}$ for $r=$ the \emph{symmetric square $L$-function} of $\GL_2$, and verify that both $\mathcal H_{\Std}$ and $\mathcal H_{\Sym^2}$ satisfy the Main Conclusions listed above.

\subsection{Index of main results}

References to chapters \ref{sec:intro2} and above refer to the continuation of this paper, \cite{SaTransfer2}. The main results of this series of two papers are:

\begin{itemize}
 \item Theorem \ref{thmRudnick}, on the comparison between the Kuznetsov formula and the stable trace formula for $\SL_2$. It is complemented by Theorem \ref{groupdegen} which performs the same comparison for their asymptotic cones.
 \item Similarly, Theorem \ref{torusdegen} performs the asymptotic cone comparison for the transfer between the Kuznetsov formula for $\PGL_2$ and the relative trace formula for the quotient $\Gm\backslash \PGL_2/\Gm$. The actual transfer between these two quotients was performed earlier in \cite{SaBE1}.
 \item Theorem \ref{thmSym2}, on the Hankel transform responsible for the functional equation of the symmetric-square $L$-function of $\GL_2$ (or, more correctly, of $\Gm\times\PGL_2$). It is complemented by Theorems \ref{Sym2degen} and \ref{Stddegen} that compare this Hankel transform (and the one for the standard $L$-function) to the corresponding transforms on the asymptotic cone.
 \item Theorems \ref{transfertokappa} and \ref{twistedRudnick}, which describe the transfer from the Kuznetsov formula to the endoscopic parts of the trace formula for $\SL_2$, and Theorem \ref{thmVenkatesh}, which describes the transfer from the Kuznetsov formula of $\SL_2$ to the Schwartz space of a torus.
\end{itemize}

\subsection{Example: Comparison between the Kuznetsov formula and the stable trace formula for $\SL_2$.} \label{ExRudnick}

In Section \ref{sec:Rudnick} (with some proofs to be postponed until Section \ref{sec:sym2}) I develop the local theory behind Rudnick's thesis \cite{Rudnick} (and its potential generalization to non-holomorphic automorphic forms). Here, the goal is to obtain a transfer operator
$$\mathcal T: \mathcal S^-_{L_X}(N,\psi\backslash G/N,\psi) \xrightarrow\sim \mathcal S(\frac{G}{G})$$
between the spaces of test measures for the Kuznetsov formula and the stable trace formula for $G=\SL_2$. The operator should have the property that its adjoint takes the stable character $\Theta_\Pi$ of any tempered $L$-packet $\Pi$ to the relative character (Bessel distribution) $J_\Pi$ of the Kuznetsov formula, thus realizing the functoriality map corresponding to 
$$ \LX_1 \xrightarrow\sim \LX_2,$$
where $X_1 =\SL_2$ and $X_2=$the Whittaker model of $\SL_2$. 

I show that such an operator exists, if we take an enlarged space of test measures $\mathcal S^-_{L_X}(N,\psi\backslash G/N,\psi)$ corresponding to $L_X=$the adjoint $L$-function of $\SL_2$ evaluated at $1$; what this means, at the very least, is that $\mathcal S^-_{L_X}(N,\psi\backslash G/N,\psi)$ will contain the image of the unramified Whittaker measure which corresponds to the coefficients of the local Dirichlet series for $L(\Ad,1)$ --- see \S \ref{ssnonstandard}. Explicitly, if we choose representatives $\begin{pmatrix} & -\zeta^{-1} \\ \zeta \end{pmatrix}$ for generic $N\times N$-orbits on $G$, this is the usual space of test measures for the Kuznetsov formula, except that, instead of being of rapid decay as $\zeta\to \infty$, the test measures will be equal to $C(\zeta^{-1}) d^\times \zeta$, where $C$ is a smooth function at zero. 

Moreover, I show that this operator has a very simple form: Notice that the elements of $\mathcal S^-_{L_X}(N,\psi\backslash G/N,\psi)$ are measures on the one-dimensional affine space $N\backslash G\sslash N$ (where we use the coordinate $\zeta$ above). Similarly, the elements of $\mathcal S(\frac{G}{G})$ will also be measures on an one-dimensional space, namely the Chevalley quotient $\Dfrac{G}{G}$, where we take the coordinate to be the trace. It turns out that the transfer operator $\mathcal T$ is given by the multiplicative Fourier convolution $\mathscr F_{\Id, 1}$, discussed above, applied to measures on the affine line (under the action of the multiplicative group on the coordinates that we fixed). 

The Fourier transform of Rudnick's thesis also appears, in a slightly different and certainly more general form, in the thesis of Altug \cite{Altug1,Altug2,Altug3}. It also appears in a theorem of Soundararajan and Young \cite[Theorem 1.3]{SoundYoung} that is very close to the result that I prove here (it is a global version of the comparison, restricted to hyperbolic conjugacy classes). On the other hand, the spectral side of the comparison between the Kuznetsov formula and the trace formula appears, in the setting of complex loop groups and $D$-modules, features in recent work of Ben-Zvi and Gunningham in \cite{BZvi-Gunningham}.

\subsection{Scattering theory and asymptotics}

Why should Fourier transforms on the line be of any relevance? Are the spaces $N\backslash G\sslash N$ or $\Dfrac{G}{G}$ vector spaces, in any sense? This is probably a misleading point of view, which does not seem to lead to correct conclusions in higher rank. What I show here is that there is another, conceptual way to understand these quotient spaces, or at least their ``limits'' when we replace the spaces by their asymptotic cones. Most of the present, first part of this series of two papers is devoted to developing the harmonic analysis necessary in order to produce this conceptual explanation of the transfer operators.

That is, we will replace the non-trivial character $\psi$ of $N$ by the trivial character, and we will replace the quotient $\frac{G}{G} = G^\diag\backslash (G\times G)/G^\diag$ by $G^\diag\backslash (G_\emptyset\times G_\emptyset)/G^\diag$, where $G_\emptyset$, the asymptotic cone of $\SL_2$, is the variety of $2\times 2$ matrices of determinant one. Notice that the general setting of the relative trace formula is essential here, even when we are studying the usual trace formula: to study its ``asymptotics'', we need to replace the group by a different space, and view the adjoint quotient of the trace formula as a special case of a quotient space for the relative trace formula.

When we do that, the quotients $N\backslash G/N$ and $G^\diag\backslash (G_\emptyset\times G_\emptyset)/G^\diag$ naturally become embeddings of the \emph{same} torus $A\simeq \Gm$ (the universal Cartan of $\SL_2$). Character theory on these degenerate spaces is particularly simple (because they are essentially induced from a torus), and we use local harmonic analysis (scattering theory, see Section \ref{sec:scattering}) to explain that the correct transfer operator (the one that behaves in a prescribed way with respect to characters) between suitable spaces of test measures
$$\mathcal T_\emptyset: \mathcal S^\pm_{L_X}(N\backslash G/N) \xrightarrow\sim \mathcal S(G^\diag\backslash (G_\emptyset\times G_\emptyset)/G^\diag)$$
is given by the multiplicative Fourier convolution $\mathscr F_{\Id, 1}$. Thus, geometrically $\mathcal T$ is given by the same formula as $\mathcal T_\emptyset$, if we choose appropriate coordinates. (I should mention that I do not have a conceptual reason for this on-the-nose equality; in fact, it fails in other cases, as we will see, although one operator is still a deformation of the other.)

I mention that exactly the same phenomenon is true for the comparison between the Kuznetsov formula for $\PGL_2$ and the relative trace formula for $\Gm\backslash \PGL_2/\Gm$, which was developed in \cite{SaBE1}. I briefly revisit this case in Section \ref{sec:Waldspurger}.

\subsection{Open problems}

I hope that the present paper will raise more questions than it settles. Let me list some of them:

\begin{enumerate}
 \item Let $G$ be a quasi-split group, and $\psi$ a generic character of a maximal unipotent subgroup $N$. Let $r:\LG\to \GL(V)$ be a representation of its $L$-group, not necessarily irreducible. It does not harm to assume that there is a character $\partial$ of $G$ whose dual, composed with $r$, is the canonical (central) cocharacter $\Gm\to\GL(V)$. Attached to these data there should be a canonical space $\mathcal D^-_{L(r,\frac{1}{2})}(N,\psi\backslash G/N,\psi)$ of test half-densities for the Kuznetsov formula, containing (in the non-Archimedean case) the image of the generating series of the local unramified $L$-value $L(r,\frac{1}{2})$, and such that the integrals 
 $$ \int f J_\pi |\partial|^s$$
 of elements $f\in \mathcal D^-_{L(r,\frac{1}{2})}(N,\psi\backslash G/N,\psi)$ against the relative characters $J_\pi$ of irreducible representations admit a Godement--Jacquet theory: for example, their quotients by the local $L$-function $L(r,\frac{1}{2}+s)$ should be entire in $s$, and there should be a ``Hankel transform'' 
 $$ \mathcal H_r: \mathcal D^-_{L(r,\frac{1}{2})}(N,\psi\backslash G/N,\psi) \xrightarrow\sim \mathcal D^-_{L(r^\vee,\frac{1}{2})}(N,\psi\backslash G/N,\psi)$$
 satisfying a local functional equation involving the associated local $\gamma$-factors.
 
 None of these properties completely characterizes the space. It would be desirable to have a spectral characterization (possibly by means of a local relative trace formula), as well as a geometric one. 
 If $r$ is irreducible, this would be the image of the Schwartz space $\mathcal D(G_r)$ or the $L$-monoid of Ng\^o \cite{Ngo-PS}, which also has not been defined yet. It is interesting to ask whether it is easier to define this space at the level of the Kuznetsov formula. Although the definitions of such spaces that I give in this paper are somewhat ad-hoc, one may observe that at least for $r=$ many copies of the standard or the symmetric-square representation of $\GL_2$, the definition of $\mathcal D^-_{L(r,\frac{1}{2})}(N,\psi\backslash G/N,\psi)$ seems to be quite straightforward (see \S \ref{ssnonstandard}), while monoids are not well-suited to handle multiplicity.

 \item Once local transfer operators or Hankel transforms are available, and have a form that ``in principle'' satisfies a Poisson summation formula, it would be desirable to develop such a summation formula, which would amount to an identity of relative trace formulas. Such an application was developed in \cite{SaBE2}, and I do not see serious obstacles to adapting the methods to the transforms of the present paper; however, streamlining the arguments for the global application would be important in light of future developments, and would enhance our understanding of the nature of orbital integrals close to singularities (and the behavior of those under non-standard transfer operators). For example, one could try to upgrade the local Hankel transform of \eqref{Hankel-Sym2-intro} to a trace formula-theoretic proof of the functional equation of the symmetric-square $L$-function. 
 
 \item Although the fundamental lemma for the Hecke algebra is proven in this paper for all transfer operators and Hankel transforms considered, it would be desirable to have a geometric proof of the fundamental lemma, as in the endoscopic case \cite{Ngo-FL}. Such a proof would apply, in particular, to the more general transfer operators considered in the upcoming paper \cite{SaRankone}, where there is ongoing work of Johnstone and Krishna on the fundamental lemma.
 
 \item Most important of all, though, is to understand the nature of transfer operators and Hankel transforms, and how they generalize to higher rank. In this paper, I show that these operators are ``deformations'' of abelian Fourier convolutions of the corrresponding transforms on the horospherical boundary degenerations, which are completely understood. In examples of rank one, discussed in Sections \ref{sec:Rudnick} and \ref{sec:Waldspurger}, they are actually \emph{equal} to those Fourier convolutions, for a suitable choice of coordinates; this is generalized to all rank one varieties in the upcoming paper \cite{SaRankone}. The Hankel transforms, however, discussed in Section \ref{sec:sym2}, require intermediate ``correction factors'' that I do not understand. This is also the case for transfer operators in higher rank, which I have computed for some more examples jointly with Chen Wan, generalizing the calculation of Section \ref{sec:Rudnick}. Understanding the nature of these deformations is, in my mind, the quintessential issue in order to make progress towards ``beyond endoscopy'' in higher rank.
 
 Regarding Hankel transforms, Ng\^o has recently formulated a conjecture, stunning in its simplicity, about the kernel of the transform giving rise to the functional equation of \emph{any} $L$-function, as an invariant distribution on the group \cite{Ngo-Takagi}. It would be desirable to know what transformation it induces at the level of the Kuznetsov formula.

\end{enumerate}

\subsection{Notation}\label{ssnotation}

We work over a local (locally compact) field $F$. Whenever no confusion arises, I denote the set of $F$-points of a variety $X$ simply by $X$; for example, an integral of the form $\int_{\Gm}$ denotes an integral over the group $F^\times$. When discussing (Langlands) dual groups, I will similarly denote them as algebraic groups, e.g.: $\Gm$ is the dual group of $\Gm$; it will be clear from the context if we are referring to a group or its dual.

The categorical quotient $\spec F[X]^G$ of an affine or quasi-affine variety $X$ by a $G$-action will be denoted by $X\sslash G$. When $X$ is a group and $G$ a subgroup acting on $G$ by conjugacy, I will denote this quotient by $\Dfrac{X}{G}$ (and will use $\frac{X}{G}$ as a formal notation for what is denoted by $X/G$ below).

The notation $X\times^G Y$ will denote the quotient of a product $X\times Y$ of two $G$-varieties by the diagonal action of the group $G$. The notation will be used when $G$ acts freely, and the quotient exists as a variety; only when explicitly discussing stacks will this notation be used for the quotient in the stacky sense.

For a product of spaces $X\times Y$, I will sometimes use the notation $\mathcal L_X\boxtimes \mathcal L_Y$ to indicate the tensor product of two vector bundles, one pulled back from $X$ and the other from $Y$. I will sometimes use similar notation for operators (e.g., $S\boxtimes T$), to stress that each is applied to a different variable. ``Vector bundles'' (and especially ``line bundles'') will sometimes refer to complex vector bundles over the $F$-points of a variety; they correspond to the $l$-sheaves of Bernstein and Zelevinsky, in the non-Archimedean case, and to complex vector bundles for the smooth topology, in the Archimedean case; in particular, in both cases the notion of smooth sections is defined. Typically, at least in the Archimedean case, there will also be a natural notion of ``polynomial growth'' for sections of these bundles (i.e., they will be Nash bundles on Nash manifolds, cf.\ \cite{AGSchwartz}). When this structure is clear from the setting, I will be using it without explanation.

The space of ($\CC$-valued) Schwartz measures on the $F$-points of a smooth variety $X$ will be denoted by $\mathcal S(X)$. These are just smooth, compactly supported measures in the non-Archimedean case; in the Archimedean case, they are smooth measures which decay rapidly, together with their polynomial derivatives, cf.\ \cite{AGSchwartz}. (For ease of language, we will often not differentiate between the Archimedean and non-Archimedean cases, and say ``rapid decay'' for both; the reader should interpret this as ``compact support'' in the non-Archimedean case.) These are sections of a cosheaf for the semi-algebraic topology on $X$ (or, for the usual topology in the non-Archimedean case); at a few points, I will talk about the stalks over a closed subset $Y$, which are simply defined as the quotient $\mathcal S(X)/\mathcal S(X\smallsetminus Y)$. Moreover, at a few points I will need to work in the more general context of stacks, instead of varieties. The appropriate notion of a Schwartz space of measures, in this context, was introduced in \cite{SaStacks}, but I make an effort to describe them explicitly in the examples at hand, so that the reader will not require this background.

For smooth varieties, it also makes sense to talk about spaces of Schwartz functions or half-densities; those will be denoted by $\mathcal F(X)$, resp.\ $\mathcal D(X)$. Of course, if $dx$ is a nowhere vanishing smooth positive measure of polynomial growth (such as a Haar measure), we have $\mathcal S(X) = \mathcal F(X) dx = \mathcal D(X) (dx)^\frac{1}{2}$.

The image of the push-forward map from $\mathcal S(X)$ to measures on $X\sslash G$ will be denoted by $\mathcal S(X/G)$. This is a slight departure from notation used in \cite{SaStacks} for the Schwartz space of the quotient stack $[X/G]$, so, whenever I actually need a more stacky version of such a Schwartz space, I will be using notation of the form $\mathcal S([X/G])$ (and will explain what I mean by it, in each case). Notice that $\mathcal S(X/G)$ is different from $\mathcal S(X\sslash G)$ --- the latter is simply the usual Schwartz space of the affine variety $X\sslash G$ (assumed smooth), and it is typically, but not always, contained in $\mathcal S(X/G)$. One can typically translate from measures to functions (by choosing appropriate Haar measures), and then the space $\mathcal S(X/G)$ corresponds to the space of ``stable orbital integrals'' for the $G$-action on $X$. 

We broaden these notions (and notations), to include spaces of the form $\mathcal F((H,\chi)\backslash G)$, when $\chi$ is a complex character of a subgroup $H$ of $G$; in this case, this notation means $(H,\chi)$-equivariant functions on $G$, which are smooth and of compact support (in the non-Archimedean case) or rapid decay (in the Archimedean case) modulo $H$. The character $\chi$ being of ``polynomial growth'' means that the notion of rapid decay modulo $H$ makes sense, by choosing semi-algebraic local sections of the map $G\to H\backslash G$. Elements of $\mathcal F((H,\chi)\backslash G)$ can also be thought of as Schwartz sections of a complex line bundle $\mathcal L_\chi$ over $H\backslash G$; we can, similarly, consider Schwartz measures or half-densities. For a character $\psi$ of a maximal unipotent group $N$ of $G$, we will be using the notation $\mathcal S(N,\psi\backslash G/N,\psi)$ for the space of ``Schwartz test measures for the Kuznetsov quotient of $G$'', see \S \ref{sstwistedpf}. The appropriate way to think of those is as measures valued in a complex line bundle over the stack $[N\backslash G/N]$, but here we just understand them as measures on the affine quotient $N\backslash G\sslash N$, by some conventions that we explain in \S \ref{sstwistedpf}. Sometimes, we will treat the symbol $(N,\psi\backslash G/N,\psi)$ as a ``space'', e.g., we will be talking about push-forward of measures to that ``space'', meaning the twisted push-forward to $N\backslash G\sslash N$ that is described in that section.

All of these Schwartz spaces will be viewed as abstract vector spaces, in the non-Archimedean case, and as (nuclear) Fr\'echet spaces, in the Archimedean case. The Fr\'echet structure is the usual one (see \cite{AGSchwartz}) for $\mathcal S(X)$, while $\mathcal S(X/G)$ will inherit the quotient topology. By $\hat\otimes$ I denote the completed tensor product of nuclear Fr\'echet spaces; the completion should be ignored in the non-Archimedean case, as should any references to topology. (For convenience of language, I do not always differentiate between the Archimedean and the non-Archimedean case.) The space $V_G$ of coinvariants of a Fr\'echet representation of a group $G$ will, by definition, be the completion of the algebraic coinvariant space, that is, the quotient of $V$ by the \emph{closure} of the set of vectors of the form $v-g\cdot v$.

In the Archimedean case, the appropriate category of Fr\'echet $G$-repre\-sentations to consider is that of \emph{Fr\'echet representations of moderate growth}, or \emph{$F$-representations}, in the language of \cite{BeKr}: these are countable inverse limits of $G$-representations, i.e., they have a complete system of seminorms for which the $G$-action is continuous. I point the reader to \cite{BeKr} for more details, and for the corresponding ``smooth'' notion of \emph{$SF$-representations}.

We fix throughout a non-trivial unitary character $\psi$ of the additive group $F$. If $F$ is non-Archimedean, we will be assuming that its conductor is the ring of integers $\mathfrak o$. We also fix a measure $dx$ on $F$ which is self-dual with respect to $\psi$; this induces a measure $|\omega|$ on $X(F)$, for every volume form $\omega$ on a smooth variety $X$ over $F$. We will also use the measure $d^\times x := \frac{dx}{|x|}$ on the multiplicative group $F^\times$. When $F$ is non-Archimedean and unramified over the base field $\mathbb Q_p$ or $\mathbb F_p((t))$, $dx(\mathfrak o)=1$. The absolute value of any local field is defined to be compatible with the one of the base field under the norm map; in particular, the absolute value on $\CC$ is the square of the usual one.

We write $f\ll g$ for two positive functions on a space $X$ to indicate that there is an absolute constant $C$ such that $f(x)\le C g(x)$.

For a representation $V$ of a group $G$, and a measure $h$ on $G$ (or, more general distributions in some settings, see \S \ref{ssmultipliers}), we write $h\cdot v$ for the integral of the $G$-action against $h$:
$$h\cdot v := \int_G h(g) (g\cdot v),$$
whenever this makes sense. When $G$ acts on a space $X$ (say, on the right), the action on functions, half-densities or measures will be the regular one (which is a left action); in particular, 
$$ h\cdot f(x) = \int_G h(g) f(xg).$$
There is also another, right action in this case, the convolution action, corresponding to the push-forward under the action map $X\times G\to X$. The two are related by
$$ h\star f(x) = \int_G h(g^{-1}) f(xg) = h^\vee\cdot f(x),$$
where $h^\vee (g) = h(g^{-1})$. 

For a torus $T$, I denote by $\hat T$ the group of its unitary (complex) characters (that it, characters of $T(F)$), and by $\hat T_\CC$ its \emph{complexification}, the group of all complex characters. The notion of ``complexification'' makes sense, here, because $\hat T$ is naturally a real algebraic variety (with infinitely many components, in general), if $F$ is non-Archimedean, and a real analytic subgroup of a complex analytic group $\hat T_\CC$, when $F$ is Archimedean. I point the reader to \S \ref{ssMellin} for more details.

When the torus $T$ acts on a space $X$, at various points in the paper I define Mellin transforms $\hat T_\CC\ni \chi\mapsto \check f(\chi)$ of functions, half-densities or measures $f$ on $X$. I note that my parametrization of Mellin transforms is such that the map $f\to \check f(\chi)$ is $(T,\chi)$-equivariant (sometimes, for a normalized action); this is inverse to the classical definition of Mellin transform of a function on $\Rplus$ as $\check f(s) = \int f(x) x^s d^\times x$. By this convention, if $h$ is a measure on $T$ with Mellin transform $\check h$, and $f$ is a measure on the space $X$, we have
$$ \widecheck{h\star \varphi}(\chi) = \check h(\chi) \check f(\chi)$$
for the convolution action, and 
$$ \widecheck{h\cdot\varphi}(\chi) = \check h(\chi^{-1}) \check f(\chi)$$
for the regular action. I caution that this convention for Mellin transforms is also inverse to the conventions about Satake transforms of split tori over non-Archimedean fields; if we identify the identity component of $\hat T_\CC$ with the (complex points of the) Langlands dual torus $\check T$, then the Satake transform of a measure $h$ on $T(F)/T(\mathfrak o)$ is equal to what we denote by $\check h(\bullet^{-1})|_{\check T}$. 

For a reductive group $G$, its \emph{universal Cartan}, or simply its Cartan $A_G$ is not a subgroup, but an abstract torus, defined as the quotient of any Borel subgroup $B$ by its unipotent radical $N$. It is unique up to unique isomorphism, and defined over $F$ even when $B$ is not. It comes equipped with a based root datum; in particular, the ``positive'' and ``negative'' roots of $A_G$ are well-defined subsets of its character group, and similarly for coroots. We will typically use additive notation for weights and coweights of a torus, so when we need think of them as morphisms to or from the multiplicative group $\Gm$, we may use exponential notation, like $e^\alpha$.

The Cartan of $G$ is the appropriate basis for the definition of the dual group: the connected dual group $\check G$ of $G$ contains a canonical Cartan which is dual to $A_G$, and the $L$-group $\LG$ contains the $L$-group of $A_G$. 
Given a representation $r$ of $\LG$, the associated local Langlands $L$-function of an irreducible representation $\pi$ of $G$ will be denoted by 
$$ L(\pi, r, s).$$
For example, when $G$ is a torus and $\check\lambda:\Gm\to G$ a cocharacter (thought of as a character of the dual torus), $L(\chi,\check\lambda, s)$ stands for the local Dirichlet $L$-function $L(\chi\circ e^{\check\lambda}, s)$. This notation for Dirichlet $L$-functions will also be used in this paper, and we will denote the local Dedekind zeta function of $F$ simply by $\zeta(s)$. We will sometimes also use alternative, standard notation for some $L$-functions, e.g.,
$ L(\pi_1\times \pi_2, s)$, $L(\Sym^2(\pi), s)$ etc., for the Rankin--Selberg, resp.\ the symmetric-square $L$-function of $\GL_n$. If $\chi$ is a multiplicative character, we may also use expressions of the form $L(\chi\times \Sym^2(\pi), s)$ for the $L$-function that arises from the stated ($\Sym^2$, in this example) representation of the dual group of $G$, tensored by the scalar action of $\Gm=$ the dual group of $\Gm$. We denote by $|\bullet|$ the absolute value character of $\Gm$ (that is, of $F^\times$), so, in this notation, 
$L(|\bullet|^s\times \Sym^2(\pi),0)$ is the same as $L(\Sym^2(\pi),s)$. 

For a large part of the paper (from Section \ref{sec:RS} on) we will consistenly be using $A$ to denote the Cartan of $\SL_2$, and $A_\ad$ for the Cartan of $\PGL_2$. We will consistently be identifying these groups with $\Gm$, the former via the positive half-root character, and the latter via the half-root character. These identifications translate the natural map $A\to A_\ad$ to the square map $\Gm\to \Gm$. Sometimes, when it is clear which of these two groups we are talking about, we will be using these identifications to write their $L$-functions as Dirichlet $L$-functions, that is:
\begin{description}
 \item for a character $\chi$ of the group $A$, $L(\chi,s):=L(\chi,\check\alpha,s)$;
 \item for a character $\chi$ of the group $A_\ad$, $L(\chi,s):= L(\chi,\frac{\check\alpha}{2},s)$.
\end{description}
Since these notations are not compatible with the pullback map of characters, we will be careful not to use them when both groups are in play.

The notation $\gamma(\pi, r, s, \psi)$ will be used to denote the factors of the local functional equation of an $L$-function, cf.\ \S \ref{ssTate} and \S \ref{ssHankelrelchars}. It is related to the $L$- and epsilon-factors by
\begin{equation}\label{epsilon}
 \gamma(\pi, r, s, \psi) L(\pi, r, s) = \epsilon(\pi, r, s, \psi) L(\pi, r^\vee, 1-s),
\end{equation}
where $r^\vee$ is the dual representation of $r$, so $L(\pi, r^\vee, 1-s) = L(\tilde\pi, r, 1-s)$. Here and throughout, $\tilde\pi$ denotes the admissible dual (contragredient) of a smooth, admissible representation $\pi$.

The notion of ``universal Cartan'' generalizes from groups to spherical varieties. If $X$ is a spherical variety (i.e., a normal connected variety with an open orbit $\mathring X$ for the Borel subgroup $B$) under the action of a reductive group $G$, the quotient $\mathring X\sslash N$ (where $N$ is the unipotent radical of $B$) has an action of $A_G=B/N$, that factors through a faithful action of a quotient $A_X$ of $G$ --- this is the Cartan of $X$. For example, for $X=H$ a reductive group under the $G=H\times H$-action (which we define to be a right action, i.e., $x\cdot (h_1, h_2):= h_1^{-1} x h_2$), by the Bruhat decomposition one sees that $A_X$ is the quotient of $A_G=A_H\times A_H$ by the subgroup of elements $({^wa}, a)$, where $a\in A_H$ and $w$ is the longest element of the Weyl group.

This definition creates some inconvenience for a \emph{horospherical} variety (=one where stabilizers contain maximal unipotent subgroups), like $X=N\backslash G$, because the natural action of $A=B/N$ by $G$-automorphisms on $X$ is not compatible with the above map $A\to A_X$ (which in this case is just an isomorphism), but is \emph{conjugate to it by the longest element of the Weyl group}. To resolve this notational problem, we consistenly define the $A$-action on $N\backslash G$ to be the \emph{twist} of the obvious one by the longest Weyl group element, that is, if $a\in A$ then we define
$$ a\cdot Nx := Nb x,$$
where $b\in B$ \emph{represents the element ${^wa}$} in $A=B/N$. We extend this convention to all horospherical varieties. This is an unfortunate nuissance, but it is more benign than having a different action of $A$ on the variety $X$, and a different on its universal Cartan $A_X$.

Finally, I mention that actions of groups on spaces of functions or measures on $G$-spaces will often be normalized, in order to be unitary. (Actions on half-densities need no such normalization, which makes them particularly convenient.) However, in order for the reader not to have to keep track of normalizations, I have made sure that they are not needed in the statements of the main theorems (unless explicitly stated otherwise). Similarly, the notation is adapted to unnormalized actions, even if we are working with normalized ones. For example, in \S \ref{ssHankel} we introduce certain spaces of test half-densities and measures, $\mathcal D^-_{L(\Sym^2, \frac{1}{2})}(N,\psi\backslash G/N,\psi)$ and $\mathcal S^-_{L(\Sym^2, 1)}(N,\psi\backslash G/N,\psi)$, for the Kuznetsov formula of the group $G=\Gm\times\SL_2$; the difference in notation, $L(\Sym^2, \frac{1}{2})$ vs.\ $L(\Sym^2, 1)$, has to do with their images under \emph{unnormalized} push-forward (integration over $\Gm$) to the Kuznetsov formula of $\SL_2$, despite the fact that in \S \ref{ssdescent}, and elsewhere, we work with a normalized version of this push-forward. The analog of this in a more the more familiar setting of Tate's thesis would be the spaces of Schwartz half-densities and measures on $\mathbbm A^1$, where a Tate zeta integral 
$$ \int f(a) |a|^s$$
against a Schwartz measure $f$ would be a holomorphic multiple of the local zeta function $\zeta(s+1)$, while a Tate zeta integral 
$$ \int \varphi(a) |a|^s (d^\times a)^\frac{1}{2}$$
against a Schwartz half-density $\varphi$ would be a holomorphic multiple of $\zeta(s+\frac{1}{2})$; hence, the analogous notation would be $\mathcal S_{L(\Id, 1)} (\Gm)$ for $\mathcal S(\mathbbm A^1)$, and $\mathcal S_{L(\Id, \frac{1}{2})}$ for $\mathcal D(\mathbbm A^1)$.

\subsection{Acknowledgements}

This paper would not have been possible without the constant encouragement, numerous conversations, and many references and ideas provided by Ng\^o Bao Ch\^au, who invited me to spend the winter and spring quarters of 2017 at the University of Chicago. I also thank Daniel Johnstone for a presentation of Venkatesh's thesis which initiated my understanding of it. I thank Valentin Blomer for various references on related results in analytic number theory. Last but not least, I am deeply indebted to the Institute for Advanced Study for providing me with the perfect environment in order to complete this work, during the academic year 2017--2018. In fact, my earlier papers \cite{SaBE1, SaBE2}, which gave birth to this cycle of ideas, were conceived during my previous stay at the Institute in the spring of 2011, hence this paper owes to the IAS in multiple ways.

This work was supported by NSF grant DMS-1502270, and by a stipend to the IAS from the Charles Simonyi Endowment.

\section{Basic tools: multiplicative Fourier convolutions, non-standard Kuznetsov test measures} \label{sec:background}

The tools presented in this section were also introduced in \cite{SaHanoi}. I summarize them briefly for the sake of completeness.

\subsection{Mellin transforms and multiplicative Fourier convolutions} 

\subsubsection{Mellin transforms} \label{ssMellin}

Let $T$ be a torus over $F$. The unitary dual of $T=T(F)$ will be denoted by $\widehat T$, and its entire character group by $\widehat T_\CC$. The character group has a natural structure of a complex manifold and, if $F$ is non-Archimedean, of a complex algebraic variety (with infinitely many components). The structure is automatically determined by its restriction to the identity component, which is the character group of $\Lambda_T:= $ the image of the map
\begin{eqnarray}\nonumber
  \log_T: T(F) \to \mathfrak t_\RR := \Hom_\Z(\Hom(T,\Gm), \RR)\\
 t \mapsto (\chi \mapsto \log|\chi(t)|).\label{logtorus}
\end{eqnarray}

This group is discrete, in the non-Archimedean case, so the identity component of $\widehat T_\CC$ is the complex torus with coordinate ring equal to the group ring of $\Lambda_T$. 

In the Archimedean case, 
we have $\Lambda_T = \mathfrak t_\RR$, and the identity component of $\widehat T_\CC$ can be identified with the dual $\mathfrak t_\CC^* = \Hom(T,\Gm)\otimes \CC$, by sending an element $s$ in the latter to the character $\chi_s (t) = e^{\left< \log_T t, s\right>}$. The space $\mathfrak t_\CC^*$ is the direct sum of $\RR$-vector subspaces:
$$ \mathfrak t_\CC^* = \mathfrak t_\RR^* \oplus i \mathfrak t_\RR^*,$$
with the imaginary summand corresponding to the identity component of $\widehat T$. By a  ``bounded vertical strip'' we will mean the preimage of a compact subset of the real subspace under the projection map.

In the Archimedean case there is a canonical splitting of the sequence
\begin{equation}\label{logArchim} 1\to T_0 \to T(F) \to \mathfrak t_\RR \to 1,
\end{equation}
where $T_0$ is the compact group $\ker \log_T$, whose image coincides with the subgroup generated by the subgroups $\check\lambda(\Rplus)$, where $\check\lambda$ ranges over the cocharacters into $T$. This identifies the character group 
\begin{equation}\label{dualArchimedean}
 \widehat T_\CC \simeq \widehat{T_0} \times \mathfrak t_\CC^*.
\end{equation}
Notice that the character group $\widehat{T_0}$ is discrete.

The \emph{Mellin transform} of a measure $f\in \mathcal S(T)$ is the function 
\begin{equation}\label{Mellin-torus}\widehat T_\CC\ni \chi\mapsto  \int_T f(a) \chi^{-1}(a).
\end{equation}

We recall the Paley--Wiener theorem: 

\begin{theorem}\label{PW-torus} 
In the non-Archimedean case, Mellin transform defines an isomorphism between $\mathcal S(T)$ and the space of polynomial functions on $\widehat T_\CC$ supported on a finite number of components.

In the Archimedean case, under the isomorphism \eqref{dualArchimedean}, it defines an isomorphism between $\mathcal S(T)$ and the completed tensor product of Fr\'echet spaces
 $$ \mathbb H^\PW(\widehat T_\CC):= \mathscr C(\widehat{T_0}) \hat\otimes \mathbb H^\PW(\mathfrak t_\CC^*).$$
Here, $\mathscr C(\widehat{T_0})$ denotes the dual of $\mathcal S(T_0)$, that is, the space of functions $\varphi$ on the discrete abelian group $\widehat{T_0}$ such that, for any norm $\Vert \bullet \Vert$ on the vector space $\widehat{T_0}\otimes_{\Z} \RR$, and any $N\ge 0$, the function $\Vert n \Vert^N \varphi(n)$ is bounded; and
$\mathbb H^\PW(\mathfrak t_\CC^*)$ is the Paley--Wiener space of entire functions on $\mathfrak t_\CC^*$ which are of rapid decay on bounded vertical strips, that is, on every bounded vertical strip $V$ and for every $N$ satisfy 
\begin{equation}\label{vertstrip} \sup_{s\in V} |f(s)| (1+\Im(s)|)^N <\infty,\end{equation}
where $|\bullet|$ denotes any norm on the imaginary subspace $i\mathfrak t_\RR^*$.
\end{theorem}

This theorem is very classical, but since in most references the Paley--Wiener theorem is stated for compactly supported smooth functions, the reader can consult \cite[\S 3.1]{SaSelberg} for a proof on Schwartz spaces.

Finally, I introduce the notion of \emph{average volume} of a torus with respect to the logarithmic map \eqref{logtorus}: The space $\mathfrak t_\RR$ has a canonical lattice, dual to the character lattice of $T$, hence a canonical Haar measure, induced from the standard measure on $\RR$. Given a measure $dt$ on $T$, we define 
\begin{equation}\label{avgvol}\AvgVol(T)= \AvgVol(T,dt) = \lim_{c\to\infty} \frac{dt(\log_T^{-1}(cB))}{\Vol(cB)},
\end{equation}
where $B$ is any ball around zero in $\mathfrak t_\RR$. If the character group is trivial, we take $\Vol(\RR^0)=1$, so $\AvgVol(T)=\Vol(T)$.

For example, if $T=F^\times$, with $F$ an unramified extension of $\mathbb Q_p$ with ring of integers $\mathfrak o$ and residual degree $q$, and we take $dt = d^\times x= \frac{d x}{x}$ with $dx(\mathfrak o)=1$, we have $\AvgVol(F^\times) = \frac{1-q^{-1}}{\log q}$. In general, if $dx$ is the self-dual measure with respect to the additive character $\psi$ on $F$, we have 
\begin{equation}\label{AvgVolFtimes}
\AvgVol(F^\times, d^\times x) = \Res_{s=0} \gamma(1-s,\psi),
\end{equation}
see \cite[(2.26)]{SaBE2}, where $\gamma$ is the gamma factor of the local functional equation of Tate integrals, to be recalled below in \eqref{gammaZeta}.

\subsubsection{Multipliers} \label{ssmultipliers}

Let $\mathcal M(T)$ denote the following categories of modules for $T$:

\begin{itemize}
 \item in the non-Archimedean case, smooth representations;
 \item in the Archimedean case, \emph{smooth representations of moderate growth on Fr\'echet spaces}, or, equivalently, countable inverse limits of Banach representations, cf.\ \cite{BeKr}.
\end{itemize}

The action of $T$ on any $V \in \mathcal M(T)$  extends to an action
$$ \mathcal S(T)\hat\otimes V \to V.$$
Here, $\hat\otimes$ denotes completed tensor product in the Archimedean case (where both spaces are Fr\'echet, and $\mathcal S(T)$ is nuclear), and should be identified with $\otimes$ in the non-Archimedean case. 

If $V$ denotes a space of measures on $T$, and the Mellin transform \eqref{Mellin-torus} can be extended by a convergent integral to $V$, for $\chi$ in some region in $\widehat T_\CC$, then for such $\chi$ we have
$$ \widecheck{(f \star \varphi)}(\chi) = \check f(\chi)\check\varphi(\chi),$$
for $f\in \mathcal S(T), \varphi\in V$, where $\star$ denotes convolution.

In fact, the action on any object of $\mathcal M(T)$ extends to a larger algebra 
$ \widehat{\mathcal S(T)}$,
consisting of those (tempered) distributions on $T$ which after convolution by elements of $\mathcal S(T)$, become elements of $\mathcal S(T)$. Indeed, for every $V\in \mathcal M(T)$,  the map $\mathcal S(T)\otimes V\to V$ (uncompleted tensor product!) is surjective --- in the Archimedean case, this is the Dixmier--Malliavin theorem. Thus, for every object $V\in \mathcal M(T)$ and $v = \sum h_i\cdot v_i\in V$, we can define the action of an element $h \in \widehat{\mathcal S(T)}$ as
$$ h\cdot v = \sum_{i=1}^n (h\star h_i)\cdot v.$$
To see that it is well-defined, use projectors $e_\chi \in \mathcal S(T)$ to the various $T_0$-types $\chi$ (i.e., characters of $T_0$, the maximal compact subgroup), and notice that the map $V\ni v \mapsto (v_\chi)_{\chi \in \widehat{T_0}}:= (e_\chi\cdot v)_\chi$ is injective. If $\sum h_i\cdot v_i=0$ then $\sum_i (h\star h_i) \cdot v_i = 0$, because 
$$ \left(\sum_i (h\star h_i) \cdot v_i\right)_\chi = \sum_i (e_\chi\star h) \star (e_\chi\star h) \cdot v_i =(e_\chi\star h) \cdot \sum_i  (e_\chi\star h) \cdot v_i  .$$

In the non-Archimedean case, $\widehat{\mathcal S(T)}$ coincides with the completed Hecke algebra (Bernstein center) of \emph{essentially compact} distributions, i.e., those distributions which become compactly supported after smoothing; these are the distributions whose Mellin transforms are polynomial on $\widehat T_\CC$, without the assumption of support on a finite number of components. 

In the Archimedean case, this includes, of course, the enveloping algebra of the complexified Lie algebra of $T$, and in particular the enveloping algebra of the image of $\mathfrak t_\RR$ under the splitting of \eqref{logArchim}, which via Mellin transform is identified with the polynomial algebra on $t_\CC^*$, pulled back to $\widehat T_\CC$ via the projection to the second factor of \eqref{dualArchimedean}.

\subsubsection{Tate zeta integrals} \label{ssTate}

Recall that we have fixed a non-trivial unitary character $\psi$ of the additive group $F$, and a measure $dx$ which is self-dual with respect to $\psi$. We set $d^\times x := \frac{dx}{|x|}$.

The \emph{Tate zeta integral} of a Schwartz function on the line, $\Phi \in \mathcal F(F)$, is the integral
$$ Z(\Phi, \chi, s) = \int_{F^\times} \Phi(x) \chi(x) |x|^s d^\times x.$$

Thus, the Tate zeta integral is the Mellin transform of the measure $\Phi d^\times x$, evaluated at the character $\chi'=\chi^{-1}|\bullet|^{-s}$. It is defined convergent when $\Re(s)\gg 0$, and extends to $\chi' \in \widehat{F^\times}_\CC$ as a rational function, in the non-Archimedean case, and a meromorphic one, in the Archimedean case. It is a holomorphic multiple of the $L$-factor $L(\chi, s)$, and of rapid decay in bounded vertical strips (away from the poles).

Defining the Fourier transform of the function as $\hat\Phi(t) = \int \Phi(u) \psi(ut) du$, the local functional equation of Tate \cite{Tate-Corvallis} defines a \emph{gamma factor} by:
\begin{equation}\label{gammaZeta}\gamma(\chi,s,\psi) Z(\varphi,\chi,s) = Z(\hat\varphi, \chi^{-1},1-s).
\end{equation}

The gamma factor can be written as
\begin{equation}\label{gammafactor}\gamma(\chi,1-s,\psi) = \frac{\epsilon(\chi,1-s,\psi) L(\chi^{-1},s)}{L(\chi,1-s)},\end{equation}
where the epsilon factor is entire.

We recall that, in the Archimedean case, the $L$-factor is as follows:
\begin{itemize}
 \item  If $F=\RR$ and $\chi = (\sgn)^\epsilon |\bullet|^s$, where $\sgn$ is the sign character and $\epsilon = 0$ or $1$, we have 
 $$ L(\chi, 0) = \pi^{-\frac{s+\epsilon}{2}} \Gamma(\frac{s+\epsilon}{2}).$$
 \item
 If $F=\CC$ and $\chi(z) = e^{i \cdot m  \arg(z)} |z|^s$ (the absolute value here is the square of the usual one, i.e., the absolute value of the norm to $\RR^\times$), then 
 $$ L(\chi, 0) = 2\cdot (2\pi)^{-s} \Gamma(s + \frac{|m|}{2}).$$ 
\end{itemize}
Notice that the poles of the $L$-function are contained in the poles of the function 
$$\mathfrak G(s):= \begin{cases}
                          \Gamma(s) ,\mbox{ if } F=\RR; \\
                          \Gamma(2s), \mbox{ if } F=\CC.
                         \end{cases}$$
We let $\mathbb H^\PW_{\mathfrak G}(\CC)$ denote the Fr\'echet space of holomorphic multiples of $\mathfrak G$, which are of rapid decay in bounded vertical strips, away from the poles.

We have the following description of the image and residues of Tate zeta integrals:
\begin{proposition}\label{Tateimage}
 The Tate zeta integral $\Phi\mapsto Z(\Phi,\chi,0)$ defines an isomorphism between the space $\mathcal F(F)$ and the space of polynomial multiples, in the non-Archimedean case, and holomorphic multiples, in the Archimedean case, of the function $L(\chi,0)$ on $\widehat{F^\times}_\CC$ which have the following properties: 
 \begin{itemize} 
  \item in the non-Archimedean case, they are supported on a finite number of connected components of $\widehat{F^\times}$;
  \item in the Archimedean case, factoring the character group of $F^\times$ as in \eqref{dualArchimedean}, with $T_0=$the maximal compact subgroup of $F^\times$, they belong to the completed tensor product
  $$ \mathscr C(\widehat{T_0}) \hat\otimes \mathbb H^\PW_{\mathfrak G}(\CC).$$
 \end{itemize}
 
 Moreover, the residue at the trivial character is given by the formula 
\begin{equation}\label{Tateresidue} \Res_{s=0} Z(\Phi, 1, s) = \Phi(0) \AvgVol(F^\times),\end{equation}
where $\AvgVol$ is the average volume defined in \eqref{avgvol}.
\end{proposition}

\begin{remark}\label{remarkPWspace}
 We will adopt the convention, both in the Archimedean and non-Archimedean cases, that the above space of functions on $\widehat{F^\times}_\CC$ will be denoted by $\mathbb H^\PW_{L(\bullet,0)}(\widehat{F^\times}_\CC)$.  This notion of ``Paley--Wiener functions'' (or sections) will recur, in a more general context, later in this paper.
\end{remark}

\begin{proof}
 I only sketch the proof for the image of the Tate integral in the Archimedean case; the rest of the results are found in any reference on Tate's thesis. Consider the multiplication map $T_0\times \RR_{\ge 0} \to F$. The Schwartz space of rapidly decaying smooth functions on $T_0\times \RR_{\ge 0}$ can be written 
 $$\mathcal F(T_0\times \RR_{\ge 0}) = C^\infty(T_0)\hat\otimes \mathcal F(\RR_{\ge 0}),$$
 and Mellin transform identifies the second factor on the right with the Fr\'echet space $\mathbb H^\PW_{\mathfrak G}(\CC)$. (The difference in the definition of $\mathfrak G$ is due to the fact that in the complex case we are using the square of the usual norm.) Thus, 
 $$ \mathcal F(T_0\times \RR_{\ge 0}) = \mathscr C(\widehat{T_0})\hat\otimes \mathbb H^\PW_{\mathfrak G}(\CC).$$
 
 The elements of this space which descend to elements of $\mathcal F(F)$ are those belonging to the closed subspace of holomorphic multiples of $L(\chi, 0)$.
\end{proof}

\subsubsection{Fourier convolutions} \label{sssFourierconv}

Let $T$ be a torus, and $V$ a space of measures, functions or half-densities on $T=T(F)$ (with properties to be specified). For any $s\in \CC$, consider the distribution 
\begin{equation} \label{DS} D_s:=|x|^s \psi(x) d^\times x
\end{equation}
on $F^\times$. Any cocharacter $\check\lambda: \Gm\to T$ induces, by push-forward, a distribution $\check\lambda_*D_s$ on $T$. The equivariant Fourier transform $\mathscr F_{\check\lambda,s}$ is defined as the operator of convolution by $\check\lambda_* D_s$, on the given space $V$. More generally, for any character $\chi$ of $T(F)$, we define $\mathscr F_{\check\lambda,\chi, s}$ as the operator of multiplicative convolution by the measure $\check\lambda_*\left(\chi(x)|x|^s \psi(x) d^\times x\right)$; of course, by definition, $\mathscr F_{\check\lambda,\chi, s} = \mathscr F_{\check\lambda,\chi|\bullet|^s, 0}$.

If $V = \mathcal S(T)$, the convolution is convergent. In the general case, we will typically need to regularize it. For example, if $T=F^\times$ and  $V = \mathcal S(F)$ (considered by restriction as measures on $F^\times\subset F$), with $\check\lambda$ the identity cocharacter, we formally have:
$$ \mathscr F_{\check\lambda,s} f (\xi) = \int_{F^\times} |x|^s \psi(x) f(x^{-1} \xi) d^\times x = |\xi|^{-s} \int_{F^\times} |x|^{s-1} f(x^{-1}) \psi(x\xi) dx.$$
As is well-known, the association $x\mapsto |x|^{s-1} f(x^{-1})$ is a finite measure when 
$\Re(s)>-1$, and makes sense by analytic continuation as a tempered distribution for all but a countable set of values of $s$ (corresponding to the poles of a Tate zeta integral). Thus, the multiplicative Fourier convolution above will be interpreted as the Fourier transform of this distribution (valued, again, in distributions, since we have fixed the self-dual measure $dx$). Whenever we say that a multiplicative Fourier convolution should be interpreted ``in a regularized sense'', we will mean as the Fourier transform of a distribution.

As was explained in \cite[Proposition 2.1]{SaHanoi}, multiplicative Fourier convolution acts on Mellin transforms as follows:
 \begin{equation}\label{FE} \widecheck{(\mathscr F_{\lambda , s} f)}(\chi) = \gamma(\chi,\check\lambda,1-s,\psi) \check f(\chi),\end{equation} 
where $\gamma(\chi,\check\lambda,1-s,\psi)$ is the gamma factor \eqref{gammafactor} of the local functional equation for the character $\chi\circ e^{\check\lambda}$.

The equation \eqref{FE} is literally true for $f\in \mathcal S(T)$. For more general spaces of functions, where $\mathscr F_{ {{{\check\lambda}}} , s} f$ will be defined as the Fourier transform of a distribution, it will require some justification.

\subsection{Non-standard test measures for the Kuznetsov formula} 

\subsubsection{Twisted push-forward} \label{sstwistedpf}
Let $G$ be any of the groups in the sequence
$$ \SL_2 \overset{\hookrightarrow}\from\Gm\times\SL_2 \twoheadrightarrow\GL_2\twoheadrightarrow \PGL_2,$$
where the map $\Gm\to \GL_2$ is the canonical map into the center. Let $N$ denote the subgroup of upper triangular unipotent matrices, $\C$ the quotient $N\backslash G\sslash N$, and $\C^0$ its open subset corresponding to the open Bruhat cell.

Let $A$ be the universal Cartan of $G$, identified with the subgroup of diagonal matrices (or $\Gm \times$ that subgroup, in the second case) by choosing the upper triangular Borel subgroup $B$. Let $w=\begin{pmatrix} & -1 \\ 1 \end{pmatrix}$. We identify $A$ with $\C^0$ via the map $a\mapsto [wa]$.

We identify $N\simeq \Ga$ in the usual way, so that $\psi$ becomes a character of $N$. Let $C^\infty((N,\psi^{-1})\backslash BwB/(N,\psi^{-1}))$ the space of smooth functions on the open Bruhat cell which varies by the character $\psi^{-1}$ under left and right multiplication by $N$. The section $$\C^0\simeq  A \ni a \mapsto wa \in BwB$$ 
allows us to identify, by restriction, 
\begin{equation}C^\infty((N,\psi^{-1})\backslash BwB/(N,\psi^{-1})) \simeq C^\infty(\C^0).\end{equation}

Dual to this restriction map, we have well-defined \emph{twisted push-forward maps} of measures,
\begin{equation}\label{pushfsingle} \mathcal S(G)\twoheadrightarrow \mathcal S(N,\psi\backslash G) \xrightarrow{p_!} \Meas(\C^0).
\end{equation}
The image of the last map will be denoted by $\mathcal S(N,\psi\backslash G/N,\psi)$ --- it is the space of standard test measures for the Kuznetsov formula. The last map can also be identified with a $G^\diag$-invariant map (to be denoted by the same symbol)
\begin{equation}\label{pushftwocopies} \mathcal S(N,\psi\backslash G)\hat\otimes \mathcal S(N,\psi^{-1}\backslash G) \xrightarrow{p_!} \mathcal S(N,\psi\backslash G/N,\psi),
\end{equation}
again dual to the pullback of $(N,\psi^{-1})$-equivariant functions, this time through the map 
$$ (N\backslash G)^2\ni (g_1,g_2)\mapsto [g_1g_2^{-1}]\in \C.$$
It is well-known that both the maps \eqref{pushfsingle}, \eqref{pushftwocopies} can be identified with the corresponding coinvariant quotients:
\begin{equation}\label{Whittakercoinvariants} \mathcal S(N,\psi\backslash G/N,\psi) = \mathcal S(G)_{(N,\psi)^2} = \mathcal S(N,\psi\backslash G)_{(N,\psi)} = \left(\mathcal S(N,\psi\backslash G)\hat\otimes \mathcal S(N,\psi^{-1}\backslash G)\right)_{G^\diag}.
\end{equation}
Indeed, the isomorphisms among the various coinvariant spaces follow from the construction of Schwartz spaces on stacks in \cite[\S 3.4]{SaStacks} (with trivial modifications to incorporate the line bundle defined by the character $\psi$), and the  isomorphism with $\mathcal S(N,\psi\backslash G/N,\psi)$ is equivalent to the density of regular orbital integrals for the Kuznetsov formula, which is well-known --- see references in the proof of Theorem \ref{density} in the next section. 

Since the map $N\backslash G\twoheadrightarrow N\backslash G\sslash N$ is smooth,  the untwisted push-forward map: $\mathcal S(N\backslash G)\to \Meas(N\backslash G\sslash N)$ has image in Schwartz measures; in particular, elements of $\mathcal S(N,\psi\backslash G/N,\psi)$ are bounded by Schwartz measures on $\C$, and hence extend to finite measures on $\C$. In reality, we will never use this fact, but it is convenient to refer to elements of $\mathcal S(N,\psi\backslash G/N,\psi)$ as measures on $\C$.

Let $f$ belong to any of the spaces 
$$\mathcal S(G), \, \, \mathcal S(N,\psi\backslash G), \mbox{ or } \, \mathcal S(N,\psi\backslash G)\hat\otimes \mathcal S(N,\psi^{-1}\backslash G),$$ and write it in the form $\Phi dx$, where $\Phi$ is a (Whittaker, in the second and third cases) Schwartz function and $dx$ is an invariant measure on the corresponding space. For compatible choices of invariant measures, we have the integration formula 
\begin{equation}\label{integration}
 p_! f(a) = \delta(a) O_a (\Phi) da,
\end{equation}
where $\delta$ is the modular character of the Borel subgroup, considered as a function on $A\subset \C$, $O_a$ is the orbital integral
$$ O_a(\Phi) = \int_{N\times N} \Phi(n_1 wa n_2) \psi^{-1}(n_1 n_2) dn_1 dn_2$$ 
in the first case, and similarly in the others, and $da$ is a (multiplicative) Haar measure on $A$. 

For a Schwartz function $\Phi$ on $G$ (or one of the other spaces), we define its twisted push-forward by 
$$ p_! \Phi(a) = O_a(\Phi),$$
and for a Schwartz half-density $\varphi = \Phi d^\frac{1}{2}x$ we define
\begin{equation}\label{pushf-densities} p_! \varphi(a) = \delta^\frac{1}{2}(a) O_a(\Phi) d^\frac{1}{2}a.
\end{equation}
These push-forwards depend on the choice of a Haar measure on $N$, which however we have fixed to be the self-dual measure with respect to our character $\psi$. The data $d^\frac{1}{2}x$, $d^\frac{1}{2}a$ are then proportional to each other by the integration formula \eqref{integration}, and hence these twisted push-forwards are determined by the choice of Haar measure on $N$. The image of the space of Schwartz Whittaker half-densities under the twisted push-forward will be denoted by 
$$\mathcal D(N,\psi\backslash G/N,\psi).$$
Note that these are \emph{densely defined} half-densities on $\C$; more precisely, they are defined on the open subset $\C^0\simeq A$.

Finally, I mention that for $G=\PGL_2$ or $\SL_2$ we will be identifying the space $\C$ with $\Ga$ by choosing the coordinate on $A\subset\C$ which for $\PGL_2$ is the positive root character, and for $\SL_2$ is the positive half-root character. Thus, explicitly, we have identified $\Gm$ as a subset of $\C$, and moreover we have a section obtained from the embedding $A\mapsto wA$, as follows:
\begin{eqnarray}\label{sections}\nonumber\mbox{For $\PGL_2$: }\tilde \xi = w e^{\frac{\check\alpha}{2}}(\xi) = \begin{pmatrix} & -1 \\ \xi \end{pmatrix},\,\, \xi\in \Gm;\\
\mbox{For $\SL_2$: }\tilde\zeta = w e^{\check\alpha}(\zeta) = \begin{pmatrix} & -\zeta^{-1} \\ \zeta \end{pmatrix},\,\, \zeta\in \Gm.\end{eqnarray}

\subsubsection{An invariant formulation} \label{ssinvariant}

It will be necessary to have a more invariant description of the Whittaker model and its Kuznetsov quotient.

Let $V$ be a two-dimensional vector space, and $G=\SL(V)$. We take the action of $G$ on $V$ to be a \emph{right} action. 
 
Let $V^\vee$ denote the dual vector space, and $\tilde V  = \{ (v, v^\vee)\in V\times V^\vee| \left<v,v^\vee\right>=1\}$. Then $\tilde V$ is a torsor over $V^*:=V\smallsetminus\{0\}$ for the group scheme $\mathbb S$ of stabilizers, in $G$, of the points of $V^*$ (and similarly over $V^{\vee *} := V^\vee\smallsetminus\{0\}$).
 
The action of $G$ on this group scheme exhibits it as a constant group scheme, and we can fix a $G$-equivariant isomorphism $\mathbb S\simeq \Ga \times V^*$. We will call such an isomorphism a \emph{Whittaker structure} for $V$. Equivalently, this turns $\tilde V$ into a $G$-equivariant $\Ga$-torsor over $V$.

One way to fix such a structure, is to endow $V$ with an invariant symplectic form $\omega$. This defines an isomorphism $\iota_\omega: V\xrightarrow\sim V^\vee$ by $\left<u,\iota_\omega(v)\right> = \omega(u,v)$, the space $\tilde V$ becomes the space of pairs $(v,w)\in V^*\times V^*$ with $\omega(v, w)=1$, and the $\Ga$-action on $\tilde V$ is given by 
\begin{equation}\label{Gaaction}
x\cdot (v, w) = (v, w - xv).
\end{equation}
 
It is immediate to see that this is a bijection between the sets of
 \begin{itemize}
  \item Whittaker structures on $V^*$, and 
  \item non-zero alternating forms on $V$.
 \end{itemize}

Any two choices of a Whittaker structure are conjugate by a \emph{unique} $G$-automorphism (i.e., scalar automorphism) of $V$; in that sense, the resulting $\Ga$-bundle $\tilde V\to V^*$ is rigid.  

From this point on, we fix a Whittaker structure, and use it to view $\tilde V$ as a subvariety of $V^*\times V^*$, so we have two projection maps $\tilde V\underset{t}{\overset{s}\rightrightarrows} V^*$.
 
The name ``Whittaker structure'' is due to the fact that, 
through the additive character $\psi$ of $F$, a $\Ga$-bundle induces a ($\CC^\times$-bundle and hence a) complex line bundle $\mathcal L_\psi$ on the $F$-points of $V^*$, whose sections are the Whittaker functions for $G$. The Whittaker line bundle $\mathcal L_\psi$ comes with a trivialization of its pullback to $\tilde V$, arising from the canonical isomorphism (obtained from the projection and action maps): $\Ga\times \tilde V\simeq \tilde V\times_V \tilde V$. Explicitly, Whittaker functions are functions on $\tilde V$ which satisfy
$$ \Phi(v,u- xv) = \psi(x) \Phi(v,u).$$

If we choose the symplectic form $\omega$ and coordinates $(x,y)$ on $V$ so that the symplectic form is $\omega = dx \wedge dy$, the distinguished $G$-orbit on $V^*\times V^*$ is that of the ordered pair $((1,0), (0,1))$. The stabilizer $N^-$ of $(1,0)$ is identified with $\Ga$ by 
$$ x\mapsto \begin{pmatrix} 1 &  \\ x & 1 \end{pmatrix},$$
and this identifies its orbit through $(0,1)$ as a $\Ga$-torsor by $x\cdot (y,1) = (-x+y,1)$. If we identify an ordered pair $((a,b),(c,d))$ with the element $\begin{pmatrix} a & b \\ c & d \end{pmatrix} \in \SL_2$, Whittaker functions are functions on $\SL_2$ which satisfy
$$ \Phi \left( \begin{pmatrix} 1 \\ x & 1 \end{pmatrix}  g\right) = \psi^{-1}(x) \Phi(g).$$
The function $\Phi'(g) = \Phi(w^{-1} g)$ is then a Whittaker function which transforms under the character $\psi$ of the upper triangular subgroup $N\simeq \Ga$, as before, and the translation from $\Phi'$ back to the abstract description of the function $\Phi$ is that $\Phi'\begin{pmatrix} a & b \\ c & d \end{pmatrix} = \Phi((c,d),(-a,-b))$. 

The section of the map $G\to \C$ (over the open $\C^0$), which allowed us to define a twisted push-forward in the previous subsection, now admits the following description; more precisely, let us describe the section giving rise to \eqref{pushftwocopies}, which now can be written as a map 
\begin{equation}\label{pushftwocopies-canonical} \mathcal S(V^*,\mathcal L_\psi) \otimes \mathcal S(V^*,\mathcal L_{\psi^{-1}}) \to \mathcal S(N,\psi\backslash G/N,\psi).
\end{equation}
Since elements of $\mathcal S(V^*,\mathcal L_{\psi^{\pm 1}})$ are scalar-valued functions on $\tilde V$, 
we consider the maps 
$$ \tilde V \times \tilde V \to V^*\times V^*\to \C$$
and will describe a distinguished $G$-orbit on $\tilde V\times \tilde V$ that can be used to trivialize push-forwards.

First of all, we have an isomorphism, which can be taken as the definition 
$$ \C = V\times V \sslash G^\diag.$$
Moreover, the map $(v_1, v_2)\mapsto \omega(v_1, v_2)$ identifies $\C \simeq \Ga$. (This is compatible with the map $\Gm \to \C$ obtained from \eqref{sections}).

Now, over $\C^0=\C\smallsetminus\{0\}$ we have a distinguished $G$-stable subset $\widetilde{\mathbb V}\subset \tilde V\times \tilde V$, consisting of those pairs $(v_1, w_1)$ and $(v_2, w_2)$ such that $w_1$ and $v_2$ are colinear, and $v_1$ and $w_2$ are colinear. The map $\widetilde{\mathbb V} \to V^*\times V^*$, $(v_1,w_1,v_2,w_2)\mapsto (v_1,v_2)$, is an isomorphism, and it allows us to pull back $\mathcal L_\psi\boxtimes \mathcal L_{\psi^{-1}}$-valued measures on $V^*\times V^*$ to actual measures on $\widetilde{\mathbb V}$, and push them forward to $\C$. This gives rise to the twisted push-forward  \eqref{pushftwocopies-canonical}.

\subsubsection{Generating series for unramified $L$-values}

Assume here that $F$ is a non-Archimedean field, with ring of integers $\mathfrak o$ and residual degree $q$. Let $K=G(\mathfrak o)$, and $\mathcal H(G,K)$ the Hecke algebra of $K$-biinvariant, compactly supported measures on $G$.

Unramified characters of the universal Cartan $A$ of $G$ are in canonical bijection with points of the complex dual torus $\check A$; we will denote this bijection as $\chi \leftrightarrow \check \chi$. It is characterized by $\chi({\check\lambda}(\varpi)) = {\check\lambda}(\check\chi)$ for any coweight $\check\lambda$ into $A$, and $\varpi$ a uniformizer in $\mathfrak o$.
The Satake isomorphism 
$$ \mathcal H(G,K)\ni h \mapsto \check h \in \CC[\check G]^{\check G} = \CC[\check A]^W$$
is characterized by the property that, for the principal series representation $\pi_\chi$ obtained by normalized induction from the character $\chi$ of a Borel subgroup $B$ (through the quotient $B\twoheadrightarrow A$), we have 
$\pi_{\chi}(h) v_{K,\chi} = \check h(\check\chi) v_{K,\chi}$. 

If $X$ is a smooth, quasi-affine $G$-space over $\mathfrak o$, we call the characteristic function $1_{X(\mathfrak o)}$ of $X(\mathfrak o)$ the \emph{basic function} of $X$. We would like to focus on the Whittaker model, so we will use $X$ to denote the ``space'' $X = (N,\psi)\backslash G$,  that is, the space $N\backslash G$, but endowed with the complex line bundle defined by the character $\psi$. In that case, the basic function (still to be denoted by $1_{X(\mathfrak o)}$ will be the left-$(N,\psi)$-equivariant function on $G$ which is supported on $NK$ and equal to $1$ on $K$. 

Let $r: \check G\to \GL(V)$ be an algebraic representation, and $s\in \CC$. The \emph{local unramified $L$-value} 
$L(r,s)$
is the element 
$$L(r,s):= \frac{1}{\det(I-q^{-s}r)} \in  \CC(\check G)^{\check G}.$$ 
It can be written as a formal Taylor series in $q^{-s}$:
$$L(r,s) = \sum_{n=0}^\infty q^{-ns} \tr(\Sym^n r).$$
If $r$ is reducible, $r = \bigoplus_{i=1}^m r_i$, we can allow $s$ to denote an $m$-tuple $(s_i)_{i=1}^m$ of complex numbers, and define $L(r,s)=\prod_{i=1}^m L(r_i,s_i)$. I proceed with a single $s$, and the adjustments for the general case are obvious.

The \emph{generating series of the local $L$-value} $L(r,s)$ on $X$ is the Whittaker function
\begin{equation}\label{Lseries}\Phi_{L(r,s)} = \sum_{n=0}^\infty q^{-ns} h_{\Sym^n r} \star 1_{X(\mathfrak o)},
\end{equation}
whenever this series converges, where $h_{\Sym^n r}\in \mathcal H(G,K)$ is the element with 
$$ \check h_{\Sym^n r} = \tr(\Sym^n r).$$

It appears more appropriate ensure that there is a character $\partial: G\to \Gm$, whose dual $\partial^*: \Gm\to \mathcal Z(\check G)$, followed by $r$, gives rise to the canonical cocharacter to the center of $\GL(V)$. For example, when $G=\PGL_2$, the standard representation of $\check G= \SL_2$ should be extended to the group $\GL_2$, which corresponds to replacing $G$ by $\GL_2$. Similarly, for the symmetric-square representation of $\SL_2$, which factors through the adjoint representation of $\PGL_2$, one should take $\Gm\times\PGL_2$ to be the dual group, so that $G=\Gm \times \SL_2$. In that case, the complex parameter $s$ of the $L$-value becomes a red herring, and can be fixed to be $0$ (or $\frac{1}{2}$), since we have an equality of rational functions on $\check G$:
$$ L(r, s)(x) = L(r, 0) (x\partial^*(q^{-s})).$$
For Langlands $L$-functions of representations, this is the statement that $L(\pi, r, s) = L(\pi\otimes |\partial|^s, r, 0)$.

Moreover, in that case, the series \eqref{Lseries} makes sense for every $s$, since the $n$-th summand is compactly supported on the subset with $\val (\partial )= n$.

For our present purposes, however, we would like to allow ourselves to take \eqref{Lseries} in $\SL_2$ or $\PGL_2$, which corresponds to integrating it over the fibers of the map $\Gm\times\SL_2\to\SL_2$ or $\GL_2\to\PGL_2$. This integral converges when $\Re(s)\gg 0$; in the examples of interest, it will follow from the results that we prove that it admits meromorphic continuation to all $s$, rational in the parameter $q^{-s}$. 

Finally, we would like to work with measures instead of functions. For that purpose, we choose the invariant measure $dx$ on $N\backslash G$ such that $\Vol(N\backslash G(\mathfrak o))=1$. The product $\Phi_{L(r,s)} dx$ will be called the ``generating measure of $L(r,s)$''. 

\subsubsection{Non-standard test measures} \label{ssnonstandard}

Now let $G$ be one of the four groups above, and consider the twisted push-forward $p_!:\mathcal S(N,\psi\backslash G)\to \mathcal S(N,\psi\backslash G/N,\psi)$ of \S \ref{sstwistedpf}. It is easy to see that, restricted to $K$-invariants, where $K = G(\mathfrak o)$, it is locally finite, in the sense that for any $c\in \C = N\backslash G\sslash N$ there is only a finite number of $K$-orbits on $N\backslash G$ such that the elements of $\mathcal S(N,\psi\backslash G)^K$ which are supported on those $K$-orbits have non-zero push-forward in a neighborhood of $c$, cf.\ \cite[\S 6.3]{SaBE1}.\footnote{There is a typo on the last line of \cite[(6.2)]{SaBE1}: $m=1$ should read $m=0$. The property of local finiteness of the push-forward explained here is not special to $K$ --- it holds for invariants of any compact open subgroup.} Thus, we can extend the twisted push-forward to \emph{any} $K$-invariant Whittaker measure. In particular, the twisted push-forward of $\Phi_{L(r,s)} dx$ is well-defined whenever $\Phi_{L(r,s)}$ is, and will be denoted by 
$$ f_{L(r,s)} \in \Meas(N,\psi\backslash G/N,\psi).$$ 
Explicitly, using the integration formula \eqref{integration}, we have 
$$ f_{L(r,s)}(a) = (1-q^{-2})^{-1} \left(\int_N \Phi_{L(r,s)}(wan) \psi^{-1}(n) dn\right) \cdot \delta(a) da,$$
when the measure on $N(\mathfrak o)$ is taken to be $1$, and $da(A(\mathfrak o))=(1-q^{-1})$; indeed, the Haar measure on $G$ which gives total mass $1$ to $G(\mathfrak o)$ restricts on the open Bruhat cell $Nw A N$, in coordinates $n_1wan_2$, to the measure $(1-q^{-2})^{-1} dn_1 \delta(a) da dn_2$, and $da$ is the Haar measure with $da(A(\mathfrak o))=1-q^{-1}$.

Now we specialize to $G=\PGL_2$ and $r=$copies of the standard representation of $\check G=\SL_2$, or $G=\SL_2$ and $r=$copies of the adjoint representation of $\check G=\PGL_2$. So, we write $r = \bigoplus_i r_i$ with all $r_i$'s equal to the standard representation (for $\PGL_2$) or the adjoint representation (for $\SL_2$), and correspondingly $s=(s_i)_i$ denotes a collection of complex numbers, one for every $r_i$. Then $f_{L(r,s)} $ lives in a natural space of test measures 
$$ \mathcal S^-_{L(r,s)} (N,\psi\backslash G/N,\psi) = \mathcal S^-_{\prod_i L(r_i,s_i)} (N,\psi\backslash G/N,\psi)$$
for the Kuznetsov formula, described as follows; here, we allow again the field $F$ to be Archimedean:

We let $ \mathcal S^-_{L(r,s)} (N,\psi\backslash G/N,\psi)$ be the space of measures on $\mathfrak C$ which on any compact set coincide with elements of $\mathcal S(N,\psi\backslash G/N,\psi)$, while in a neighborhood of infinity, in the coordinates of \eqref{sections}, when all of the $s_i$'s are \emph{distinct}, they are of the form  
 \begin{equation}\label{expansionSL2}
  \sum_i C_i(\zeta^{-1}) |\zeta|^{1-s_i} d^\times \zeta,
 \end{equation}
in the case of $G=\SL_2$, and
 \begin{equation}\label{expansionPGL2}
  \sum_i C_i(\xi^{-1}) |\xi|^{\frac{1}{2}-s_i} d^\times \xi,
 \end{equation}
in the case of $G=\PGL_2$, where the $C_i$'s are smooth functions in a neighborhood of zero. When two or more of the $s_i$'s coincide with some complex number $s_0$, the corresponding summands at infinty will be replaced by $(C_1(\zeta^{-1}) + C_2(\zeta^{-1})\log|\zeta| + C_3(\zeta^{-1}) \log^2|\zeta| +\dots) |\zeta|^{1-s_0} d^\times \zeta$ (as many summands as occurences of the exponent $s_0$), and similarly for $\PGL_2$. 

More generally, to accommodate possible push-forwards from the groups $\GL_2$ and $\Gm\times \SL_2$ to $\PGL_2$, resp.\ $\SL_2$, we can define, for any collection of characters $\chi=(\chi_i)_i$ of $\Gm$, a space 
$$ \mathcal S^-_{L(r,\chi)} (N,\psi\backslash G/N,\psi)$$
whose elements 
coincide with elements of $\mathcal S(N,\psi\backslash G/N,\psi)$ away from infinity, and are of the form 
 \begin{equation}\label{expansionSL2-chi}
  \sum_i C_i(\zeta^{-1}) |\zeta| \cdot \chi_i^{-1}(\zeta) d^\times \zeta,
 \end{equation}
in the case of $G=\SL_2$, and
 \begin{equation}\label{expansionPGL2-chi}
  \sum_i C_i(\xi^{-1}) |\xi|^{\frac{1}{2}} \cdot \chi_i^{-1}(\xi) d^\times \xi,
 \end{equation}
with the analogous modifications when some of the $\chi_i$'s coincide.
 
The following was stated as \cite[Proposition 3.2]{SaHanoi}, and can be proven as in \cite[Lemma 5.3]{SaBE1}:

\begin{proposition}
If $F$ is non-Archimedean and $\Phi_{L(r,s)} dx$ is well-defined (for example, when $\Re(s)\gg 0$), its image $ f_{L(r,s)} $ is contained in $\mathcal S^-_{L(r,s)} (N,\psi\backslash G/N,\psi)$.   
\end{proposition}

We will call $f_{L(r,s)}$ the \emph{basic vector} of $\mathcal S^-_{L(r,s)} (N,\psi\backslash G/N,\psi)$.

\begin{example}
 Let $G=\SL_2$, considered as a reductive group over $\mathfrak o$, $K=G(\mathfrak o)$, $dx$ the invariant measure on $N\backslash G$ with $dx(N\backslash G(\mathfrak o))=1$, and take $r=\Ad$. Let $\Lambda^+= - \mathbb N \check\alpha$ be the set of anti-dominant elements in the coweight lattice of the universal Cartan $A$, and, for every $\check\lambda\in \Lambda^+$, let $e^{\check\lambda}$ stand for the element in $\mathcal S(N,\psi\backslash G)^K$ which is equal to the product of $dx$ by the Whittaker function which is 
 zero of the coset $N e^{-\check\lambda}(\varphi) K$, and equal to $q^{\left<\rho,\check\lambda\right>}$ on $e^{-\check\lambda}(\varphi)$. We use rational functions of the form $\frac{1}{1-q^{-s}e^{\check\lambda}}$ to denote the series $\sum_{i\ge 0} q^{-is} e^{i\check\lambda}$, but caution that $\frac{1}{1-q^{-s}e^{\check\lambda}}$ and $\frac{-q^s e^{-\check\lambda}}{1-q^s e^{-\check\lambda}}$ denote different series.
 
 Then, \cite[Theorem 7.7]{SaSatake} states that, for $\Re(s)\gg 0$, the element $\Phi_{L(r,s)} dx$ coincides with the restriction to $\Lambda^+$ of the ``series''
 \begin{equation}\label{basicadjoint} \frac{1-e^{\check\alpha}}{(1-q^{-s} e^{\check\alpha})(1-q^{-s} )(1-q^{-s} e^{-\check\alpha})}.
 \end{equation}
 Moreover, \cite[(6.2)]{SaBE1} computes the orbital integrals of the Whittaker function $\frac{q^{-\left<\rho,\check\lambda\right>}e^{\check\lambda}}{dx}$; we multiply by 
 $$q^{\left<\rho,\check\lambda\right>}\cdot (1-q^{-2})^{-1} \delta(a) da = q^{\left<\rho,\check\lambda\right>}\cdot (1-q^{-2})^{-1} |\zeta|^2 d^\times \zeta,$$ 
 to deduce that the twisted push-forward of $e^{\check\lambda}$, when $\check\lambda \ne 0$, is equal to the measure 
 $$q^{\left<\rho,\check\lambda\right>}\cdot (1-q^{-2})^{-1} |\zeta|^2 d^\times \zeta \cdot \left(1_{|\zeta|=q^{-\left<\rho,\check\lambda\right>}}- 1_{|\zeta|=q^{-\left<\rho,\check\lambda\right>-1}}\right)=$$
 $$= (1-q^{-2})^{-1} |\zeta| d^\times \zeta \cdot \left(1_{|\zeta|=q^{-\left<\rho,\check\lambda\right>}}- q^{-1} \cdot 1_{|\zeta|=q^{-\left<\rho,\check\lambda\right>-1}}\right).$$ 
 We can let $\epsilon^{\check\lambda}$ denote the restriction of the measure $|\zeta| d^\times\zeta$ to the set $|\zeta|=q^{-\left<\rho,\check\lambda\right>}$, then the twisted push-forward of  $e^{\check\lambda}$ reads 
 $$ (1-q^{-2})^{-1} (\epsilon^{\check\lambda} - q^{-1} \epsilon^{\check\lambda + \check\alpha}),$$
 and the twisted push-forward of the series \eqref{basicadjoint}, restricted to the set $|\zeta|>1$, will be equal to the ``series''
 \begin{equation}\label{basicadjointpushf}  (1-q^{-2})^{-1} \frac{(1-\epsilon^{\check\alpha})(1-q^{-1}\epsilon^{\check\alpha})}{(1-q^{-s} \epsilon^{\check\alpha})(1-q^{-s} )(1-q^{-s} \epsilon^{-\check\alpha})}
 \end{equation}
 restricted to negative multiples of $\check\alpha$.
 
 The series $\frac{1}{1-q^{-s} \epsilon^{-\check\alpha}}$, just by itself, is equal to the measure $|\zeta|^{1-s} d^\times\zeta$. I leave it to the reader to check that the asymptotics of \eqref{basicadjointpushf} are now obtained by setting $\epsilon^{\check\alpha}=q^{-s}$ in the remaining factors; we find that $f_{L(\Ad, s)}$ is equal to 
 \begin{equation}\label{asymptoticsbasicadjoint}
  \frac{1-q^{-1-s}}{(1-q^{-2})(1-q^{-2s})} |\zeta|^{1-s} d^\times\zeta
 \end{equation}
for large $|\zeta|$.
 
\end{example}

We will also work with half-densities, thus we similarly define a space $\mathcal D^-_{L(r,s)} (N,\psi\backslash G/N,\psi)$ which contains $\mathcal D(N,\psi\backslash G/N,\psi)$. In this case, the analogous to \eqref{expansionSL2}, \eqref{expansionPGL2} expansions at infinity are:
\begin{equation}\label{expansionSL2-densities}
  \sum_i C_i(\zeta^{-1}) |\zeta|^{-s_i} (d^\times \zeta)^\frac{1}{2},
 \end{equation}
in the case of $G=\SL_2$, and
 \begin{equation}\label{expansionPGL2-densities}
  \sum_i C_i(\xi^{-1}) |\xi|^{-s_i} (d^\times \xi)^\frac{1}{2},
 \end{equation}
by comparing \eqref{integration} and \eqref{pushf-densities}. The \emph{basic vector} of these spaces, in the non-Archimedean case, will be the quotient of the measure $f_{L(r,s)} $ by $\delta^\frac{1}{2}(a) d^\frac{1}{2}a$, that is, by $|\zeta| (d^\times \zeta)^\frac{1}{2}$ in the case of $\SL_2$ and $|\xi|^\frac{1}{2}(d^\times \xi)^\frac{1}{2}$ in the case of $\PGL_2$. It will again be denoted by $f_{L(r,s)} $, when it is clear that we are referring to half-densities, instead of measures.

Finally, we mention that the regular action of the unramified Hecke algebra $\mathcal H(G,K)$ on $\mathcal S(N,\psi\backslash G)^K$ descends to its image in $\mathcal S(N,\psi\backslash G/N,\psi)$; indeed, the action of $\mathcal H(G,K)$ coincides with the action of the unramified component of the Bernstein center, which descends to the coinvariant space $\mathcal S(N,\psi\backslash G)_{(N,\psi)} = \mathcal S(N,\psi\backslash G/N,\psi)$. Thus, we will feel free to write $h\cdot f$ for $h\in \mathcal H(G,K)$ and $f\in \mathcal S(N,\psi\backslash G/N,\psi)$ in the image of $\mathcal S(N,\psi\backslash G)^K$  (or a series of such elements).

\section{Scattering theory} \label{sec:scattering}

A main theme of the present paper is the comparison between transfer operators involving such a variety $X$, and transfer operators involving its \emph{boundary degeneration} or \emph{asymptotic cone} $X_\emptyset$, as horospherical $G$-space which is, roughly, responsible for the continuous spectrum of $X$.

\subsection{Asymptotic cone}

The asymptotic cone can be defined using coordinate rings: Suppose that $X$ is quasi-affine, and $k[X]= \bigoplus_{\lambda} V_\lambda$ \emph{as a $G$-module}, a multiplicity-free direct sum of irreducible submodules $V_\lambda$, with highest weight $\lambda$ varying over some submonoid $ \Lambda_X^{++}$ of the character group of the universal Cartan $A$ of $G$. Notice that, fixing a Borel subgroup, for two highest-weight vectors $v_\lambda \in V_\lambda$, $v_\mu\in V_\mu$ we have 
\begin{equation}\label{highestweight}v_\lambda\cdot v_\mu = v_{\lambda+\mu}\end{equation} for some highest weigt vector $v_{\lambda+\mu}\in V_{\lambda+\mu}$, but in general $V_\lambda \cdot V_\mu\nsubset V_{\lambda+\mu}$. We define $k[X_\emptyset^a] := \bigoplus_{\lambda} V_\lambda$ \emph{as an algebra}, where the algebra structure is defined by projecting the product of $V_\lambda$ and $V_\mu$ to the direct summand $V_{\lambda+\mu}$, and take $X_\emptyset$ to be the open $G$-orbit in $X_\emptyset^a$. For the Whittaker case, we define the boundary degeneration by retaining the same space $N\backslash G$, but making the character on $N$ trivial.

Here is a list of the spaces $X$ that we are using in this paper, and the isomorphism classes of their asymptotic cones:
\begin{equation}\label{tableX} \begin{array}{|c|c|c|}
  \hline G & X & X_\emptyset \\
  \hline
  \PGL_2, \SL_2, \GL_2 \mbox{ or } \Gm\times \SL_2 & (N,\psi)\backslash G & N\backslash G \\
  \PGL_2 & \Gm\backslash \PGL_2 & N\backslash \PGL_2 \\
  \SL_2^2/\{\pm 1\}^\diag\simeq \SO_4 & \SL_2 & T^\diag(N\times N^-)\backslash \SL_2^2\\
  \hline
 \end{array}
\end{equation}

In the last case, $N$ and $N^-$ are the unipotent radicals of two opposite Borel subgroups of $\SL_2$, and $T$ denotes their intersection. In that case, $X_\emptyset$ can be identified with the variety of $2\times 2$-matrices of rank one.  

In every case, the asymptotic cone $X_\emptyset$ has an action of a torus $A_X$ of $G$-automorphisms; the character group of $A_X$ is generated by the monoid $\Lambda_X^{++}$ of weights appearing in the highest weight decomposition above, and the action is equivalent to the grading of the coordinate ring. Notice, however, that \eqref{highestweight} translates to a \emph{canonical} isomorphism 
\begin{equation}\label{samehoro} X\sslash N_G = X_\emptyset \sslash N_G,
\end{equation}
where $N_G$ is the unipotent radical of a Borel subgroup of $G$. This identifies general $N_G$-orbits on $X$ and $X_\emptyset$, and rigidifies $X_\emptyset$, in the sense that the action of a non-trivial element of $A_X$ would not preserve this isomorphism.  

In the examples above, for the Whittaker model and the variety $\Gm\backslash \PGL_2$, we have $A_X = A_G$, the Cartan of $G$, while for $X=\SL_2$ we have $A_X=$ the Cartan $A$ of $\SL_2$. There is a canonical finite reflection group $W_X$ acting on $A_X$, the \emph{little Weyl group} of $X$; for the examples above, this Weyl group is isomorphic to $\Z/2$.  

By definition, the action of $A_X$ on $X_\emptyset$ is compatible with its action on $X_\emptyset\sslash N_G = X\sslash N_G$. Unfortunately, this creates some confusion for a variety such as $X=N_G\backslash G$, because $A_X = A_G$ in this case is also identified with the quotient $B_G/N_G$, but \emph{this is not the way it acts on $X$}; indeed, it would be preferable to represent $X$ as $N_G^-\backslash G$, where $N_G^-$ is \emph{opposite to the Borel subgroup used to define the Cartan $A_G$}; equivalently, if $x\in X$ is a point with stabilizer $N_G$, \emph{the quotient $A_G=B_G/N_G$ acts on $X$ by twisting the natural action by $w_0$, the longest element of the Weyl group}. To avoid excessive notation, we will usually be denoting a maximal unipotent stabilizer by $N_G$ or $N$, not by $N_G^-$, \textbf{but the reader should keep this convention about the $A_X$-action in mind}, to avoid confusion. 

Moreover, for a space of the form $X=S\backslash G$, where $N_G\subset S \subset \ker(e^{2\rho}) \subset G$, where ${2\rho}$ is the sum of positive roots, so that $X$ admits a $G$-invariant measure $dx$, we \emph{define a normalized action of $A_X = (B_G/S)^{w_0}$ on functions and measures on $G$}, as follows:
\begin{equation}\label{action-normalized-functions}
 (a\cdot \Phi)(x) = \delta^{\frac{1}{2}}(a) \Phi( a\cdot x)
\end{equation}
on functions, and 
\begin{equation}\label{action-normalized-measures}
 (a\cdot \mu)(x) = \delta^{-\frac{1}{2}}(a) \mu( a\cdot x)
\end{equation}
on measures, where $\delta = |e^{2\rho}|$ is the modular character of $B_G$. This action is unitary on the $L^2$-spaces of functions or measures. On half-densities, no normalization is needed.

\subsection{Scattering operators}

From now on until the end of this section, let $F$ be non-Archimedean. I describe some of the results of \cite{DHS}. For each of the varieties $X$ of Table \eqref{tableX}, there is a space $\mathcal S^+(X_\emptyset)$ of smooth measures on $X_\emptyset$, with an $A_X$-semilinear action of $W_X$ by $G$-automorphisms, the \emph{scattering morphisms} 
$$\mathfrak S_w: \mathcal S^+(X_\emptyset) \xrightarrow\sim \mathcal S^+(X_\emptyset),$$
such that $\mathcal S^+(X_\emptyset)$ is generated by $\mathcal S(X_\emptyset)$ under these scattering morphisms. The support of elements of $\mathcal S^+(X_\emptyset)$ has compact closure in the affine variety $X_\emptyset^a$ that we saw above. Moreover, there is a 
canonical ``asymptotics'' morphism
\begin{equation}\label{asymptotics} e_\emptyset^*: \mathcal S(X)\to \mathcal S^+(X_\emptyset),
\end{equation}
defined in \cite[Section 5]{SV} which, according to the Paley--Wiener theorem \cite[Theorem 1.8]{DHS} has image precisely in the subspace of $W_X$-invariants under the scattering morphisms.
 
I will not repeat here what makes the asymptotics morphism canonical; roughly speaking, it is the only morphism such that a measure $\varphi$ is ``equal'' to $e_\emptyset^* \varphi$ ``close to infinity'', see \cite[Section 5]{SV} for details. The scattering operators are characterized by the above properties, but we need a more explicit description of them, in order to compute them. This description is given by \cite[(9.4) and Proposition 10.18]{DHS}, and we repeat it here; unfortunately, the theoretical description is quite involved; the reader may want to skip directly to our computation of scattering operators in the three examples of Table \eqref{tableX}, which is performed in the following subsections and is quite straightforward, and return to the definitions as needed.

We assume, for simplicity, that $X$ admits a $G$-invariant measure; by \cite[\S 4.2]{SV}, any such measure induces a $G$-invariant measure on $X_\emptyset$. We can work with functions instead of measures: any element of $\mathcal S^+(X_\emptyset)$ can be written as the product of a $G$-invariant measure $dx$ by a function in a space $\mathcal F^+(X_\emptyset)$, and once the scattering operators have been defined for functions, they are defined for measures by multiplying by $dx$. The normalizations \eqref{action-normalized-functions}, \eqref{action-normalized-measures} ensure that multiplication by $dx$ is $A_X\times G$-equivariant.

An element $\Phi\in \mathcal F^+(X_\emptyset)$ can be reconstructed from its Mellin transform 
\begin{equation}\label{Mellin-Splus} \check\Phi(\chi)(x) = \int_{A_X} a\cdot \Phi(x) \chi^{-1}(a) da \in C^\infty(A_X,\chi\backslash X_\emptyset)
\end{equation}
Here, $C^\infty(A_X,\chi\backslash X_\emptyset)$ means that the function is $\chi$-equivariant with respect to the \emph{normalized} action of $A_X$. If we choose a base point on $X_\emptyset$, with stabilizer contained in a Borel subgroup $B$, the space $C^\infty(A_X,\chi\backslash X_\emptyset)$ becomes equal to the normalized induced representation $I_B({^{w_0}\chi})$, where $w_0$ is the longest element of the Weyl group. (Recall the conventions about the $A_X$-action on $X_\emptyset$, described above.) This Mellin transform is convergent once $\chi^{-1}$ vanishes fast enough on the complement of $X_\emptyset$ in $X_\emptyset^a$ (that is, it vanishes fast enough on the boundary of an $A_X$-orbit),  extends rationally to the variety of all complex characters of $A_X$, and choosing such a character $\omega$ that vanishes fast enough on the boundary of an $A_X$-orbit, the \emph{inverse Mellin transform} is:
$$\Phi(x) = \int_{\omega^{-1} \widehat{A_X}} \check\Phi(\chi)(x) d\chi.$$
The Haar measure $da$ on $A_X$ chosen to define the Mellin transform is not important here, but one has to choose the dual Haar measure $d\chi$ on its unitary dual $\widehat{A_X}$.

The scattering operator $\mathfrak S_w$ can be accordingly decomposed:
\begin{equation}\label{scatteringMellin} \mathfrak S_w \Phi = \int_{\omega^{-1} \widehat{A_X}} \mathscr S_{w,\chi}\check\Phi({^{w^{-1}}\chi})(x) d\chi,
\end{equation}
where\footnote{There is a slight difference here from the notation of \cite{DHS}: scattering operators are indexed by the character \emph{in their image}, not in their source. This is to ensure compatibility with other morphisms, like the $\mathfrak N_\chi$'s below, which have only a character in appearing their image, not in their source.}
\begin{equation}\label{fiberwisescattering}\mathscr S_{w,\chi}: C^\infty(A_X,{^{w^{-1}}\chi}\backslash X_\emptyset) \to C^\infty(A_X,\chi\backslash X_\emptyset)\end{equation}
are the \emph{fiberwise scattering operators}, varying rationally in $\chi$, that are characterized by the commutativity of the following diagram:

\begin{equation}\label{fiberscatteringdiagram} \xymatrix{ && C^\infty(A_X,{^{w^{-1}}\chi}\backslash X_\emptyset^h) \ar[rr]^{\mathfrak M_{{^{w^{-1}}\chi}}^{-1}}  && C^\infty(A_X,{^{w^{-1}}\chi}\backslash X_\emptyset) \ar[dd]^{\mathscr S_{w,\chi}}  
\\ \mathcal F(X) \ar[urr]^{\mathfrak N_{^{w^{-1}}\chi}}\ar[drr]_{\mathfrak N_{\chi}}&&&  \\ 
&& C^\infty(A_X,{\chi}\backslash X_\emptyset^h) \ar[rr]^{\mathfrak M_{\chi}^{-1}}  && C^\infty(A_X,{\chi}\backslash X_\emptyset).}
\end{equation}

The notation here is as follows: The space $X_\emptyset^h$ is the space of \emph{generic horocycles} on $X$, or on $X_\emptyset$. It classifies pairs $(B,Y)$, where $B$ is a Borel subgroup of $G$, with unipotent radical $N$, and $Y$ is an $N$-orbit in the open $B$-orbit on $X$, or on $X_\emptyset$; by \eqref{samehoro}, $X$ and $X_\emptyset$ give canonically isomorphic spaces by this construction. If $B$ and $B^-$ are two opposite Borel subgroups of $G$ with $B\cap B^-\simeq T$, and we represent $X_\emptyset$ as $SN^-\backslash G$, where $S\subset T$, then $X_\emptyset^h\simeq SN\backslash G$. Although we will not stick with it, it is very useful here to represent the unipotent radical of the stabilizer of a point on $X_\emptyset$ by $N^-$, and the unipotent radical of the stabilizer of a generic horocycle through that point by $N$. The action of $A_X$ on $X_\emptyset^h$ is induced by its action on $X\sslash N$, as suggested by this notation: denoting by $SN$ the stabilizer of a point on $X_\emptyset^h$, the universal Cartan acts on that point via its \emph{defining} identification with $B/N$, unlike the case of $X_\emptyset$. The action of $A_X$ on functions and measures on this space is again defined to be unitary, i.e.,  as in \eqref{action-normalized-functions}, \eqref{action-normalized-measures} but with $\delta$ replaced by $\delta^{-1}$.

The operator $\mathfrak M_{\chi}$, which can be thought of as the ``standard intertwining operator'', is the operator which, in a region of convergence, takes a function in $C^\infty(A_X,{\chi}\backslash X_\emptyset)$ and integrates it over generic horocycles. \emph{Because there is no canonical measure on those horocycles, this operator depends on a choice of such measures}, and more canonically has image in a certain line bundle over $X_\emptyset^h$ (the line bundle dual to the line bundle whose fiber over a horocycle is the set of invariant measures on it --- see \cite[\S 15.2]{SV}). However, in the cases of Table \eqref{tableX} that we are interested in in this paper, such a choice can be made $G$-equivariantly, and it will not matter for the commutativity of the diagram --- the important point here being that horocycles in $X_\emptyset$ and $X$ are identified by \eqref{samehoro}, and the choices of Haar measures must be made compatibly.

The operator $\mathfrak N_{\chi}$ is, similarly, the integral over the horocycle on $X$, followed by an averaging over horocycles in the same $B$-orbit, against the character $\chi^{-1}$ of $A_X$; that is, for a horocycle $Y$, considered both as a point in $X_\emptyset^h$ and as a subspace of $X$, 
$$ \mathfrak N_\chi \Phi (Y) = \int_{A_X} \left(\int_{aY} \Phi(y) dy \right) \chi^{-1}\delta^{-\frac{1}{2}}(a) da.$$

The measure used on $A_X$ is here the same as in the definition of Mellin transform on $X_\emptyset$, so that $\mathfrak M_\chi$ composed with Mellin transform is equal to $\mathfrak N_\chi$ when $X=X_\emptyset$.

We will now calculate the fiberwise scattering operators $\mathscr S_w$ of \eqref{fiberwisescattering} for the non-trivial element $w$ of $W_X\simeq \Z/2$, for the cases of Table \eqref{tableX}.\footnote{For the first line of Table \eqref{tableX}, we assume that $G$ is split of semisimple rank one; the formula for the general split case, which we will not need, can be deduced from this, and the fact that scattering operators compose as in $W_X$, i.e., define an action of $W_X$.} The final result will have the form
$$ \mathscr S_{w,\chi} = \gamma(\chi) \mathfrak R_\chi,$$
where $\mathfrak R_\chi$ is as standard intertwining operator between principal series (essentially, the same as $\mathfrak M_\chi$ after some non-canonical identifications of the spaces involved), and $\gamma(\chi)$ a constant depending on $\chi$ (and expressed in terms of abelian gamma factors of the local functional equation of Tate integrals).

The calculation is quite elementary, on one hand; on the other, it is quite fine to normalize operators such as ``the'' standard intertwining operators between principal series. The constructions needed to formulate a precise result are essentially the constructions needed to prove it; therefore, breaking with the principles of good mathematical exposition, I will formulate the result in the end; the reader can jump ahead to Theorem \ref{thmscattering} to read it.

The calculations that follow hold both over non-Archimedean and over Archimedean fields; thus, despite the fact that the results on asymptotics do not hold as stated in the Archimedean case, we take diagram \eqref{fiberscatteringdiagram} as the defining diagram for the fiberwise scattering operators, and work over an arbitrary local field.

\subsection{The Whittaker case}\label{scatWhittaker}

 In order to compute the scattering maps in the Whittaker case, we will adopt the abstract point of view on the Whittaker model, that was introduced in \S \ref{ssinvariant}. Hence, we take $G=\SL(V)$, where $V$ is a two-dimensional \emph{symplectic} vector space. In particular, it is endowed with a Whittaker structure, and Whittaker functions are functions on the $\Ga$-torsor $\tilde V =\{ (v, v^\vee)\in V \times V^\vee| \left<v,v^\vee\right>=1\}$ over $V^* = V\smallsetminus\{0\}$, which vary by the character $\psi$ of $\Ga$. 
 
 We identify the dual $V^\vee$ with $V$ through the isomorphism $\iota_\omega: V\xrightarrow\sim V^\vee$ by $\left<u,\iota_\omega(v)\right> = \omega(u,v)$.
 Notice that $V^{\vee*}$, the complement of zero in $V^\vee$, can be identified with the variety $V^h$ of ``generic horocycles'' on $V$, that is, affine lines which do not contain the origin; the correspondence sends a functional $v^\vee$ to the affine line of those $v\in V$ with $\left<v,v^\vee\right>=1$, that is, the fiber of $\tilde V$ over $v^\vee$. Thus, we can identify $\tilde V$ with the tautological $G$-orbit on $V^*\times V^h$, consisting of pairs $(v,V_u)$, where $V_u\subset V$ is a generic affine line containing $v$.

 Having fixed the symplectic form $\omega$, we get isomorphisms $V^*\simeq V^{\vee*}\simeq V^h$,  
 and we can identify $\tilde V$ as the subset of $V^*\times V^*$ consisting of pairs $(v,u)$ with $\omega(v,u)=1$. This is a \emph{generic} $G$-orbit in $V^*\times V^*$, that is, one where stabilizers of the two points do not belong to the same Borel subgroup. On the other hand, the identification $V^*\xrightarrow\sim V^h$ corresponds to a \emph{special} $G$-orbit on $V^*\times V^h$, i.e., the stabilizers of two points belong in the same Borel subgroup. Let us write $\mathcal B$ for the flag variety of Borel subgroups of $G$, then it is immediate that, by these associations, a Whittaker structure is also equivalent to the following:
\begin{itemize}   
  \item a \emph{generic} $G$-orbit on $V^*\times V^*$;
  \item a $G$-orbit on $V^*\times_{\mathcal B} V^h$.
\end{itemize}

 Now we define \emph{Fourier transform}, \emph{Radon transform} and the \emph{Jacquet integral}.

 Fourier transform $\mathfrak F$ will be defined as an $\SL(V)$-equivariant endomorphism of the Schwartz space of functions $\mathcal F(V)$ by the formula
\begin{equation}\label{Fourierconvention}\mathfrak F \Phi(v^*) = \int_V \Phi(v) \psi(\omega(v,v^*)) |\omega|(v).\end{equation}

 Radon transform is the map 
 $$ \mathfrak R: \mathcal F(V)\to C^\infty(V^*)$$
 given by the pull-push construction under the above maps, i.e., 
 $$ \mathfrak R = t_! s^*,$$
 where $s^*$ denotes pullback of functions under the second projection, and $t_!$ is integration over the fibers of $t$, \emph{which are $\Ga$-torsors and hence are endowed with the fixed Haar measure of $F$}.

 Explicitly, in coordinates $(x,y)$ with $\omega = dx \wedge dy$,
 \begin{equation}\label{Radonincoordinates} \mathfrak R\Phi(0,1) = \int \Phi(1,x) dx.
 \end{equation}

 The Jacquet integral\footnote{Usually, this name is given to its adjoint $\mathfrak J^*$ that appears below.} is the map
 $$ \mathfrak J: \mathcal F(V^*, \mathcal L_\psi)\to C^\infty(V^*)$$
 given by 
 $$ \mathfrak J = t_! s_\psi^*,$$
 where $s_\psi^*$ is the pullback of Whittaker sections to $\tilde V$ (where, again, the pullback of $\mathcal L_\psi$ is equipped with a trivialization). In coordinates, it is given by the same formula \eqref{Radonincoordinates} as the Radon transform above, as long as the argument of $\Phi$ inside of the integral is replaced by the pair $((1,x), (0,1))$.
 
 The measure $|\omega|$ on $V$ gives rise to a duality pairing between functions, or between sections of the line bundle $\mathcal L_\psi$ and sections of the line bundle $\mathcal L_{\psi^{-1}}$ defined by the inverse character, and 
 the adjoint of the Jacquet integral (a morphism $\mathcal J^*: \mathcal F(V^*)\to C^\infty(V^*,\mathcal L_{\psi^{-1}})$) can be written
 $$ \mathfrak J^* = s_{\psi,!} t^*,$$
 where the twisted push-forward $s_{\psi,!}$ is the dual map to $s_\psi^*$ with respect to the fixed Haar measure on the fibers of the $\Ga$-torsor $\tilde V\xrightarrow{s} V$. 
 
 We are interested in \emph{functional equations (scattering operators) for the Jacquet integral}. 
 
 Consider the $L^2$-normalized action of $\Gm$ on $C^\infty(V^*)$: 
 $$ a\cdot \Phi(v) = |a| \Phi(av).$$
 The Radon transform is anti-equivariant with respect to this action:
 $$\mathfrak R(a\cdot \Phi) = a^{-1} \cdot \mathfrak R\Phi.$$

 We have the following relation between Fourier and Radon transforms:
  \begin{equation}\label{RadonFourier}  \mathfrak F\Phi(v) = \int \mathfrak R(a\cdot \Phi)(v)  \psi^{-1}(a) |a|d^\times a.\end{equation}
 
 Indeed, for $v\in V$, choose a section $a\mapsto u_a = a u_1$ of the quotient map $V\ni u \mapsto \omega(u,v)\in \Ga$; then, by definition, the value of Radon transform at $v$ is
 $$ \mathfrak R(\Phi)(v) = \int \Phi(u_1 - z v) dz,$$
 Hence
 $$ \int \mathfrak R(a\cdot \Phi)(v)  \psi(a) |a|d^\times a = \int |a| \int \Phi(a u_1 + zav)  dz \psi(a) da =  \int \int \Phi(a u_1 + z' v) dz' \psi(a) da.$$
 By definition, the measure $|\omega|$ on $V$ is equal to $da dz'$ when $(a,z)$ are the coordinates in the basis $(u_1, v)$, so the last integral can be written
 $$ \int_V \Phi(u) \psi(\omega(u,v)) |\omega|(u) = \mathfrak F\Phi(v).$$
 
 On the other hand, the adjoint of the Jacquet integral, evaluated on $(v,u)\in \tilde V$, can be written
 $$ \mathfrak J^*\Phi(v,u) = \int \Phi(u-zv) \psi(z) dz.$$
 
 Hence, if $\mathfrak F^*$ denotes Fourier transform defined with the character $\psi^{-1}$, instead of $\psi$,
 $$\mathfrak J^*\mathfrak F^*\Phi(v,u) = \int \iint \Phi(av+bu) \psi(-a-bz+z) da db  dz = 
 \int \Phi(av+u) \psi(-a) da = \mathfrak J^*\Phi(v,u).$$
 
 We have shown: 
 $$\mathfrak J^* \circ \mathfrak F^* = \mathfrak J^*,$$
 or, taking adjoints and noticing that the adjoint of $\mathfrak F^*$ is $\mathfrak F$,
  \begin{equation}\label{JacquetFourier} \mathfrak F \circ \mathfrak J = \mathfrak J.\end{equation}
 
Now we project to coinvariants with respect to various characters of $\Gm$. Mellin transform:
$$ \check\Phi(\chi) (v) = \int a\cdot \Phi(v) \chi^{-1}(a) d^\times a$$
is a morphism
$$\mathcal F(V^*) \to C^\infty(\Gm,\chi\backslash V^*),$$
where the notation $C^\infty(\Gm,\chi\backslash V^*)$ means $(\Gm,\chi)$-equivariant functions with respect to the normalized action. For $\Phi\in \mathcal S(V)$, or for $\Phi$ in the image of Radon transform and the Jacquet integral, it converges for $\chi$ in some domain, and admits rational continuation to all $\chi$. Because of their equivariance properties with respect to the $\Gm$-action, all the above transforms descend to meromorphic families of transforms between the coinvariant spaces, that will be denoted by the index $\chi$:

$$ \xymatrix{ 
\mathcal F(V) \ar[d]\ar[r]^{\mathfrak F} &\mathcal F(V)\ar[d] \\
C^\infty(\Gm,\chi^{-1}\backslash V^*) \ar[r]^{\mathfrak F_\chi} &C^\infty(\Gm,\chi\backslash V^*);}$$

$$ \xymatrix{ 
\mathcal F(V) \ar[d]\ar[r]^{\mathfrak R} & \mathfrak R(\mathcal F(V))\ar[d] \\
C^\infty(\Gm,\chi^{-1}\backslash V^*) \ar[r]^{\mathfrak R_\chi} &C^\infty(\Gm,\chi\backslash V^*);}$$

$$ \xymatrix{ 
\mathcal F(V^*, \mathcal L_\psi) \ar[dr]_{\mathfrak J_\chi} \ar[r]^{\mathfrak J} & \mathfrak J(\mathcal F(V))\ar[d] \\
 &C^\infty(\Gm,\chi\backslash V^*).}$$ 

\begin{remark}\label{remarkcaution}
 Some caution with the notation is needed here when comparing with the operators $\mathfrak N_\chi, \mathfrak M_\chi$ of \eqref{fiberscatteringdiagram}, when $X_\emptyset=V^*$: the identification $V^h\simeq V^*$ that we have here is \emph{anti-equivariant} with respect to the action of $\Gm$. Thus, the space $C^\infty(\Gm,\chi\backslash V^*)$ should be denoted by $C^\infty(\Gm,\chi^{-1}\backslash V^h)$, if we replaced $V^*$ by $V^h$, and what is denoted here with $\mathfrak R_\chi$ would be $\mathfrak M_{\chi^{-1}}$ in the notation of \eqref{fiberscatteringdiagram}, and $\mathfrak J_\chi$ would be $\mathfrak N_{\chi^{-1}}$.
\end{remark}

The relation \eqref{RadonFourier} translates to 
\begin{equation}\label{RadonFourierspectral}
\mathfrak F_\chi= \gamma(\chi,0,\psi)  \mathfrak R_\chi.
\end{equation}

Indeed, we have 
$$ \mathfrak F_\chi(\check\Phi(\chi^{-1}))(v) = \widecheck{\mathfrak F \Phi}(\chi)(v)  = \int z\cdot \mathfrak F\Phi (v) \chi^{-1}(z) d^\times z = $$
$$= \int \int z\cdot \mathfrak R(a\cdot \Phi)(v)  \psi(a) |a|d^\times a \chi^{-1}(z) d^\times z $$
$$ = \int |z| \int x^{-1}\cdot \mathfrak R\Phi(v) \psi(xz) dx \chi^{-1}(z) d^\times z,$$
which is the Tate integral 
$$ Z(\hat\varphi,\chi^{-1}, 1) = \int \hat\varphi(z) |z|\chi^{-1}(z) d^\times z$$
of the function $\hat\varphi(z) = \int \varphi(x) \psi(zx) dx$, where $\varphi (x) = x^{-1}\cdot \mathfrak R\Phi(v)$.  

By the functional equation \cite{Tate-Corvallis}:
 $$ \gamma(\chi,s,\psi) Z(\varphi,\chi,s) = Z(\hat\varphi, \chi^{-1},1-s).$$
we get 
$$\mathfrak F_\chi(\check\Phi(\chi^{-1}))(v) = \gamma(\chi,0,\psi) \int z^{-1}\cdot \mathfrak R\Phi(v) \chi(z) d^\times z = \gamma(\chi,0,\psi)  \mathfrak R_\chi \Phi(v).$$

Similarly, \eqref{JacquetFourier} translates to 
\begin{equation}\label{JacquetFourierspectral}
 \mathfrak F_{\chi^{-1}} \circ \mathfrak J_{\chi} = \mathfrak J_{\chi^{-1}}.
\end{equation}

If we identify $V^*$ with $N\backslash \SL_2$, under our conventions the universal Cartan $A$ of $\SL_2$ acts on $V^*$ by the character $e^\frac{\alpha}{2}$. Thus, for a character $\tilde\chi$ of $A$, we set $\chi = \tilde\chi\circ e^{\check\alpha}$. Then the operators $\mathfrak N_{\tilde\chi}$, $\mathfrak M_{\tilde\chi}$ of \eqref{fiberscatteringdiagram} correspond to $\mathfrak J_{\chi^{-1}}$, $\mathfrak R_{\chi^{-1}}$, respectively (see Remark \ref{remarkcaution}), and we have the corresponding commutative diagram, with \eqref{JacquetFourierspectral} added to it:
$$ \xymatrix{ && C^\infty(\Gm,\chi\backslash V^*) \ar[dd]^{\mathfrak F_{\chi^{-1}}}\ar[r]^{\mathfrak R^{-1}_\chi}  & C^\infty(\Gm,\chi^{-1}\backslash V^*)  \ar[dd]^{\mathscr S_{w,\tilde\chi}}  
\\ \mathcal F(V^*,\mathcal L_\psi) \ar[urr]^{\mathfrak J_{\chi}}\ar[drr]_{\mathfrak J_{\chi^{-1}}}&&&  \\ 
&& C^\infty(\Gm,\chi^{-1}\backslash V^*) \ar[r]^{\mathfrak R_{\chi^{-1}}^{-1}}  & C^\infty(\Gm,\chi\backslash V^*) .}$$

Thus, we get 
$$ \mathscr S_{w,\tilde\chi} = \mathfrak R_{\chi^{-1}}^{-1} \circ \mathfrak F_{\chi^{-1}} \circ \mathfrak R_\chi,$$ 
and, invoking \eqref{RadonFourierspectral}, this is equal to 
$$ \gamma(\chi^{-1},0,\psi)  \mathfrak R_{\chi},$$
or, in other words (writing now $\tilde\chi = \chi\circ e^{\check\alpha}$ as $\chi$):

\begin{equation}\label{scattering-Whittaker} \mathscr S_{w,\chi} = \gamma(\chi, -\check\alpha,0,\psi) \cdot \mathfrak R_{\chi}.\end{equation}

Although we have worked with $\SL(V)$ up to now, this formula remains valid for the Whittaker model of any split group $G$ of semisimple rank one; indeed, given a non-trivial unipotent subgroup $N$ of $G$ with an identification $N\simeq\Ga$, and compatible maps $\SL(V)\to G$, $V^*\to Y:=N\backslash G$, this induces a Whittaker structure on $V$, and hence a distinguished $\SL_2(V)$-orbit in $\subset V^*\times_{\mathcal B} V^h$, whose image determines a distinguished $G$-orbit on $Y\times_{\mathcal B} Y^h$, which, it is immediate to confirm, does not depend on choices. Morover, generic horocycles on $V^*$ map isomorphically to generic horocycles on $Y$, so we can transfer the measures induced by the Whittaker structure, and define the operator $\mathfrak R_{\chi}$ accordingly. The fiberwise scattering maps $\mathscr S_{w,\chi}$ should then be rational multiples of $\mathfrak R_\chi$, and the rational scalar can be computed by pullback to $V^*$; thus, equation \eqref{scattering-Whittaker} remains valid for $G$. 

I add the following corollary to \eqref{scattering-Whittaker}, which will be used later:
\begin{corollary}\label{corWhittakerasymptotics}
 Let $F$ be a non-Archimedean field, $G$ a split group of semisimple rank one, and $X=(N,\psi)\backslash G$ a symbol for the Whittaker model of $G$. Let $A=A_X$ be the universal Cartan of $G$, acting on measures on $X_\emptyset$ by the normalization described in \eqref{action-normalized-measures}. Let $h$ be the measure on $A$ with Mellin transform 
 $$\check h(\chi) = L(\chi,\check\alpha, 1)^{-1}.$$ (It belongs to the Bernstein center, i.e., the completed Hecke algebra of $A$.) Then, for any $\varphi\in \mathcal S^+(X_\emptyset)$, the element
 $$ h\cdot \varphi (x) = \int_A a\cdot \varphi(x) h(a) $$
 belongs to $\mathcal S(X_\emptyset)$.
\end{corollary}

\begin{proof}
 Indeed, $\mathcal S^+(X_\emptyset)$ is generated by $\mathcal S(X_\emptyset)$ under the action of the scattering operator $\mathfrak S_w$, which is expressed in terms of the Mellin transform by \eqref{scatteringMellin}. By \eqref{scattering-Whittaker}, 
 $$ \mathscr S_{w,\chi} (h\cdot \varphi) = \gamma(\chi, -\check\alpha,0,\psi) \cdot \mathfrak R_{\chi} \widecheck{(h\cdot \varphi)({^w\chi})} =
\gamma(\chi, -\check\alpha,0,\psi) \check h({^w\chi^{-1}}) \cdot \mathfrak R_\chi \check\varphi({^w\chi}) =$$
$$ =  \gamma(\chi, -\check\alpha,0,\psi)L(\chi,\check\alpha, 1)^{-1} \cdot \mathfrak R_{\chi} \check \varphi({^w\chi}) = \epsilon(\chi, -\check\alpha,0,\psi)L(\chi, -\check\alpha,0)^{-1} \cdot \mathfrak R_{\chi} \check \varphi({^w\chi})$$
(see \eqref{gammafactor}).

The factor $\epsilon(\chi, -\check\alpha,0,\psi)L(\chi, -\check\alpha,0)^{-1}$ is polynomial in $\chi$, hence corresponds to another element $h'$ of the completed Hecke algebra of $A$. Thus, applying inverse Mellin transform \eqref{scatteringMellin},
\begin{equation}\label{ShR} \mathfrak S_w h\cdot \varphi = h'^\vee \cdot \mathfrak R \varphi,\end{equation}
where $h^\vee (a) = h(a^{-1})$.

The support of the measure $\mathfrak S_w (h\cdot \varphi)$ has compact closure in the affine completion $X_\emptyset^a = \spec F[N\backslash G]$. On the other hand, $\mathfrak R \varphi$ is supported away from the ``cusp'' $X_\emptyset^a\smallsetminus X_\emptyset$, as is very easy to see from the definition. (When we identify the space $X_\emptyset$ with its horocycle space, a point approaching the cusp corresponds to a horocycle approaching ``infinity'' in $X_\emptyset^a$.)

Thus, \eqref{ShR} implies that $\mathfrak S_w (h\cdot \varphi) \in \mathcal S(X_\emptyset)$.

\end{proof}

\subsection{The case of $\Gm\backslash\PGL_2$} \label{scattorus}

Now let $G=\PGL(V)$, where $V$ is a two-dimensional vector space. We assume that $V$ is endowed with an alternating form, and let $X=$ the $G$-variety of quadratic forms of discriminant $-\frac{1}{4}$ (so that in some coordinates $(y,z)$ for a standard symplectic basis, such a form is given by $yz$). Notice that we think of $G$ as $\SL(V)/\{\pm 1\}$ in order to define a $G$-action that fixes the discriminant. Thus, $X\simeq \Gm\backslash \PGL_2$. 

The boundary degeneration $X_\emptyset\simeq N\backslash G$ can be identified with the space of degenerate quadratic forms of rank one, which is canonically isomorphic to $(V^{\vee*})/\{\pm 1\}$, by sending a linear functional to its square. Having fixed the symplectic form $\omega$, and hence the isomorphism $\iota_\omega: V\xrightarrow\sim V^\vee$ as in the previous subsection, we will identify $X_\emptyset$ with $V^*/\{\pm 1\}$.

The evaluation map gives rise to canonical isomorphisms
\begin{eqnarray} (X\times V)\sslash \SL(V)\xrightarrow\sim \Ga, \nonumber \\
 \label{isomorphisms} (X_\emptyset\times V)\sslash \SL(V)\xrightarrow\sim \Ga,\end{eqnarray}
which are the ones inducing the canonical bijection of horocycles \eqref{samehoro}. The preimage of $1$, in each case, is a distinguished $G$-orbit, which projects to distinguished $G$-orbits
$$ \tilde X\subset X\times X_\emptyset,$$
$$ \tilde X_\emptyset \subset X_\emptyset.$$

Notice that, in the case of $\tilde X_\emptyset$, this is the projection of the canonical $G$-orbit $\tilde V\subset V\times V$, induced by the symplectic form on $V$ (as in the previous subsection).

Fourier transform on $V$ descends to Fourier transform on $X_\emptyset$, but we have to be careful, because the map $V^*\to X_\emptyset$ is not surjective at the level of $F$-points. One therefore needs to treat $X_\emptyset$ as an open subset of the stack $[V/\{\pm 1\}]$, and define Fourier transform on functions (or measures) on the $F$-points of this stack. Specifically, this means the following: For any $\alpha\in H^1(F, \Z/2)$ (corresponding to a quadratic extension $E^\alpha$ of $F$, including the trivial one $F\oplus F$), let $R_\alpha$ be the corresponding $\Z/2$-torsor over $F$ (isomorphic to a pair of distinct conjugate points of $E^\alpha$), and let $V^\alpha \simeq V\times^{\Z/2} R^\alpha$. It is an $F$-vector space which can be identified with $V\otimes_F \Im(E^\alpha)$, where $\Im(E^\alpha)$ is the ``imaginary'' line of elements of $E^\alpha$ which are conjugate to their opposite. Then we have
$$ X_\emptyset (F) = \bigoplus_{\alpha \in H^1(F,\Z/2)} V^{\alpha *}(F)/\{\pm 1\}.$$

The symplectic form $\omega: V\times V\to \Ga$ is invariant under the diagonal $\Z/2$-action, and hence induces a symplectic form on $V^\alpha$. Explicitly, if we choose an element $e\in \Im(E^\alpha)$ to write any element of $V^\alpha$ as $v\otimes e$, we have 
$$ \omega(v_1\otimes e, v_2\otimes e) = e^2 \omega(v_1, v_2).$$
This defines Fourier transform on the Schwartz space $\mathcal F(V^\alpha)$, and in particular defines a Fourier transform $\mathfrak F$ on 
$$ \mathcal F(X_\emptyset^a):= \bigoplus_\alpha \mathcal F(V^\alpha(F))^{\Z/2}.$$

(The notation $X_\emptyset^a$ stands for the affine closure of $X_\emptyset$.) It is easy to see that $\mathfrak F$ is a $G$-equivariant and $A_X$-anti-equivariant, endomorphism 
$$\mathfrak F: \mathcal F(X_\emptyset^a) \to \mathcal F(X_\emptyset^a),$$
where, here, $A_X=$ the universal Cartan of $G$, and the action of $A_X$ is the normalized one, as in \eqref{action-normalized-functions}. 

Similarly, the correspondence $\tilde X_\emptyset \underset{t}{\overset{s}\rightrightarrows} X_\emptyset$, together with the Haar measure on horocycles (i.e., fibers of the map $t$) descending from that on $V^*$, gives rise to Radon transform
$$ \mathfrak R: \mathcal F(X_\emptyset^a)  \to C^\infty (X_\emptyset),$$
defined as before.

The analog of the Jacquet integral here is the morphism
$$ \mathfrak I: \mathcal F(X)\to C^\infty(X_\emptyset)$$
obtained by the correspondence 
$$\tilde X \subset X\times X_\emptyset.$$
again with the measure on generic horocycles on $X$ obtained by their identification with generic horocycles of $X_\emptyset$.

The adjoint of $\mathfrak I$ (with respect to invariant measures, which we do not necessarily need to fix)  
$$ \mathfrak I: \mathcal F(X_\emptyset) \to C^\infty(X)$$
given by the integral
\begin{equation}\label{dualI} \mathfrak I^*\Phi(x) = \int_{(X_\emptyset)_x} \Phi(v) \mu(v),\end{equation}
where $(X_\emptyset)_x\subset X_\emptyset$ is the fiber of $\tilde X$ over $x$, and $\mu$ is a $G_x$-invariant measure on it.

As before, the composition of $\mathfrak I$ with Mellin transform will be denoted
$$\mathfrak I_\chi: \mathcal F(X) \to C^\infty(A_X,\chi\backslash X_\emptyset).$$

Dualizing, we get a morphism
$$\mathfrak I_\chi^*: C^\infty(A_X,\chi^{-1}\backslash X_\emptyset) \to C^\infty(X),$$
which in some domain of convergence is given again by \eqref{dualI}.

The composition of this with Mellin transform
$$\tilde{\mathfrak I}_\chi^*: \mathcal F(X_\emptyset) \to C^\infty(A_X,\chi^{-1}\backslash X_\emptyset) \xrightarrow{\mathfrak I_\chi^*} C^\infty(X)$$
can be written as 
\begin{equation}\label{dualII} \tilde{\mathfrak I}_\chi^*\Phi(x) = \int_{X_\emptyset} \Phi(v) \tilde\chi\circ p((x,v)) dv,\end{equation}
where $\tilde\chi(z)=|z|^{-\frac{1}{2}}\chi\circ e^{\frac{\check\alpha}{2}}(z)$ and 
 $$ p: (X\times X_\emptyset)\sslash G \xrightarrow\sim \Ga$$
is the canonical isomorphism induced from the evaluation map \eqref{isomorphisms}. 

Notice that, if we denote by $A$ the universal Cartan of $G'= \SL(V)$ which acts on $V^*$, the map $A\backslash V^* \to A_X\backslash X_\emptyset$ is an isomorphism. Thus, we can think of $\mathfrak I_\chi^*$ as a morphism from $C^\infty(A,\chi'^{-1}\backslash V^*)$, where $\chi'$ is the pullback of $\chi$ to $A$. This morphism, though, will depend on $\chi$ itself, not just $\chi'$, and if we unfold the definitions and compose with Mellin transform, we will see that the resulting functional
$$\tilde{\mathfrak I}'^*_\chi :\mathcal F(V^*) \to C^\infty(A,\chi'^{-1}\backslash X'_\emptyset) \xrightarrow{\mathfrak I_\chi^*} C^\infty(X)$$
is given by the formula 
$$ \tilde{\mathfrak I}'^*_\chi\Phi(x) = \int_{V^*} \Phi(v') \tilde\chi\circ p((x,v')) dv',$$
where we are using the same letter ($p$) for the evaluation map on $X\times V^*$.

We explicate this functional: Given $x\in X$, choose a standard symplectic basis $(u,v)$ on $V$ (i.e., $\omega(u,v)=1$) such that the quadratic form associated to $x$ is 
$$ y u + z v \mapsto yz.$$
In these coordinates, the above integral is equal to 
$$\int \Phi(z,y) \tilde\chi(yz) dy dz,$$
and adjoint Fourier transform on $V$ is biven by 
 
$$ \mathfrak F^*\Phi(y,z) = \int_V \Phi(a, b) \psi^{- 1}(az-by) da db.$$

We compute their composition. 
Assuming $\Phi(y,z) = \Phi_1(y) \Phi_2(z)$ for convenience of notation, we have
$$ \tilde{\mathfrak I}'^*_\chi \mathfrak F^*\Phi(x) =  \int \int \Phi_1(a) \Phi_2(b) \psi^{-1}(az-by) da db \tilde\chi(yz) dy dz =$$
$$ = Z(\hat\Phi_1^{\psi^{-1}}, \tilde\chi, 1) Z(\hat\Phi_2^{\psi}, \tilde\chi, 1),$$
where the exponent of Fourier transforms denotes the character they are defined by. Applying the local functional equation, we get that this is equal to 
$$ \gamma(\tilde\chi^{-1}, 0, \psi) \gamma(\tilde\chi^{-1},0, \psi^{-1}) \int \Phi(z,y) \tilde\chi^{-1}(zy) d^\times z d^\times y = $$
$$ = \gamma(\tilde\chi, 1, \psi^{-1})^{-1} \gamma(\tilde\chi,1, \psi)^{-1} \tilde{\mathfrak I}'^*_{\chi^{-1}} \Phi(x) = $$
$$ = \gamma(\chi, \frac{\check\alpha}{2}, \frac{1}{2}, \psi^{-1})^{-1} \gamma(\chi, \frac{\check\alpha}{2}, \frac{1}{2}, \psi)^{-1} \tilde{\mathfrak I}'^*_{\chi^{-1}} \Phi(x). $$

Dualizing, 
$$\mathfrak F_{\chi^{-1}} \circ \tilde{\mathfrak I}'_\chi = \gamma(\chi, \frac{\check\alpha}{2}, \frac{1}{2}, \psi^{-1})^{-1} \gamma(\chi, \frac{\check\alpha}{2}, \frac{1}{2}, \psi)^{-1} \tilde{\mathfrak I}'_{\chi^{-1}}.$$

Thus, the diagram \eqref{fiberscatteringdiagram} now reads
$$ \xymatrix{ && C^\infty(A_X,\chi\backslash X_\emptyset) \ar[dd]^{c_\chi\cdot  \mathfrak F_{\chi^{-1}}}\ar[r]^{\mathfrak R_{\chi^{-1}}^{-1}}  & C^\infty(A_X,\chi^{-1}\backslash X_\emptyset)  \ar[dd]^{\mathscr S_{w,\chi}}  
\\ \mathcal F(X) \ar[urr]^{\mathfrak I_{\chi}}\ar[drr]_{\mathfrak I_{\chi^{-1}}}&&&  \\ 
&& C^\infty(A_X,\chi^{-1}\backslash X_\emptyset) \ar[r]^{\mathfrak R_{\chi}^{-1}}  & C^\infty(A_X,\chi\backslash X_\emptyset),}$$
where $c_\chi = \gamma(\chi, \frac{\check\alpha}{2}, \frac{1}{2}, \psi^{-1}) \gamma(\chi, \frac{\check\alpha}{2}, \frac{1}{2}, \psi)$, 
and invoking \eqref{RadonFourierspectral} again, we get:

\begin{equation}\label{scattering-torus}\mathscr S_{w,\chi} = \gamma(\chi, \frac{\check\alpha}{2}, \frac{1}{2}, \psi^{-1}) \gamma(\chi, \frac{\check\alpha}{2}, \frac{1}{2}, \psi) \gamma(\chi,-\check\alpha, 0,\psi)  \mathfrak R_{\chi},\end{equation}

\subsection{The group case} \label{ssscatteringgroup}

Let us now study the group case, $X=H = \SL(V)$, where $V$ is a two-dimensional vector space, and $G=H\times H$. (Again, $H$ acts on the right on $V$.) The asymptotic cone $X_\emptyset=H_\emptyset$ can be identified with the subspace of $\End(V)$ of elements of rank one. 
The identification \eqref{samehoro} of generic horocycles, then, is as follows: a generic horocycle, both in $H$ and in $H_\emptyset$, is equivalent to a pair $(L, L')$ of lines in $V$, together with a non-zero homomorphism $\tau: L\to V/L'$; these data correspond to the horocycle of the pair $(B\times B', Y)$, where $B, B'$ are the stabilizers of $L, L'$, and $Y\subset H$ or $Y\subset H_\emptyset$ is the subvariety of endomorphisms which induce $\tau$. Notice that these data are also equivalent to an element of $\Gm\backslash V^*\times V^h$, namely the class of $(v, \tau(v))$, where $v$ is a non-zero element of $L$. Thus, $H_\emptyset^h = \Gm\backslash V^*\times V^h$, canonically.

For a pair $(v,Y) \in V^*\times V^h$, we will write $[v:Y]\in \Ga$ for the scalar $\lambda$ such that $\lambda v \in Y$. 
We can also identify $H_\emptyset$, canonically, with the space $\Gm^\diag\backslash (V^h\times V^*)$ (notation as before), by mapping a pair $(Y,v)$ consisting of a generic horocycle and a non-zero vector to the rank-one endomorphism that sends $Y$ to $v$, in other words for the endomorphism $u\mapsto [u:Y]\cdot v$.

For any pair 
$$\tau_1: V/L_1\to L_1'\subset V$$
$$\tau_2: V/L_2\to L_2'\subset V$$
of elements of $H_\emptyset$ we obtain, by restriction, homomorphisms
$$ \tau_2|_{L_1} : L_1\to L_2', \mbox{ and }$$
$$ \tau_1|_{L_2}: L_2\to L_1'.$$
As long as $L_1\ne L_2$ and $L_1'\ne L_2'$, there is a unique element $M_{\tau_1, \tau_2}\in \GL(V)$ whose restrictions also induce these endomorphisms, and the pairs $(\tau_1,\tau_2)$ with $M_{\tau_1,\tau_2}\in \SL(V)$ form a distinguished $G^\diag$-orbit $\tilde H_\emptyset \subset H_\emptyset\times H_\emptyset$. Equivalently, the distinguished $G^\diag$-orbit is characterized by the property that
\begin{equation}\label{groupcondition} v_1\wedge v_2 = \tau_2(v_1) \wedge \tau_1(v_2)
\end{equation}
for all $v_1\in L_1, v_2\in L_2$.

The correspondence between generic $G^\diag$-orbits on $H_\emptyset\times H_\emptyset$ and $G^\diag$-orbits on $H_\emptyset\times_{\mathcal B_H \times\mathcal B_H} H_\emptyset^h$ sends a pair $(\tau_1,\tau_2)$ as above to the pair $(\tau_1, \tau_2^\circ)$, where $\tau_2^\circ\in H_\emptyset^h$ is the horocycle represented by the composition
$$ L_1 \xrightarrow{\tau_2|_{L_1}} L_2' \to V/L_1'.$$
The distinguished $G^\diag$-orbit corresponds to the set of pairs 
$$ (V/L\xrightarrow{\tau_1} L'\subset V, V\supset L\xrightarrow{\tau_2^\circ} V/L') \in H_\emptyset\times_{\mathcal B_H \times\mathcal B_H} H_\emptyset^h$$
such that any lift of $(\tau_1,\tau_2^\circ)$ to an endomorphism of $V$ has determinant $1$.

Equivalently, representing 
$$H_\emptyset\times_{\mathcal B_H \times\mathcal B_H} H_\emptyset^h =\Gm^\diag\backslash(V^h\times V^*) \times_{\mathcal B_H\times \mathcal B_H} \Gm^\diag\backslash (V^* \times V^h),$$
the distinguished $G^\diag$-orbit is obtained as follows: choose \emph{any} $H$-orbit on $V^*\times_{\mathcal B} V^h$, say with a representative $(v, Y)$, and take the $G^\diag$-orbit represented by $(Y, v, -v, Y)$. Indeed, for any $y\in Y$, the endomorphism defined by $y\mapsto v, -v\mapsto y$ has determinant one.

This gives rise to a \emph{canonical} Radon transform 
$$\mathfrak R : \mathcal F(H_\emptyset)  \to C^\infty(H_\emptyset),$$
which is given by
$$ \mathfrak R \Phi(V/L\xrightarrow\tau L') = \int_{(\P(V)\smallsetminus\{L\})\times (\P(V)\smallsetminus\{L'\})} \Phi(\sigma^\tau_{(L_1, L_1')}) d(L_1, L_1'),$$
where $\sigma^\tau_{(L_1, L_1')}: V/L_1 \to L_1'\subset V$ is the unique such morphism with the property that \eqref{groupcondition} is satisfied when $\tau_1=\sigma$ and $\tau_2=\tau$. The measure on $(\P(V)\smallsetminus\{L\})\times (\P(V)\smallsetminus\{L'\})$ is the following: choose \emph{any} pair of vectors $(v\in L, v'\in L')$, let $Y =\tau^{-1}(v')\subset V$ and let $Y'\subset V$ be the horocycle of all vectors $y'$ with $y\wedge v = v'\wedge y'$ for all $y\in Y$. Then $Y\times Y'$ is a $\Ga\times \Ga$-torsor by $(x,x')\cdot (y, y') = (y-x v, y'-x' v')$, hence carries a measure induced from our fixed measure on $F^2$, and under the projection map it can be identified with $(\P(V)\smallsetminus\{L\})\times (\P(V)\smallsetminus\{L'\})$. It is immediate that the resulting measure on $(\P(V)\smallsetminus\{L\})\times (\P(V)\smallsetminus\{L'\})$ does not depend on the choice of $(v,v')$.

Another way to explicate this Radon transform is to choose any Whittaker structure (equivalently, a symplectic form $\omega$) on $V^*$. The $\Gm$-anti-equivariant identification $V^*\xrightarrow\sim V^h$ that it induces (which we have normalized so that $v$ corresponds to the horocycle $\{u\in V| \omega(u,v)=1\}$) allows us to identify 
$$H_\emptyset =\Gm^\diag\backslash (V^h\times V^*) \simeq \Gm^\adiag\backslash (V^*\times V^*),$$ 
and defines two correspondences $\tilde V \times V^* \rightrightarrows V^*\times V^*$ and $V^*\times \tilde V \rightrightarrows V^*\times V^*$, where $\tilde V = \{(u,v)|\omega(u,v)=1\}$. The product of the resulting Radon transforms (where we denote by an index the variable that we apply the transform to):
$${\mathfrak R = \mathfrak R_1 \boxtimes \mathfrak R_2}: \mathcal F(V^*\times V^*) \to  C^\infty(V^*\times V^*)$$
descends to $\Gm^\adiag\backslash V^*\times V^*$, i.e., to a map
$$\mathfrak R: \mathcal F(V^*\times V^*) \to  C^\infty(V^*\times V^*).$$
This depends on the choice of a symplectic form, but its pullback to $H_\emptyset$ does not, and coincides with the Radon transform described above, as a simple calculation shows. 

More explicitly, take a Borel subgroup of $G$ of the form $B\times B^-$, where $B, B^-$ are two opposite Borel subgroups of $H$ with unipotent radicals $N, N^-$ and intersection $B\cap B^-=T$, and identify $H_\emptyset = T\times^{B\times B^-} G$ in such a way that the embeddings of $T$ into $H$ and $H_\emptyset$ are compatible with the isomorphism \eqref{samehoro}: $N\backslash H\sslash N^- = H_\emptyset\sslash (N\times N^-)$. If we identify $H$ with $\SL_2$, $B=$ the upper triangular subgroup and $B^-=$ the lower triangular subgroup, then we have
$$ \mathfrak R\Phi(1) = \int_{F^2} \Phi\left( \begin{pmatrix} & -1 \\ 1 \end{pmatrix} \begin{pmatrix} 1 & x \\ & 1 \end{pmatrix}, \begin{pmatrix} & -1 \\ 1 \end{pmatrix} \begin{pmatrix} 1 &  \\ y& 1 \end{pmatrix} \right) dx dy.$$

By \cite[Proposition 15.2]{DHS}, the scattering operator for the non-trivial element $w\in W_X=W_H$ is given by 
$$ \mathscr S_{w,\chi}= \mathfrak R_{1,\chi} \boxtimes \mathfrak R_{2,\chi^{-1}}^{-1}: C^\infty(A_H,\chi^{-1}\backslash H_\emptyset) \to C^\infty(A_H,\chi\backslash H_\emptyset).$$
The individual factors of this depend on the choice of a Whittaker structure, but the product does not. Moreover, the expression is symmetric in the two factors, i.e., we have $\mathfrak R_{1,\chi} \boxtimes \mathfrak R_{2,\chi^{-1}}^{-1} = \mathfrak R_{1,\chi^{-1}}^{-1}\boxtimes\mathfrak R_{2,\chi}$. 

We now compute this as a multiple of the operator $\mathfrak R_\chi= \mathfrak R_{1,\chi}\boxtimes R_{2,\chi}$. This is the calculation of the Plancherel measure, but it follows immediately from \eqref{RadonFourierspectral}: Invoking Fourier transform on $V^*$, and the fact that $\mathfrak F^* \circ \mathfrak F = 1$, where $\mathfrak F^*$ is Fourier transform defined with the character $\psi^{-1}$, instead of $\psi$, and setting $\tilde\chi = \chi\circ e^{\check\alpha}$ (so that it corresponds to the character $\chi$ of \eqref{RadonFourierspectral}),
we have 
$$ 1 = \mathfrak F_{\chi^{-1}} \circ \mathfrak F^*_\chi = \gamma(\tilde\chi^{-1},0,\psi) \mathfrak R_{2,\chi^{-1}} \circ \gamma(\tilde\chi,0,\psi^{-1}) \mathfrak R_{2,\chi} \Rightarrow$$
$$ \mathfrak R_{2,\chi^{-1}}^{-1} =  \gamma(\tilde\chi^{-1},0,\psi) \gamma(\tilde\chi,0,\psi^{-1}) \mathfrak R_{2,\chi},$$
and we get
\begin{equation}\label{scattering-group}
 \mathscr S_{w,\chi}= \gamma(\chi,\check\alpha, 0,\psi^{-1}) \gamma(\chi,-\check\alpha,0,\psi)  \cdot \mathfrak R_{1,\chi} \boxtimes \mathfrak R_{2,\chi} = \gamma(\chi,\check\alpha, 0,\psi^{-1}) \gamma(\chi,-\check\alpha,0,\psi)  \cdot \mathfrak R_\chi.
\end{equation}

\vspace{2cm}

We summarize the formulas \eqref{scattering-Whittaker}, \eqref{scattering-torus} and \eqref{scattering-group} for the scattering operators.

\begin{theorem}\label{thmscattering}
 For the cases of Table \eqref{tableX}, in terms of the canonical\footnote{To summarize: Radon transform depends on a choice of $G$-orbit on $X_\emptyset\times_{\mathcal B} X^h_\emptyset$, and there is a \emph{canonical} choice in any one of the three cases.} spectral Radon transforms $\mathfrak R_\chi$ that were described in the previous subsections,
 $$\mathfrak R_\chi: C^\infty(A_X,\chi^{-1}\backslash X_\emptyset) \to C^\infty(A_X,\chi\backslash X_\emptyset),$$  
 the scattering operator $\mathscr S_{w,\chi}$ for the non-trivial element $w$ of $W_X$ is given by the following formulas:
 
 \begin{itemize}
 \item  For the Whittaker case, $X = (N,\psi)\backslash G$, 
\begin{equation} \mathscr S_{w,\chi} = \gamma(\chi, -\check\alpha,0,\psi) \cdot \mathfrak R_{\chi};\end{equation}

 \item  for the variety $X = \Gm\backslash \PGL_2$,
\begin{equation}\mathscr S_{w,\chi} = \gamma(\chi, \frac{\check\alpha}{2}, \frac{1}{2}, \psi^{-1}) \gamma(\chi, \frac{\check\alpha}{2}, \frac{1}{2}, \psi) \gamma(\chi,-\check\alpha, 0,\psi) \cdot \mathfrak R_{\chi};\end{equation}

 \item for the group case, $X = H=\SL_2$ under the $G=H\times H$-action, 
 \begin{equation} 
  \mathscr S_{w,\chi}= \gamma(\chi,\check\alpha, 0,\psi^{-1}) \gamma(\chi,-\check\alpha,0,\psi)  \cdot \mathfrak R_\chi .
\end{equation}
\end{itemize}

\end{theorem}

\subsection{Relative characters}\label{ssrelchars}

The quotient $\C_\emptyset:=(X_\emptyset \times X_\emptyset)\sslash G$ has an $A_X$-action, which we normalize so that it descends from the action on the first copy of $X_\emptyset$, and the canonical generic $G$-orbit $\tilde X_\emptyset\subset X_\emptyset\times X_\emptyset$, that was also used to define Radon transform, defines in every case an isomorphism
$$ (X_\emptyset \times X_\emptyset)^\circ\sslash G \xrightarrow\sim A_X \subset \C_\emptyset,$$
where $(X_\emptyset \times X_\emptyset)^\circ$ is the open subspace of pairs of points whose stabilizers belong to distinct Borel subgroups. 

Characters of $F^\times \subset F$ pull back, generically, to $G$-invariant \emph{generalized functions} on $X_\emptyset\times X_\emptyset$, in fact, to functionals on the extended Schwartz space $\mathcal S^+(X_\emptyset\times X_\emptyset)$ obtained by applying scattering operators (on both variables) to $\mathcal S(X_\emptyset\times X_\emptyset)$,
as the following proposition shows. To formulate it, 
 let $p:(X_\emptyset \times X_\emptyset)^\circ \to A_X$
 be the quotient map; its fiber over $1\in A_X$ is the distinguished $G$-orbit $\tilde X_\emptyset$. Fix a $G$-invariant measure on $X_\emptyset$, and take the Haar measure $da$ on $A_X$ which disintegrates the Haar measure $dx\times dx$ on $X_\emptyset \times X_\emptyset$ as the product of $\delta(a) da$ with the $G$-invariant measure on $\tilde X_\emptyset$ for which the following formula holds:
 $$ \int_{\tilde X_\emptyset} \Phi(u,v) d(u,v) = \int_{X_\emptyset} \mathfrak R_1\Phi(v,v) dv.$$
 Here, as before, $\mathfrak R_1$ denotes Radon transform in the first variable.

\begin{proposition}
Given a measure $\varphi(u,v)= \Phi(u,v) du dv \in \mathcal S^+(X_\emptyset\times X_\emptyset)$, the ``open'' Mellin tranform of its push-forward $f$ to $\C_\emptyset$ (an element of the space of push-forward measures that we should denote by $\mathcal S^+(X_\emptyset\times X_\emptyset/G)$)
\begin{equation}\label{Mellin-boundary}  \check f(\chi) = \int_{A_X} f(\xi) \chi^{-1}(\xi) 
\end{equation}
converges when $\chi^{-1}$ vanishes sufficiently fast on the complement of $A_X$ in $\C_\emptyset$, and admits rational continuation.

The measure $f\in \mathcal S^+(X_\emptyset\times X_\emptyset/G)$ can be reconstructed by inverse Mellin transform:
\begin{equation}\label{invMellin}
 f = \left(\int_{\omega\widehat{A_X}} \check f(\chi) d\chi\right) da,
\end{equation}
where $\omega$ is a character such that $\omega^{-1}$ vanishes sufficiently fast on the complement of $A_X$ in $\C_\emptyset$, and $(da, d\chi)$ is a pair of dual Haar measures on $A_X$ and $\widehat{A_X}$.

Finally, the ``open'' Mellin transform can be written in terms of ``closed'' Mellin transforms and Radon transforms as follows:
\begin{equation}\label{closed}
\check f(\chi)= \int_{X_\emptyset} \mathfrak R_{1, {^w\chi} \delta^{\frac{1}{2}}} \Phi(v, v) dv,
\end{equation}
where $\mathfrak R_{1, {^w\chi} \delta^{\frac{1}{2}}}: \mathcal F(X_\emptyset\times X_\emptyset) \to \mathcal F((A_X, {^w\chi} \delta^{\frac{1}{2}} \backslash X_\emptyset)\times X_\emptyset)$ is the spectral Radon transform in the first variable, encountered in the previous subsections, defined with that Haar measure $da$ on $A_X$ described above.
\end{proposition}

The terms ``open'' and ``closed'', here, refer to the ``Bruhat cell'' $X_\emptyset \times X_\emptyset$ on which the integration takes place. Notice that only the ``open'' Mellin transform is a functional on the space of push-forward measures $\mathcal S^+(X_\emptyset\times X_\emptyset/G)$; the closed Bruhat cell lives over a proper subvariety of $\C_\emptyset$, and information about it is lost when we take push-forwards.
The spectral Radon transform $\mathfrak R_{1,\chi}$, here, is (by abuse of notation) the composition of what was denoted before by $\mathfrak R_\chi$ with Mellin transform \eqref{Mellin-Splus}, applied to the first variable; when $\Phi(u,v) = \Phi_1(u) \Phi_2(v)$ the integral \eqref{closed} can also be written in terms of the Mellin transforms of $\Phi_1$, $\Phi_2$ as 
\begin{equation}\label{closed-alt}\int_{A_X\backslash X_\emptyset} \mathfrak R_{1, {^w\chi} \delta^{\frac{1}{2}}} \left(\check\Phi_1(\chi \delta^{-\frac{1}{2}})\right) \cdot \check\Phi_2({^w\chi^{-1}} \delta^{-\frac{1}{2}}).
\end{equation}

Notice that, unlike the Mellin transform \eqref{Mellin-Splus} on $X_\emptyset$, in the definition of Mellin transform for $\mathcal S^+(X_\emptyset\times X_\emptyset/G)$, we do not normalize the action of $A_X$ on measures or functions on $\C_\emptyset$, which is why, as we will see, the Mellin transform defined here is a relative character for $C^\infty (A_X, \chi \delta^{-\frac{1}{2}} \backslash X_\emptyset) \simeq I(\chi^{-1} \delta^{\frac{1}{2}})$, the principal series representation obtained by normalized induction from the character $\chi^{-1} \delta^{\frac{1}{2}}$ of $A_X$.

\begin{proof}
  We compute:
$$ \check f(\chi) = \int_{X_\emptyset\times X_\emptyset} \Phi(u, v) \chi^{-1}(p(u,v)) du dv = $$
$$= \int_{A_X} \left(\int_{\tilde X} \Phi(a\cdot u, v) d(u,v)\right) \chi^{-1}(a) \delta(a) da=$$
$$ = \int_{A_X} \int_{X_\emptyset} \mathfrak R_1\Phi(a^{-1} v, v) dv \chi^{-1}(a) da = \int_{X_\emptyset} \mathfrak R_{1, {^w\chi} \delta^{\frac{1}{2}}} \Phi(v, v) dv.$$

This is rational in $\chi$, as the expression \eqref{closed-alt} shows; recall that an element of $\mathcal S^+(X_\emptyset)$ can be written as $\varphi_1\oplus \mathfrak S_w \varphi_2$, where both $\varphi_1$ and $\varphi_2$ belong to $\mathcal S(X_\emptyset)$, and that the scattering map $\mathfrak S_w$ is decomposed \eqref{scatteringMellin} in terms of the fiberwise scattering maps $\mathscr S_{w,\chi}$, which are rational in $\chi$.

The formula for the inverse Mellin transform \eqref{invMellin} follows from the fact that elements of $\mathcal S^+(X_\emptyset\times X_\emptyset)$ are of moderate growth, and up to a rapidly decaying measure are supported on a compact subset of an affine embedding $X^a_\emptyset\times X^a_\emptyset$; thus, their push-forwards will be of moderate growth on $A_X$ and, up to rapid decay, supported on a compact subset of the affine embedding $\C_\emptyset^a:= X^a_\emptyset\times X^a_\emptyset\sslash G$.\footnote{In the examples of Table \eqref{tableX}, this coincides with $\C_\emptyset$.} Thus, the measure $\omega^{-1}f$ will be in $L^2(A_X)$, for $\omega$ as in the statement of the proposition, and the result follows from standard Fourier analysis.

 \end{proof}

Let $M_\chi$ denote the functional $f\mapsto \check f(\chi \delta^{\frac{1}{2}})$ on $\mathcal S^+(X_\emptyset\times X_\emptyset/G)$. 
As is clear from \eqref{closed-alt}, it is a \emph{relative character} for the representation
$$\pi_{\chi}:=C^\infty (A_X, \chi \backslash X_\emptyset) \simeq I(\chi^{-1}),$$
i.e., its pullback to $\mathcal S^+(X_\emptyset\times X_\emptyset)$ factors through a morphism 
$$ \mathcal S^+(X_\emptyset\times X_\emptyset) \to \pi_{\chi} \otimes \widetilde{\pi_{\chi}} \xrightarrow{\left<\,\, , \,\,\right>} \CC.$$

Now we restrict to the case when $F$ is non-Archimedean.\footnote{Replacing the Schwartz space with the Harish-Chandra Schwartz space, and restricting to $\chi$ unitary, the results that follow extend to the Archimedean case, using asymptotics of admissible generalized matrix coefficients. We will not need them, so we avoid introducing more material.} 
We would like to consider its further pullback to $\mathcal S(X\times X)$ via the asymptotics map $e_\emptyset^*\otimes e_\emptyset^*$, and determine its role in the Plancherel formula  of $X$. For that purpose, denote by $M_\chi^\cl$ the corresponding ``closed'' Mellin transform that we saw in \eqref{closed} (without the composition with Radon transform), that is,
$$ M_\chi^\cl(f) = \int_{X_\emptyset} \check\Phi_1({\chi} ) \cdot \check\Phi_2({\chi^{-1}} ),$$
when $f$, here, denotes the image of $\Phi_1(u)\Phi_2(v) du dv$ \emph{in the $G$-coinvariant space $\mathcal S^+(X_\emptyset\times X_\emptyset)_G$}. As mentioned above, here we cannot identify $f$ with its push-forward, because this would lose the information about ``closed'' Mellin transforms. 

Let $I_\chi$, $I_\chi^\cl$ be the pullbacks of $M_\chi$, $M_\chi^\cl$, respectively, to $\mathcal S(X\times X)$ via the asymptotics map $e_\emptyset^*\boxtimes e_\emptyset^*$.

The following theorem, proven in \cite[\S 14.1]{SV}, states that the relative characters $I_\chi^\cl$ decompose the ``most continuous summand'' of the space  $L^2(X)$, in the sense of the Plancherel decomposition. The ``most continuous summand'' is a canonical subspace $L^2(X)_\emptyset\subset L^2(X)$, whose definition I will not repeat here. Fix an invariant measure $dx$ on $X$ and a compatible measure on $X_\emptyset$ \cite[\S 4.2]{SV}, and use them to consider $L^2(X)$, $L^2(X_\emptyset)$ as spaces of measures.

\begin{theorem}\label{Plancherel}
For $\varphi_1, \varphi_2\in L^2(X)_\emptyset$, we have 
$$\int_X \frac{\varphi_1 \cdot \varphi_2}{dx} = \frac{1}{|W_X|}\int_{\widehat{A_X}} I_\chi^\cl(\varphi_1\otimes\varphi_2)  d\chi,$$
where the Haar measure $d\chi$ on the unitary dual $\widehat{A_X}$ is the one dual to the Haar measure $da$ on $A_X$. 
\end{theorem}

Thus, the relative characters $I_\chi^\cl$ are, in some sense, the canonical characters which decompose the most continuous spectrum of $X$ against Haar--Plancherel measure $d\chi$. 
The whole point of the present section was to compare the pullbacks $I_\chi$ of the ``open'' Mellin transforms $M_\chi$ to these relative characters:

\begin{theorem} \label{thmpullbackfrombd}
For each of the cases of Table \eqref{tableX}, define a rational function $\mu_X(\chi)$ on $\widehat{A_X}_\CC$, as follows:
  \begin{itemize}
 \item  For the Whittaker case, $X = (N,\psi)\backslash G$, 
\begin{equation} \mu_X(\chi) = \gamma(\chi, \check\alpha,0,\psi);\end{equation}

 \item  for the variety $X = \Gm\backslash \PGL_2$,
\begin{equation}\mu_X(\chi) = \gamma(\chi, -\frac{\check\alpha}{2}, \frac{1}{2}, \psi^{-1}) \gamma(\chi, -\frac{\check\alpha}{2}, \frac{1}{2}, \psi) \gamma(\chi,\check\alpha, 0,\psi);\end{equation}

 \item for the group case, $X = H=\SL_2$ under the $G=H\times H$-action, 
 \begin{equation} 
  \mu_X(\chi) = \gamma(\chi,-\check\alpha, 0,\psi^{-1}) \gamma(\chi,\check\alpha,0,\psi).
\end{equation}
\end{itemize}

Then the pullbacks of the ``open'' and ``closed'' Mellin transforms are related by the formula
\begin{equation}\label{relationopenclosed}
I_\chi(\Phi_1 \otimes \Phi_2) = \mu_X(\chi)^{-1} I_\chi^\cl(\Phi_1\otimes\Phi_2).\end{equation}

In particular, the relative characters $I_\chi$ decompose the space $L^2(X)_\emptyset$ \emph{with Plancherel measure $\mu_X(\chi)d\chi$}:
$$\int_X \frac{\varphi_1 \cdot \varphi_2}{dx} = \frac{1}{|W_X|}\int_{\widehat{A_X}} I_\chi(\varphi_1\otimes\varphi_2)  \mu_X(\chi) d\chi.$$
\end{theorem}

\begin{proof}
 Recall that the image of the asymptotics map is invariant under the scattering operators. In particular, if $\Phi_i = e_\emptyset^* \varphi_i$ ($i=1,2$) then $\mathscr S_{w,\chi}\check\Phi_i({^{w^{-1}}\chi}) = \check\Phi_i(\chi)$. 
 
 On the other hand, by Theorem \ref{thmscattering}, $\mathscr S_{w,\chi} = \mu_X({^w\chi}) \cdot \mathfrak R_\chi$. The $G$-equivariant scattering operator, applied to the first variable:
 $$\mathfrak S_{w,1}: \mathcal S^+(X_\emptyset\times X_\emptyset) \to \mathcal S^+(X_\emptyset\times X_\emptyset)$$ 
 descends to $G^\diag$-coinvariants, and hence \eqref{closed} reads 
$$ M_\chi = \mu_X({\chi})^{-1} M^\cl_{^w\chi} \circ \mathfrak S_{w,1}.$$

 Therefore, the pullbacks to $\mathcal S(X\times X)$ satisfy \eqref{relationopenclosed}.
 
 The final statement follows from the Plancherel formula of the previous theorem.
\end{proof}

In our applications of this theorem, we will want to compare the transfer operator or Hankel transform for two relative trace formulas:
$$ \mathcal T: \mathcal S(X\times X/G) \to \mathcal S(Y\times Y/G')$$
(typically with non-standard spaces of test measures, which do not appear in our notation here), 
with an abelian transfer operator for the corresponding degenerations
$$ \mathcal T_\emptyset: \mathcal S^+(X_\emptyset\times X_\emptyset/G) \to \mathcal S(Y_\emptyset\times Y_\emptyset/G'),$$
which is chosen so that the following diagram commutes:
\begin{equation}\label{Bcommute} \xymatrix{
\mathcal S(X\times X/G) \ar[rr]^{e_\emptyset^*\otimes e_\emptyset^*}\ar[d]^{\mathcal T} && \mathcal S^+(X_\emptyset\times X_\emptyset/G) \ar[d]^{\mathcal T_\emptyset} \\
\mathcal S(Y\times Y/G') \ar[rr]^{e_\emptyset^*\otimes e_\emptyset^*} && \mathcal S^+(Y_\emptyset\times Y_\emptyset/G') }.
\end{equation}

Here, by abuse of notation, we denote by $e_\emptyset^*\otimes e_\emptyset^*$ the descent of the morphism
$$ e_\emptyset^*\otimes e_\emptyset^*: \mathcal S(X\times X)\to \mathcal S^+(X_\emptyset\times X_\emptyset)$$
to the spaces of push-forward measures. The fact that it descends follows from the following fact, true for each of the spaces $X$ of Table \ref{tableX}, but not for their degenerations $X_\emptyset$:

\begin{theorem}\label{density}
 For each of the spaces $X$ of Table \ref{tableX}, the map 
 $$\mathcal S(X\times X)_G\to \mathcal S(X\times X/G),$$ 
 from coinvariants for the diagonal $G$-action to the space of push-forward measures, is an isomorphism.
\end{theorem}

\begin{proof}
 For the Kuznetsov case, this is proven in more generality by Jacquet, Aizenbud and Gourevitch in \cite[Theorem 1.1]{Jacquet-density}, \cite[Corollary 6.0.4]{AGsmoothtransfer}. 
 
 The quotient stack $[\Gm\backslash \PGL_2/\Gm]$ looks locally like $[\mathbbm A^2/\Gm]$ (with action $(x,y)\cdot a = (ax, a^{-1}y)$ \cite[Lemma 3.2]{SaBE1}. The density of regular orbital integrals is then \cite[Lemma 2.3]{SaBE1}, which we will revisit in Lemma \ref{coinvariantA2} below to fill in some missing details in the proof.
 
 The theorem on the group is a well-known theorem of Harish-Chandra. Notice that $[\SL_2 \times \SL_2 /\SO_4^\diag] \simeq [\frac{\SL_2}{\PGL_2}]$, so we can invoke the density of regular semisimple orbital integrals on the open subset $\SL_2(F)/\{\pm 1\}$ of $\PGL_2(F)$.
\end{proof}

The way to verify that diagram \eqref{Bcommute} commutes is to examine pullbacks of Mellin transforms. Theorem \eqref{thmpullbackfrombd} enables us to do that. For example, if $A_X=A_Y$, $W_X=W_Y$ and the map $\mathcal T$ is designed to respect Plancherel measures, that is (for the most continuous spectrum), to pull back the relative character $I_\chi^{Y,\cl}$ for $Y$ to the relative character $I_\chi^{X,\cl}$ for $X$ (see Theorem \ref{Plancherel}), then 
\begin{quote}
\emph{the pullback of Mellin transform $M_\chi^{Y_\emptyset}$ on $Y_\emptyset$ via $\mathcal T_\emptyset$ should be}
\begin{equation}\mathcal T_\emptyset^* M_\chi^{Y_\emptyset} = \frac{\mu_X(\chi)}{\mu_Y(\chi)}\cdot M_\chi^{X_\emptyset}.
\end{equation}
\end{quote}

By the explicit form of the scalars $\mu_X(\chi)$, this corresonds to a composition of the \emph{multiplicative Fourier convolutions} that we defined in \S \ref{sssFourierconv} --- compare with \eqref{FE}. This is the conceptual explanation that I can presently give for all the transfer operators $\mathcal T$ and Hankel transforms that will appear in the rest of this paper: their geometric expressions are equal or, at least, \emph{deformations} of the geometric expressions for the transfer operators $\mathcal T_\emptyset$ of the boundary degenerations. These are given by Fourier convolutions determined by the scattering operators which, in turn, are closely related to the $L$-functions of the associated global period integrals, by \cite[\S 17]{SV}. Hence, the examples discussed in this paper suggest that \emph{the $L$-functions attached to spherical varieties inform the way that their relative trace formulas will be geometrically compared}.

\section{Transfer between the Kuznetsov formula and the stable trace formula for $\SL_2$} \label{sec:Rudnick}

\subsection{Relative characters}\label{ssrelchars-Kuz}

Here we discuss the local comparison behind Rudnick's thesis \cite{Rudnick}. Let $G=\SL_2$. Both $G$ as a $G\times G$-variety and the Whittaker space $(N,\psi)\backslash G$ have the same dual group, namely, $\PGL_2$. We will construct a local transfer map 
$$ \mathcal T: \mathcal S^-_{L(\Ad,1)}((N,\psi)\backslash G/(N,\psi)) \xrightarrow\sim \mathcal S(\frac{G}{G}),$$
which gives rise to stable functoriality between the Kuznetsov and the Selberg trace formula. The non-standard space $\mathcal S^-_{L(\Ad,1)}((N,\psi)\backslash G/(N,\psi)) $ of orbital integrals was defined in \S \ref{ssnonstandard}. In terms of the representatives 
$$ \tilde\zeta =\left(\begin{array}{cc}
 & -\zeta^{-1} \\
\zeta &
\end{array}\right)$$
of regular orbits for the Kuznetsov formula, it consists of measures which, in a neighborhood of $\zeta=0$, coincide with the usual test measures $\mathcal S((N,\psi)\backslash G/(N,\psi))$ for the Kuznetsov formula, while in a neighborhood of $\zeta =\infty$ they are of the form 
$$ C(\zeta^{-1}) d^\times \zeta,$$ 
where $C$ is a smooth function in a neighborhood of zero.

We can think of $N\backslash G$ as $V^*$, the complement of zero in a two-dimensional vector space, and the identification $N\simeq \Ga$ as a Whittaker structure. Then $N\backslash G\sslash N\simeq (V^*\times V^*)\sslash G$ was canonically identified in \S \ref{ssinvariant} with $\Ga$ through the symplectic pairing. This identification is compatible with the section $\zeta\mapsto \tilde\zeta$ over $F^\times$.

Let $\pi$ denote an irreducible tempered representation of $\SL_2$, and $\Pi$ its $L$-packet (the restriction of an irreducible tempered representation of $\GL_2$). We assume that $\pi$ is the unique generic element of $\Pi$ with respect to the character $\psi$ of $N$. 
Define a morphism
$$ \tilde\pi\hat\otimes \pi\to C^\infty((N,\psi^{-1})\backslash G \times (N,\psi)\backslash G)$$
so that evaluation at the coset represented by $(1,1)$ is given by
\begin{equation}\label{IIWhittaker} \tilde v \otimes v\mapsto \int^*_{N} \left<\pi(n) v, \tilde v\right> \psi(n) dn,\end{equation}
with the measure on $N$ induced by its identification with $\Ga$.
This regularized integral is understood as the value at $\lambda =1$ of the Fourier transform (defined as $\int \Phi(x)\psi(\lambda x) dx$) of the $L^2$-function $n\mapsto \left<\pi(n) v, \tilde v\right>$ (where $N$ is identified again with $\Ga$). I point the reader to \cite[\S 2]{LM} and \cite[\S 6.3]{SV} for details.  

The space $C_{\rm mod}^\infty((N,\psi^{-1})\backslash G \times (N,\psi)\backslash G)$ of smooth Whittaker functions of moderate growth is in canonical duality with $\mathcal S(((N,\psi)\backslash G \times (N,\psi^{-1})\backslash G)$. The restriction ``moderate growth'' only applies to the Archimedean case, and the image of the morphism defined by \eqref{IIWhittaker} automatically lands in it; however, from now on, for notational simplicity, we will be abusing notation and writing $C^\infty$ for $C^\infty_{\rm mod}$.  

The adjoint of the map above is a morphism 
$$\mathcal S((N,\psi)\backslash G \times (N,\psi^{-1})\backslash G) \to \pi\hat\otimes \tilde\pi.$$
Its composition with the canonical pairing $\pi\hat\otimes\tilde\pi \to \CC$ is a $G^\diag$-invariant functional on $\mathcal S((N,\psi)\backslash G \times (N,\psi^{-1})\backslash G)$, which factors through the coinvariant space 
$\mathcal S(N,\psi\backslash G/ N,\psi)$ (see \eqref{Whittakercoinvariants}), will be denoted by $J_\pi$ or $J_\Pi$:
$$J_\Pi: \mathcal S(N,\psi\backslash G/ N,\psi) \to \CC.$$
This is the \emph{relative character} (or Bessel distribution) attached to the packet $\Pi$. 

Explicitly,
$$ J_\Pi(\Phi_1\otimes\Phi_2) = \sum_{(v,\tilde v)} \int_{(N\backslash G)^2} \Phi_1(x_1)\Phi_2(x_2) \int^*_{N} \left<\pi(nx_1) v, \tilde\pi(x_2)\tilde v\right> \psi(n) dn,$$
where $(v,\tilde v)$ runs over dual pairs in dual bases of $\pi$ and $\tilde\pi$. Notice that it does not make a difference whether we sum over a dual basis for $\pi$ or for the entire $L$-packet $\Pi$; since the other elements of the packet are not generic, their contribution will be automatically zero.

\begin{example}
Let $F$ be non-Archimedean, with ring of integers $\mathfrak o$, residual degree $q$, and a uniformizer $\varpi$.

 Suppose that $\pi = I(\chi)$ is a $K=G(\mathfrak o)$-unramified principal series representation, unitarily induced from an unramified character $\chi$ of the upper triangular Borel subgroup (identified with a character of $F^\times$).  Let $\phi_{K,\chi}\in \pi$, $\phi_{K,\chi^{-1}}\in\tilde \pi$ be $K$-invariant vectors satisfying $\left<\phi_{K,\chi},\phi_{K,\chi^{-1}}\right>=1$. 
 It takes an elementary calculation (or an application of Macdonald's formula on zonal spherical functions) to show that the unramified matrix coefficient
 $$\Phi(y) = \left< \pi \begin{pmatrix} 1 & y \\ & 1 \end{pmatrix} \phi_{K,\chi}, \phi_{K,\chi^{-1}}\right>$$
 depends only on the absolute value of $y$, and satisfies $\Phi(y)=1$ on $\mathfrak o$ and $\Phi(y) = \frac{q^{-1}-q^{-2}+q^{-1}\chi(\varpi) + q^{-1}\chi(\varpi)^{-1}}{1+q^{-1}}$ when $y\in \varphi^{-1}\mathfrak o^\times$.
 
 Thus, 
 if $f\in \mathcal S(N,\psi\backslash G/N,\psi)$ is the image of the identity element of the Hecke algebra, then
 $$ \left< f, J_\pi\right>  = \int \Phi(y) \psi(y) dy = $$
 \begin{equation}\label{Bessel-basic}1-\Phi(\varpi^{-1}) = \frac{(1-q^{-1}\chi(\varpi))(1-q^{-1}\chi(\varpi)^{-1})}{1+q^{-1}} = \frac{\zeta(2)}{L(\pi,\Ad,1)}.
 \end{equation}

\end{example}

On the other hand, consider the space of test measures $\mathcal S(\frac{G}{G})$ for the stable trace formula of $G$. The \emph{stable character} $\Theta_\Pi$ of the $L$-packet $\Pi$ is a functional on this space, descending from the sum of the characters of the elements of $\Pi$, which are generalized functions on the group (that is, functionals on $\mathcal S(G)$).

It was proven in \cite[Theorem 6.3.4]{SV} that the characters $J_\Pi$ and $\Theta_\Pi$ satisfy the Plancherel formula for their respective spaces, for the \emph{same} Plancherel measure. More precisely, let $\widehat{G}_\st^\temp$ denote the ``stable tempered dual'' of $G$, i.e., the set of tempered representations modulo the equivalence of belonging to the same $L$-packet (or, equivalently, to the restriction of the same tempered representation of $\GL_2$). Then, fixing a measures $dg$ on $G$ to define $L^2(G)$ as a space of measures, there is a unique measure $\mu_G(\Pi)$ on $\widehat{G}_\st^\temp$ such that, for any $\varphi_1, \varphi_2\in \mathcal S(G)$ we have:
\begin{equation}\label{Plancherel-group}
 \int_G \frac{\varphi_1\cdot \varphi_2}{dg} = \int_{\widehat{G}_\st^\temp} \Theta_\Pi(\varphi_1\otimes \varphi_2) \mu_G(\Pi).
\end{equation}

Let $dx$ be the measure on $N\backslash G$ which factorizes the measure $dg$ with respect to the fixed Haar measure on $N\simeq\Ga$. Then, Theorem 6.3.4 of \cite{SV} states that 
for any $\varphi_1, \varphi_2\in \mathcal S(N,\psi\backslash G)$ we have:
\begin{equation}\label{Plancherel-Whittaker}
 \int_G \frac{\varphi_1\cdot \varphi_2}{dg} = \int_{\widehat{G}_\st^\temp} J_\Pi(\varphi_1\otimes \varphi_2) \mu_G(\Pi)
\end{equation}
for \emph{the same} measure $\mu_G$.

\subsection{The main theorem}

Fix the isomorphism $\Dfrac{G}{G} \simeq \Ga$ via the trace map --- hence, both $\mathcal S(\frac{G}{G})$ and the non-standard test measures $\mathcal S^-_{L(\Ad,1)}((N,\psi)\backslash G/(N,\psi))$ for the Kuznetsov formula are understood as measures on $\Ga$.  In this section we will prove part \eqref{two} of the following theorem, assuming the other statements, which will be proven in Section \ref{sec:sym2}.

\begin{theorem}\label{thmRudnick}
Consider the equivariant Fourier transform $\mathcal T:=\mathscr F_{\Id,1}$ of multiplicative convolution with the measure $D_1 = \psi(\zeta) d\zeta = \psi(\zeta) |\zeta|d^\times\zeta$ on $\Gm$.

Then:
\begin{enumerate}
 \item \label{one} The convolution makes sense on $\mathcal S(N,\psi\backslash G/N,\psi)$ as the Fourier transform of a distribution, and maps it into $\mathcal S(\frac{G}{G})$. 
 \item \label{two} For every tempered packet $\Pi$ and $J_\Pi$ as above, \begin{equation} \label{pullbackchar}
        \mathcal T^*\Theta_\Pi = J_\Pi.
       \end{equation} 
 \item \label{three} The transform extends to an isomorphism, given by the same convolution understood, again, as the Fourier transform of a(n $L^2$-)distribution:
 \begin{equation}\label{transferRudnickeq}
  \mathcal T: \mathcal S_{L(\Ad,1)}^-(N,\psi\backslash G/N,\psi) \xrightarrow\sim \mathcal S(\frac{G}{G}).
 \end{equation}
 \item \label{four} At non-Archimedean places, unramified over the base field $\mathbb Q_p$ or $\mathbb F_p((t))$, it satisfies the fundamental lemma for the Hecke algebra up to a factor of $\zeta(2)=(1-q^{-2})^{-1}$, namely: for all $h\in \mathcal H(G,K)\subset S(G)$, it takes the element
 $$ h\cdot f_{L(\Ad, 1)} \in \mathcal S_{L(\Ad,1)}^-(N,\psi\backslash G/N,\psi)$$ to the image of $\zeta(2) h$ in $\mathcal S(\frac{G}{G})$.
\end{enumerate}
\end{theorem}

\begin{remarks}
\begin{enumerate}
 \item  The part of Statement \eqref{four} about non-trivial elements of the Hecke algebra is quite redundant, since it can be deduced for the fundamental lemma for the identity element of the Hecke algebra, and Statement \eqref{two}. In any case, in Section \ref{sec:sym2} we will obtain it ``for free''. In generalizations of this paper, one would expect to prove an analog of Statement \eqref{four} locally, and then use it to deduce, by a global-to-local argument involving comparisons of relative trace formulas, an analog of Statement \eqref{two}.
 \item The factor $\zeta(2)$ in Statement \eqref{four} is compatible with the calculation of the relative character applied to the standard basic function of $\mathcal S(N,\psi\backslash G/N,\psi)$ in \eqref{Bessel-basic}. 
 \item Theorem 1.3 in \cite{SoundYoung} can be seen as a global application of the isomorphism \eqref{transferRudnickeq}, restricted to hyperbolic conjugacy classes.
 \item As Valentin Blomer pointed out to me, it seems that \eqref{pullbackchar} can be directly obtained, in the real case, from classical identities such as \cite[(6.611.1)]{Gradshteyn-Ryzhik}, expressing the Fourier transform of Bessel functions in terms of exponentials; I have not checked the details.
 \item In a different, but related, setting, David Ben-Zvi and Sam Gunningham have recently established what can be seen as a comparison between the quotient spaces associated to the Kuznetsov formula and the trace formula, in the setting of loop groups of arbitrary complex reductive groups. Their theorem \cite[Theorem 6.16]{BZvi-Gunningham} constructs, by spectral arguments, what can be informally be described as a map from ``$D$-modules on the quotient space $(N,\psi)\backslash G/(N,\psi)$'' to ``$D$-modules on the adjoint quotient of $G$''. It would be interesting to see an explicit geometric description of their comparison, similar to the Fourier transform of the above theorem.
\end{enumerate}

\end{remarks}

We will now prove Statement \eqref{two}, assuming the rest (more precisely assuming Statement \eqref{one}). Statements \eqref{one} and \eqref{three} will be proven at the end of \S \ref{Pfpart3}, and Statement \eqref{four} will be proven at the end of \S \ref{Pfpart4}.

\begin{proof}[Proof of Statement \eqref{two}, assuming the rest]
Let $f(\zeta) = \Phi(\zeta) d^\times\zeta$. Since the map $N\backslash G\twoheadrightarrow \C := N\backslash G\sslash N$ is smooth,  the untwisted push-forward map 
$$\mathcal S(N\backslash G)\to \Meas(\C)$$ 
has image in locally bounded measures (in fact, Schwartz measures); a fortiori, the same is true for the twisted push-forward. Hence, $\Phi(\zeta)\ll 1$. 

We write, formally,
$$ D_1\star f (\zeta) = d^\times \zeta\cdot  D_1 \star (|\bullet|\Phi(\bullet))(\zeta) = $$
$$= d^\times\zeta\cdot \int_{F^\times} |z^{-1}\zeta|\Phi(z^{-1} \zeta) \psi(z) |z| d^\times z = d \zeta\cdot  \int_F \Phi(z^{-1}) \psi(z\zeta) dz,$$
and hence interpret $\mathcal Tf$ as the product of the measure $d\zeta$ with the Fourier transform of the $L^2$-function $\zeta\mapsto \Phi(\zeta^{-1})$. 

The extension of this to an isomorphism 
  $$\mathcal T: \mathcal S_{L(\Ad,1)}^-(N,\psi\backslash G/N,\psi) \xrightarrow\sim \mathcal S(\frac{G}{G}),$$
  together with the fundamental lemma, will be postponed until \S \ref{sec:sym2}, where they will be obtained as special cases of a more general theorem.
  
We prove the statement on relative characters. It relies on the following formal relation:
\begin{equation}\label{transfer-trace}\mathcal T(f) = r_!(f \times dn),\end{equation}
where $r$ is the quotient map
$$r: \frac{G}{N} \to \frac{G}{G},$$ 
and $f\times dn$ denotes the ``pullback'' of $f$ to the $N$-torsor $\frac{G}{N} \to N\backslash G/N$ defined by the fixed Haar measure on $N$. \emph{Here we do not think of $f$ as a scalar-valued measure on $N\backslash G\sslash N$} (by the trivialization described in \S \ref{sstwistedpf}), \emph{but as an $(N,\psi)$-equivariant measure on $\Dfrac{G}{N}$, divided by the Haar measure of $N$}, where $n\in N$ acts by sending the class of $g\in G$ to the class of $ng$.  If we fix coordinates $\begin{pmatrix} a & b \\ c & d \end{pmatrix}\mapsto (\zeta = c, t=\tr)$ for $\Dfrac{G}{N}\simeq \mathbbm A^2$, the trivialization of $f$ introduced in \S \ref{sstwistedpf} is the one obtained by restricting it to the line $t=0$. If $f_1$ is the push-forward to $\Dfrac{G}{N}$ of a Schwartz measure on $G$, and $f$ its twisted push-forward to $\mathcal S(N,\psi\backslash G/N,\psi)$. 
and if $R_n$ denotes the translation action of $N$ on measures on $\Dfrac{G}{N}$, we have
\begin{equation}\label{ff1} f \times dn= \int R_n f_1 \cdot \psi^{-1}(n) dn.
\end{equation}
In the coordinates $(\zeta,t)$ above, the group $N=\Ga$ acts as $(\zeta,t)\cdot x = (\zeta, t+\zeta x)$.

We will need to give a rigorous meaning to the formal relation \eqref{transfer-trace}, because the push-forward $r_!$ does not converge absolutely. But, before we do that, let us justify this relation in a formal manner, pretending that all integrals were convergent.

First of all, let us show why \eqref{transfer-trace} should imply \eqref{pullbackchar}: With $f_1$ and $f$ as above we have, formally,
$$ \left< \mathcal T f ,\Theta_\Pi\right> \xlongequal{\eqref{transfer-trace}} \left<f\times dn, p^*\Theta_\Pi\right> = $$
$$=\left<\int_N R_x f_1 \cdot \psi^{-1} (x) dx \right> = \int_N \left< R_x f_1, \Theta_\Pi\right> \psi^{-1}(x) dx = $$
$$=\left<f_1, \int R_x \Theta_\Pi \psi(x) dx \right> = \left< f_1 , J_\Pi\right> = \left< f, J_\Pi\right>,$$
the last equality because the generalized function $J_\Pi$ is already $(N,\psi^{-1})$-equivariant on both sides.

To formally justify \eqref{transfer-trace}, we claim that, for a measure $f \in \mathcal S(N,\psi\backslash G/N,\psi)$ which, after trivialization, is written $f(\zeta) = \Phi(\zeta) d\zeta$, the pullback measure $f\times dn$ on $\frac{G}{N}$ is $\Phi(\zeta)|\zeta|^{-1} \psi(\frac{t}{\zeta}) d \zeta dt$. Indeed, this relies on the calculation that a matrix in $\SL_2$ with given $(\zeta = c,t)$ and $c\ne 0$ can be written as
$$ \begin{pmatrix} 1 & x \\ & 1 \end{pmatrix} \begin{pmatrix}  & -c^{-1} \\ c \end{pmatrix} \begin{pmatrix} 1 & y \\ & 1 \end{pmatrix}$$
with $x+y = \frac{t}{c}$. 

Thus, the push-forward $r_!(f \times dn)$ is 
\begin{equation}\label{rshriek} r_!(f \times dn)(t) = dt \cdot \left(\int_F \Phi(c)|c|^{-1} \psi(\frac{t}{c}) d c \right) = \int f(\frac{t}{z}) \psi(z) |z| d^\times z.
\end{equation}

Now let us rigorously prove \eqref{pullbackchar}: As before, consider a Schwartz measure on $G$, and let $f_1$ be its push-forward to $\Dfrac{G}{N}$, and $f$ its twisted push-forward to $\mathcal S(N,\psi\backslash G/N,\psi)$. The proof relies on the (rigorous) relation
\begin{equation}
 \int^* \left<r_! (R_x f_1),\Theta_\Pi\right> \psi^{-1}(x) dx = \left<\mathcal T f,\Theta_\Pi\right>,
\end{equation}
as measures on $\Dfrac{G}{G}$, where the left hand side should be interpreted, as we have done for the right hand side when defining $\mathcal Tf$, as the Fourier transform of a distribution; more precisely, as the value of the generalized function $u\mapsto \int_F \left<\pi_! (R_x f_1),\Theta_\Pi\right>  \psi(ux) dx $ at $1$, after we show that this generalized function is a continuous function in a neighborhood of $1$. 

Let us first show the corresponding identity for generalized functions, so choose a Schwartz measure $\varphi(u) du$ on $F$. We denote Fourier transforms by $\varphi\mapsto \hat\varphi$, whether they are defined with the character $\psi$ or $\psi^{-1}$, and leave it to the curious reader to figure out where each character is being used. We also write $\Theta$ for $\Theta_\Pi$. We compute:

$$ \int \varphi(u)  \int^* \left<r_! (R_x f_1),\Theta\right> \psi^{-1}(ux) dx du \xlongequal{\mbox{by def.}}  \int \hat\varphi(x)   \left<\pi_! (R_x f_1),\Theta\right> dx  = $$
$$ = \iiint \hat\varphi(x) f_1(\zeta,t+\zeta x) \Theta(t) dx,$$
where we have used the fact that the triple integral is absolutely convergent. Indeed,  
writing $f_1 = \Phi_1 d\zeta dt$, the product 
$$\hat\varphi(x) \Phi_1(\zeta,t+\zeta x) \Theta(t)$$ has local $L^1$-seminorms (i.e., $L^1$-seminorms over additive translates of any fixed compact subset in the variables $(x,\zeta,t)$) of rapid decay in the variable $\min(|x|, |\zeta|, |t|)$, because the function $(\zeta,t)\mapsto \Theta(t)$ has local $L^1$-seminorms of polynomial growth, $\hat\varphi$ is of rapid decay, and $\Phi_1(\zeta,t)$ has rapidly decaying local-$L^1$ seminorms in both variables.

We can now write it as an iterated integral, with the variable $x$ to the interior. Consider the function $x\mapsto \Phi_1(\zeta, t+\zeta x)$. For fixed $\zeta\ne 0$ and $t$, the function is rapidly decaying in $x$. Its Fourier transform against the character $\psi^{-1}$, evaluated at a point $u$, is ${^u\Phi}(\zeta) \psi(u\frac{t}{\zeta})$, where ${^u\Phi}$ is the Fourier transform of $x\mapsto \Phi_1(\zeta, \zeta x)$ (so that, by \eqref{ff1}, $f = {^1\Phi} d \zeta$).  Thus, the last integral can be written: 
$$ \int_t \Theta(t) \int_\zeta \int_u \varphi(u) {^u\Phi}(\zeta) \psi(u\frac{t}{c}) du d\zeta dt.$$

We now want to switch the order of integration over $\zeta$ and $u$, interpreting the integral over $\zeta$ as a Fourier transform in the sense of distributions: 
\begin{equation}\label{toprove}\int_\zeta \int_u \varphi(u) {^u\Phi}(\zeta) \psi(u\frac{t}{\zeta}) du d \zeta = \int_u \varphi(u) \int_\zeta^* {^u\Phi}(\zeta) \psi(u\frac{t}{\zeta}) du d \zeta.\end{equation}
To show this, let $\mapsto F_u$ be the Fourier transform of the function $\zeta\mapsto {^u\Phi}(\frac{1}{\zeta})$ \emph{against the character $x\mapsto \psi(ux)$}. Notice that for $u=1$ the measure $F_1(\zeta) d\zeta$ is precisely the image of ${^1\Phi} d\zeta$ under the transfer operator $\mathcal T$, and hence belongs to $\mathcal S(\frac{G}{G})$. As will be clear from the proof of Statement \eqref{one} (in \S \ref{proofRudnick}), there is nothing special about $u=1$; more precisely, for $u$ in a neighborhood of $1$, the map $u\mapsto F_u d\zeta$ is a continuous section of $\mathcal S(\frac{G}{G})$. 

The inverse map $F_u \mapsto {^u\Phi} $ is given by Fourier transform using the character $\psi^{-1}$, instead of $\psi$, which converges absolutely. Indeed, it is straightforward to see that measures in $\mathcal S(\frac{G}{G})$ are bounded by Schwartz measures, and hence $F_u$ is bounded (and of rapid decay). As $u$ varies in a neighborhood of $1$, the Fourier transforms of the functions $F_u$ converge uniformly, since $u\mapsto F_ud\zeta \in \mathcal S(\frac{G}{G})$ is continuous.

Thus, if $\kappa$ is another Schwartz function, and we compute both sides of the desired equality \eqref{toprove} as distributions in the variable $t$, we have
$$  \int_t \hat \kappa(t) \int_u \varphi(u) \int_c^*  {^u\Phi}(c) \psi(u\frac{t}{c}) du d^\times c = \int_t \hat\kappa(t) \int_u \varphi(u) F_u(t) dt =$$
$$ = \int_u \varphi(u) \int_t \hat\kappa(t) F_u (t) dt du = \int_u \varphi(u) \int_z \kappa(z) \widehat{^u F}(z) dz du,$$ 
where $\widehat{F_u}$ is defined using the character $\psi^{-1}$, so the last expression is
$$ \int_u \varphi(u) \int_z \kappa(z) \int_t F_u(t) \psi^{-1}(tz) dt dz du$$

But this converges absolutely as a triple integral, so we can write it as
$$ \int_z \kappa(z) \int_u \varphi(u) \int_t  F_u(t) \psi^{-1}(tz) dt du dz = $$
$$ =\int_z \kappa(z) \int_u \varphi(u) \int_t  F_u(t) \psi^{-1}(u t\frac{z}{u}) dt du dz = $$
$$ = \int_z \kappa(z) \int_u \varphi(u) {^u\Phi}(\frac{u}{z}) du dz.$$
Applying Fourier transform in the variable $z$ again, this is equal to 
 
$$  \int_t \hat\kappa(t) \int_c \int_u \varphi(u) {^u\Phi}(\frac{u}{\zeta}) \psi(\zeta t)   du d\zeta dt$$
$$ = \int_t \hat\kappa(t) \int_\zeta \int_u \varphi(u) {^u\Phi}(\zeta) \psi(\frac{ut}{\zeta}) du d\zeta   dt,$$
 completing the proof of \eqref{toprove}.

Putting all together, we have shown that
$$ \int \varphi(u)  \int^* \left<r_! (R_x f_1),\Theta\right> \psi^{-1}(ux) dx du  = \int_t \Theta(t) \int_u \varphi(u) F_u(t) du dt =$$ 
$$ = \int_u \varphi(u)  \int \Theta(t) F_u(t) dt du,$$
the last step because of the continuity of the section $u\mapsto F_u\in \mathcal S(\frac{G}{G})$, and the rapid decay of elements of $\mathcal S(\frac{G}{G})$. Now, the inner integral on the right hand side, as a generalized function of $u$, is represented by a continuous function, hence so is the inner integral on the left hand side, and the equality holds for $u=1$.

\end{proof}

\subsection{Comparison with the degeneration}

In this subsection, $F$ is a non-Archimedean field.

Let $X$ be a symbol for the Whittaker model $(N,\psi)\backslash G$, and let $X_\emptyset = N\backslash G$ be its asymptotic cone. Set $Y=G=\SL_2$, $G'=G\times G/\{\pm 1\}^\diag$, and $Y_\emptyset$ the asymptotic cone of $Y$. We recalled in \S \ref{ssrelchars} that there is a canonical ``asymptotics'' morphism
$$ e_\emptyset^*\otimes e_\emptyset^*: \mathcal S(X\times X) \to \mathcal S^+(X_\emptyset\times X_\emptyset),$$
and saw that it descends to spaces of push-forward measures.
Together with the transfer operator $\mathcal T$, they give rise to most of the diagram \eqref{Bcommute}, which we repeat here:

$$ \xymatrix{
\mathcal S(X\times X/G) \ar[rr]^{e_\emptyset^*\otimes e_\emptyset^*}\ar[d]^{\mathcal T} && \mathcal S^+(X_\emptyset\times X_\emptyset/G) \ar[d]^{\mathcal T_\emptyset} \\
\mathcal S(Y\times Y/G') \ar[rr]^{e_\emptyset^*\otimes e_\emptyset^*} && \mathcal S^+(Y_\emptyset\times Y_\emptyset/G') }.
$$

What is missing is the transfer operator $\mathcal T_\emptyset$ making the diagram commute. 

\begin{theorem}\label{groupdegen}
 Identify $\C_\emptyset:= Y_\emptyset\times Y_\emptyset\sslash G' = X_\emptyset\times X_\emptyset\sslash G = \Ga$ as in \S \ref{ssrelchars}, namely, sending the distinguished $G'$-orbit, resp.\ $G$-orbit, to $1$ (and the singular one to $0$). There is a unique $A_X$-equivariant operator $\mathcal T_\emptyset$ making the above diagram commute, given by the multiplicative Fourier convolution $\mathscr F_{\check\alpha, 1}$ --- again, understood as the Fourier transform of a distribution.
\end{theorem}

\begin{remark}
The reader will notice that the only property of the transfer operator $\mathcal T$ used in the proof is that it satisfies \eqref{pullbackchar}, in other words, that it relates the relative characters that correspond to \emph{the same} Plancherel measure, cf.\ \eqref{Plancherel-group} and \eqref{Plancherel-Whittaker}. At no point will we use the explicit expression for the transfer operator as $\mathscr F_{\Id,1}$. Thus, this theorem, with a simple comparison of coordinates that follows, gives a \emph{conceptual reason} why the transfer operator $\mathcal T$ is given by this formula: it is ``the same'' as the operator $\mathcal T_\emptyset$! In higher rank, in some examples computed together with Chen Wan, things are similar, but not so simple: the operator $\mathcal T$ tends to be a \emph{deformation} of the boundary operator $\mathcal T_\emptyset$; we currently do not understand the nature of this deformation.
\end{remark}

\begin{proof}
 Notice that for both $X$ and $Y$ the character $\delta$ on $A_X=A_Y$ coincides --- in the coordinate $\zeta$ as above, $\delta^\frac{1}{2} = |\zeta|$. 

 For what follows, denote by $J_\Pi$, $\Theta_\Pi$  by $J_\chi$, $\Theta_\chi$, respectively, when $\Pi$ is the principal series representation of $\SL_2$ obtained by unitary induction from the character $\chi$ of $A_X$. 
 Let $M_\chi$ denote the Mellin transform $f\mapsto \check f(\chi\delta^\frac{1}{2}) = \check f(\chi|\bullet|)$ on either of the spaces $\mathcal S^+(X_\emptyset\times X_\emptyset/G)$ and $\mathcal S^+(Y_\emptyset\times Y_\emptyset/G')$. 
 Let $I_\chi^X$, resp.\ $I_\chi^Y$ be its pullback via the asymptotics map to $\mathcal S(X\times X/G)$, resp.\ $\mathcal S(Y\times Y/G')$. 

 Comparing the Plancherel formulas \eqref{Plancherel-group}, \eqref{Plancherel-Whittaker} and Theorem \ref{thmpullbackfrombd},  
 we deduce that
 $$ I_\chi^Y \mu_Y(\chi) = \Theta_\chi \mu_G(\chi)$$
 and 
 $$ I_\chi^X \mu_X(\chi) = J_\chi \mu_G(\chi),$$
 where $\mu_G(\chi) d\chi$ is the Plancherel measure for $\SL_2$. (It can be shown to be equal to $\mu_Y(\chi) d\chi$ for a correct choice of Haar measures, hence $I_\chi^Y = \Theta_\chi$, but we don't need this here.)
 
 Since $\mathcal T^* \Theta_\chi = J_\chi$, for the diagram to commute we would need 
 $$ \mathcal T^*_\emptyset M_\chi  =\frac{\mu_X(\chi)}{\mu_Y(\chi)} M_\chi =  \gamma(\chi,-\check\alpha, 0,\psi^{-1})^{-1} M_\chi,$$
 or, in other words, 
 $$ \widecheck{\mathcal T_\emptyset f}(\chi) = \gamma(\chi,-\check\alpha, 1,\psi^{-1})^{-1} \check f(\chi) = \gamma(\chi, \check\alpha, 0,\psi) \check f(\chi).$$

 Now consider the multiplicative Fourier convolution $\mathscr F_{\check\alpha, 1}$ of an element $f\in \mathcal S^+(X_\emptyset\times X_\emptyset/G)$. As in the beginning of the proof of Theorem \ref{thmRudnick}, if $f(\zeta) = \Phi(\zeta) d^\times\zeta$ then $\mathscr F_{\check\alpha, 1}$ will be understood as the Fourier transform of the distribution $D(\zeta):= \Phi(\zeta^{-1}) d\zeta = |\zeta| f(\zeta^{-1})$. Notice that $\Phi$ is smooth away from $0$, and of compact support on $F$, therefore this distribution is smooth; its Fourier transform is a compactly supported distribution on $F$. We will now argue that $\mathscr F_{\check\alpha, 1} f$ can be reconstructed from its Mellin transform, precisely as in \eqref{invMellin}; moreover, that \eqref{FE} holds for its Mellin transform:
 $$ \widecheck{\mathscr F_{\alpha, 1}f}(\chi) = \gamma(\chi, \check\alpha, 0,\psi) \check f(\chi).$$
 This will imply, by \eqref{invMellin}, that $\mathcal T_\emptyset = \mathscr F_{\check\alpha,1}$.
 
 To see this, we apply Corollary \ref{corWhittakerasymptotics}: Suppose that $f$ is the push-forward of an element $\varphi \in \mathcal S^+(X_\emptyset\times X_\emptyset)$. Acting on $\varphi$, in both variables,  by the element $h\in \widehat{\mathcal S(A_X)}$ (completed Hecke algebra) whose Mellin transform is $\check h(\chi) = L(\chi,\check\alpha, 1)^{-1}$, we obtain an element $\varphi'\in \mathcal S(X_\emptyset\times X_\emptyset)$, whose push-forward we denote by $f'$. Taking into account the normalization of the action of $A_X$ on measures on $X_\emptyset$, which we did not adopt on the quotient space $\C_\emptyset$, 
 $$ f' = (|\bullet|^{-1} h) \cdot (|\bullet|^{-1} h) \cdot f,$$
 where now $A_X$ is identified with $\Gm$ through the positive root character.
 The morphism $X_\emptyset\times X_\emptyset \to \C_\emptyset$ is smooth, hence $f'$ is a Schwartz measure on the line. Acting on it once more by the measure $h$:
 $$ f'' := h\cdot (|\bullet|^{-1} h) \cdot (|\bullet|^{-1} h) \cdot f,$$
 it becomes supported away from zero. 
 
 Hence, the distribution 
 $$ D'' := (|\bullet| h^\vee)\cdot h^\vee \cdot h^\vee \cdot D,$$
 where $h^\vee (a) = h(a^{-1})$, is (smooth and) of compact support on $F^\times$. In particular, the theory of Tate zeta integrals applies to it:  $\mathscr F_{\check\alpha, 1} f''=$ the Fourier transform of $D''$ is a smooth, compactly supported measure on the affine line, both $D''$ and $\mathscr F_{\check\alpha, 1} f''$ can be reconstructed from their Mellin transforms, as in \eqref{invMellin}, and their Mellin transforms satisfy the functional equation \eqref{gammaZeta}, which can be written as in \eqref{FE}:
 $$ \widecheck{\mathscr F_{\check\alpha, 1}f''}(\chi) = \gamma(\chi, \check\alpha, 0,\psi) \check f''(\chi).$$
 
 But, by construction, 
 $$\check f''(\chi) =   L(\chi,-\check\alpha, 1)^{-2}  L(\chi,-\check\alpha, 2)^{-1} \cdot \check f(\chi),$$
 and, by the equivariance of Fourier convolution,
 \begin{equation}\label{equiv} \mathscr F_{\check\alpha, 1}f'' = h\cdot (|\bullet|^{-1} h) \cdot (|\bullet|^{-1} h) \mathscr F_{\check\alpha, 1}f.
 \end{equation}
It is now easy to see that, since  the factor $L(\chi,-\check\alpha, 1)^2  L(\chi,-\check\alpha, 2)$ has no pole at $\chi=1$, the inverse Mellin transform \eqref{invMellin}, applied to 
 $$\gamma(\chi, \check\alpha, 0,\psi) L(\chi,-\check\alpha, 1)^2  L(\chi,-\check\alpha, 2) \check f''(\chi) = \gamma(\chi, \check\alpha, 0,\psi) \check f(\chi),$$ 
 represents the unique compactly supported distribution $\mathscr F_{\check\alpha, 1}f$ on $F$ (in fact, in this case, a measure represented by an $L^1$-function) which satisfies \eqref{equiv}.

 This implies the claim. 

 Finally, we argue that $\mathscr F_{\check\alpha, 1} f \in \mathcal S^+(Y_\emptyset\times Y_\emptyset)$. By the Mellin inversion formula \eqref{invMellin}, again, it suffices to show that its Mellin transform is contained in the space of Mellin transforms of elements of $\mathcal S^+(Y_\emptyset\times Y_\emptyset)$. If $f$ is obtained as the asymptotics of some $\tilde f \in \mathcal S(X\times X/G)$, we have already seen that 
 $$M_\chi \mathscr F_{\check\alpha, 1} f = \gamma(\chi,-\check\alpha, 0,\psi^{-1})^{-1} M_\chi f = \gamma(\chi,-\check\alpha, 0,\psi^{-1})^{-1} I^X_\chi \tilde f = J_\chi \tilde f,$$
 but also
 $$ J_\chi\tilde f = \Theta_\chi (\mathcal T \tilde f) = M_\chi (e_\emptyset\otimes e_\emptyset(\mathcal T\tilde f)).$$
 
 Therefore, 
 $\mathscr F_{\check\alpha, 1} f = e_\emptyset^*\otimes e_\emptyset^*(\mathcal T\tilde f)  \in \mathcal S^+(Y_\emptyset\times Y_\emptyset)$.

\end{proof}

Finally, let us notice that the coordinates that we have been using on the spaces $Y\times Y\sslash G'$ and $Y_\emptyset\times Y_\emptyset\sslash G'$ are compatible, in the following sense: There is a family $\mathcal Y \to \Ga$, namely, 
$$\mathcal Y=\Mat_2\xrightarrow{\det}\Ga,$$ the space of $2\times 2$ matrices, with general fiber $G'$-equivariantly isomorphic to $Y$, and special fiber (over $0\in \Ga$) equal to $Y_\emptyset^a$. The identification of $Y_\emptyset$ with the open $G'$-orbit in the fiber over $0$ was described in \eqref{ssscatteringgroup}. 
In coordinates, pick the Borel subgroup $B^-\times B\subset G'$, where $B^-=$ lower triangular matrices and $B=$ upper triangular; then $N^-\backslash \SL_2\sslash N$ is identified with $Y_\emptyset\sslash(N^-\times N)$ via the top left entry of a matrix --- this is the isomorphism \eqref{samehoro}.

Now we will see that there is an isomorphism
$$(\mathcal C, \det): \mathcal Y\times_{\Ga}\mathcal Y\sslash G \simeq \Ga\times \Ga,$$
such that the restriction of $\mathcal C$ to $Y\times Y\sslash G'=$ the fiber of $1$ is the trace map:
$$(g_1,g_2)\mapsto \tr(g_1 g_2^{-1}),$$
while its restriction to the fiber of $0$ is our preferred coordinate for $Y_\emptyset\times Y_\emptyset\sslash G'$. 

Indeed, let $w = \begin{pmatrix} & -1 \\ 1\end{pmatrix}$ and take $\mathcal C(g_1,g_2) = \tr(g_1 w g_2^t w^{-1})$. For $g_2\in \SL_2$ we have $g_2^{-1} = wg_2^t w^{-1}$. To check that this coincides with our distinguished coordinate for $Y_\emptyset\times Y_\emptyset\sslash G'=Y^a_\emptyset\times Y^a_\emptyset\sslash G'$, it suffices to observe that the zero matrix is mapped to $0$, and the pair $(g_1,g_2) = (\begin{pmatrix} 1& 0 \\ 0& 0\end{pmatrix}, \begin{pmatrix}0 & 0 \\0 & 1\end{pmatrix})$, which belongs to the distinguished $G'$-orbit on $Y_\emptyset\times Y_\emptyset$, is mapped to $1$.

For $X$ and $X_\emptyset$ the analogous statement is a tautology, since the underlying spaces are the same (and we have been using the same coordinate both for $X$ and its degeneration).

\section{Transfer between the Kuznetsov formula and the relative trace formula for torus periods} \label{sec:Waldspurger}

Now consider the case of $Y=T\backslash G$, where $G=\PGL_2$ and $T\simeq \Gm$ is a split torus. One could also consider a non-split torus, but would need to slightly modify the equivariant Fourier transforms that we defined in \S \ref{sssFourierconv}. We review the local transfer operator $$ \mathcal T: \mathcal S^-_{L(\Std,1)^2}((N,\psi)\backslash G/(N,\psi)) \xrightarrow\sim \mathcal S(T\backslash G/T),$$
constructed in \cite{SaBE1, SaBE2}. As in \S \ref{scattorus}, we will identify $X$ with the space of quadratic forms of discriminant $-\frac{1}{4}$ on a two-dimensional symplectic space $V$.

Recall that the non-standard space $\mathcal S^-_{L(\Std,\frac{1}{2})^2}((N,\psi)\backslash G/(N,\psi)) $ of orbital integrals was defined in \S \ref{ssnonstandard}. In terms of the representatives 
$$ \tilde\xi =\left(\begin{array}{cc}
 & -1 \\
\xi &
\end{array}\right)$$
of regular orbits for the Kuznetsov formula, it consists of measures which, in a neighborhood of $\xi=0$, coincide with the usual test measures $\mathcal S((N,\psi)\backslash G/(N,\psi))$ for the Kuznetsov formula, while in a neighborhood of $\xi =\infty$ they are of the form 
$$ (C_1(\xi^{-1})+C_2(\xi^{-1}) \log|\xi|) d^\times \xi,$$ 
where $C_1$, $C_2$ are smooth functions in a neighborhood of zero.

The role of the character $\Theta_\Pi$, here, will be played by a relative character $I_\pi$ for an irreducible tempered representation of $\PGL_2$, for the quotient space $\Gm\backslash \PGL_2/\Gm$. The definition of the relative character $I_\pi$ is completely analogous to that of the Kuznetsov relative character $J_\pi$: It is given as the composition
$$ \mathcal S(Y\times Y)\to \pi\hat\otimes\tilde\pi\to\CC,$$
where the dual of the map to $\pi\hat\otimes\tilde\pi$ is the morphism
$$ \tilde\pi\otimes\pi\to C^\infty(Y\times Y)$$
that, composed with evaluation at $T1 \times T1$ is given by:
$$\tilde v \otimes v\mapsto \int_{T} \left<\pi(t) v, \tilde v\right> dt.$$ 
Here the integral is convergent (for tempered representations), and no normalization is needed. Moreover, the $L$-packets for the group $\PGL_2$ are singletons (if we do not consider its inner forms, which we should have, in the case of a non-split torus), therefore there is no need to distinguish, notationally, between $\pi$ and its $L$-packet $\Pi$. 
The measure on $T$ is fixed to be the multiplicative Haar measure $d^\times x$ on $F^\times$, after identifying $T\simeq\Gm$ --- there are two inverse ways to perform this identification, and they give rise the same measure.

\begin{theorem}
Consider the equivariant Fourier transform $\mathcal T:=\mathscr F_{\Id,1} \circ \mathscr F_{\Id,1}$ of multiplicative convolution, twice, with $D_{1} = \psi(\bullet) |\bullet| d^\times\bullet$ on measures on $\Gm$.

Then:
\begin{enumerate}
 \item The convolution makes sense on $\mathcal S(N,\psi\backslash G/N,\psi)$ as the Fourier transform of a distribution, and maps it into $\mathcal S(T\backslash G/T)$. 
  \item For every tempered representation $\pi$, \begin{equation} 
        \mathcal T^*I_\pi = J_\pi.
       \end{equation} 
 \item The transform extends to an isomorphism, given by the same convolution understood, again, as the Fourier transform of a(n $L^2$-)distribution:
 \begin{equation}
  \mathcal T: \mathcal S_{L(\Std,\frac{1}{2})^2}^-(N,\psi\backslash G/N,\psi) \xrightarrow\sim \mathcal S(T\backslash G/T).
 \end{equation}
 \item At non-Archimedean places, it satisfies the fundamental lemma for the Hecke algebra, up to a factor of $\zeta(1)^2=(1-q^{-1})^{-2}$, namely: for all $h\in \mathcal H(G,K)$, it takes the element
 $$ h\cdot f_{L(\Std, \frac{1}{2})^2} \in \mathcal S_{L(\Std,\frac{1}{2})^2}^-(N,\psi\backslash G/N,\psi)$$ to the image of $\zeta(1)^2\cdot h$ in $\mathcal S(T\backslash G/T)$.
\end{enumerate}
\end{theorem}

For precise references to \cite{SaBE1, SaBE2}, and an explanation of how the formulas there relate to the above transfer operator (given the different coordinates that we are using), I point the reader to \cite{SaHanoi}. Here, I would like to discuss the relation to transfer operators on the asymptotic cone. 

First, let us fix compatible coordinates for $Y\times Y\sslash G$ and for $Y_\emptyset\times Y_\emptyset \sslash G$. Consider the family $\mathcal Y\to \Ga$ whose fiber $\mathcal Y_t$ over $t\in \Ga$ is the space of quadratic forms of discriminant $-\frac{t^2}{4}$ on $V$ (so that all non-zero fibers correspond to split non-degenerate forms). The fiber over $0$ contains the boundary degeneration $Y_\emptyset$, i.e., the space of rank-one quadratic forms, as was explained in \S \ref{scattorus}. As we did there, we identify the quotient $\mathcal Y_t\times V\sslash G\simeq \Ga$ by the evaluation map. 

We can fix an isomorphism 
$$(\mathcal C, t): \mathcal Y\times_{\Ga} \mathcal Y\sslash G \xrightarrow \sim \Ga\times \Ga,$$
where $t$ is the defining morphism to $\Ga$, and $\mathcal C$ is as follows: Consider a quadratic form $q$ on $V$; it defines a morphism $V\to V^\vee$, which combined with the fixed isomorphism $\iota_\omega^{-1}: V^\vee \to V$ induced by the symplectic form, gives rise to an endomorphism
$$ \iota_q: V\to V.$$
Explicitly, $\omega (u,\iota_q(v)) = q(u,v)$ for all $u, v\in V$. We now define $\mathcal C (q_1, q_2) = \tr(\iota_{q_1}\circ \iota_{q_2}) + \frac{t}{2}$.

The reader can check that, on the fiber over $t=1$ ($=Y\times Y\sslash G$), the coordinate $\mathcal C$ is the one that we fixed above, while on the fiber over $t=0$ ($=Y_\emptyset\times Y_\emptyset\sslash G$) the coordinate descends from the map 
$$ V\times V \ni (u,v)\mapsto -\omega(u,v)^2\in \Ga$$
under the isomorphism $Y_\emptyset\simeq V^*/\{\pm 1\}$. Therefore, this is \emph{opposite} to the ``canonical'' isomorphism $Y_\emptyset\times Y_\emptyset\sslash G\to \Ga$ that was discussed in \S \ref{scattorus} --- we will need to take this difference into account. Notice that the choice of $\mathcal C$ seems quite arbitrary, and indeed, it is only justified because this turns out to give, over $t=1$, the coordinate that works for comparison to the Kuznetsov formula, for the representatives of $N\backslash G/N$ cosets that we have chosen. However, we could not have preserved the coordinate at $t=1$ and multiplied the one at $t=0$ by $(-1)$, if we want to have a coordinate that extends over the whole family. 

We fix the coordinate $\xi(v,u)= \omega(v,u)$ for $V \times V\sslash G$, as in \S \ref{ssinvariant}. Then we have the following:

\begin{theorem} \label{torusdegen}
There is a unique $A_X=A_Y$-equivariant operator $\mathcal T_\emptyset$ making the following diagram commute: 
$$ \xymatrix{
\mathcal S(X\times X/G) \ar[rr]^{e_\emptyset^*\otimes e_\emptyset^*}\ar[d]^{\mathcal T} && \mathcal S^+(X_\emptyset\times X_\emptyset/G) \ar[d]^{\mathcal T_\emptyset} \\
\mathcal S(Y\times Y/G) \ar[rr]^{e_\emptyset^*\otimes e_\emptyset^*} && \mathcal S^+(Y_\emptyset\times Y_\emptyset/G) }.
$$

The operator $\mathcal T_\emptyset$ is given by the multiplicative Fourier convolutions $\mathscr F_{\frac{\check\alpha}{2}, 1}\circ \mathscr F_{\frac{\check\alpha}{2}, 1}$ --- again, understood as the Fourier transform of a distribution.
\end{theorem}

\begin{proof}
 The proof is the same as for Theorem \ref{groupdegen}. We only need to explain why, for the above choices of coordinates on $X_\emptyset\times X_\emptyset\sslash G$ and $Y_\emptyset\times Y_\emptyset\sslash G$, the operator $\mathcal T_\emptyset$ must act on Mellin transforms as follows:

 $$\widecheck{(\mathcal T_\emptyset f)} (\chi)  = \gamma(\chi,\frac{\check\alpha}{2}, 0,\psi)^2 \check f(\chi).$$
 
 If we were following the arguments of Theorem \ref{groupdegen} using the coordinates for $X_\emptyset\times X_\emptyset\sslash G$ and $Y_\emptyset\times Y_\emptyset\sslash G$ that we used in \S \ref{scattorus}, we would, instead, have the factor 
 $$\gamma(\chi,\frac{\check\alpha}{2}, 0,\psi) \gamma(\chi,\frac{\check\alpha}{2}, 0,\psi^{-1})$$
 instead (originating in Theorem \ref{thmpullbackfrombd}). Now, however, that we are using the \emph{negative} of this coordinate for $Y_\emptyset\times Y_\emptyset\sslash G$ (but not for $X_\emptyset\times X_\emptyset\sslash G$!), we have to mutiply this factor by $\chi(-1)$. This turns the factor $\gamma(\chi,\frac{\check\alpha}{2}, 0,\psi^{-1})$ to $\gamma(\chi,\frac{\check\alpha}{2}, 0,\psi)$.
 
\end{proof}

\bibliographystyle{alphaurl}
\bibliography{biblio}

\end{document}